\documentclass[11pt]{amsart}

\usepackage{amsmath,yhmath}
\usepackage{amssymb}
\usepackage{amsfonts}
\usepackage{mathrsfs} 
\usepackage{bbm} 
\usepackage{mathtools}
\usepackage{tikz}
\usepackage{tikz-cd}
\usepackage[inline]{enumitem}
\allowdisplaybreaks
\usepackage{stmaryrd} 

\numberwithin{equation}{section}

\usepackage{caption}
\usepackage{geometry}
\usepackage{parskip}
\usepackage{setspace}
\usepackage{tikz}
\tikzset{Equal/.style={-,double line with arrow={-,-}}} 
\usetikzlibrary{decorations.markings}
\tikzset{double line with arrow/.style args={#1,#2}{decorate,decoration={markings,
mark=at position 0 with {\coordinate (ta-base-1) at (0,1pt);
\coordinate (ta-base-2) at (0,-1pt);},
mark=at position 1 with {\draw[#1] (ta-base-1) -- (0,1pt);
\draw[#2] (ta-base-2) -- (0,-1pt);
}}}}

\usepackage{aliascnt}
\usepackage[colorlinks, linkcolor=black,citecolor=blue, urlcolor=cyan, unicode, bookmarksnumbered]{hyperref} 
\newcommand*{\doi}[1]{\href{https://dx.doi.org/#1}{doi: #1}}

\newtheoremstyle{note}
{\topsep}
{\topsep}
{}
{0pt}
{\bfseries}
{.}
{0.5em }
{}
\theoremstyle{plain}

\newtheorem{Thm}{Theorem}[section]

\newtheorem*{Thm*}{Theorem}

\newaliascnt{prop}{Thm}
\newtheorem{Prop}[prop]{Proposition}
\aliascntresetthe{prop}

\newaliascnt{lemma}{Thm}
\newtheorem{Lemma}[lemma]{Lemma}
\aliascntresetthe{lemma}

\newaliascnt{coro}{Thm}
\newtheorem{Coro}[coro]{Corollary}
\aliascntresetthe{coro}

\newaliascnt{conjecture}{Thm}

\aliascntresetthe{conjecture}

\newtheorem{Lemma*}{Lemma}
\newtheorem*{Prop*}{Proposition}
\newtheorem*{Coro*}{Corollary}

\theoremstyle{definition}

\newaliascnt{def}{Thm}
\newtheorem{Def}[def]{Definition}
\aliascntresetthe{def}

\newtheorem*{Def*}{Definition}

\newaliascnt{Eg}{Thm}
\newtheorem{eg}[Eg]{Example}
\aliascntresetthe{Eg}

\newtheorem*{eg*}{Example}

\theoremstyle{remark}

\newaliascnt{rmk}{Thm}
\newtheorem{RMK}[rmk]{Remark}
\aliascntresetthe{rmk}

\newtheorem*{RMK*}{Remark}

\theoremstyle{plain}

\newtheorem{iThm}{Theorem}[section]

\newtheorem*{iThm*}{Theorem}

\newaliascnt{iprop}{iThm}

\aliascntresetthe{iprop}

\newaliascnt{iCoro}{iThm}
\newtheorem{iCoro}[iCoro]{Corollary}
\aliascntresetthe{iCoro}

\newaliascnt{iconjecture}{iThm}

\aliascntresetthe{iconjecture}

\theoremstyle{definition}

\newaliascnt{idef}{iThm}

\aliascntresetthe{idef}

\newtheorem*{iDef*}{Definition}

\newaliascnt{iEg}{iThm}

\aliascntresetthe{iEg}

\newtheorem*{ieg*}{Example}

\theoremstyle{remark}

\newaliascnt{irmk}{iThm}

\aliascntresetthe{irmk}

\newtheorem*{iRMK*}{Remark}

\newcommand{\bQ} { {\mathbb{Q}}}
\newcommand{\bZ} { {\mathbb{Z}}}
\newcommand{\bR} { {\mathbb{R}}}
\newcommand{\bK} { {\mathbb{K}}}
\newcommand{\bN} { {\mathbb{N}}}

\newcommand{\cD} { {\mathcal{D}}}

\newcommand{\cK} { {\mathcal{K}}}

\newcommand{\RHOM}{{ \operatorname{RHom}  }}
\newcommand{\HOM}{{ \textnormal{Hom}  }}

\newcommand{\dotimes}{{\overset{L}\otimes }}
\newcommand{\dboxtimes}{{\overset{L}\boxtimes }}
\newcommand{\pt}{{\operatorname{pt}}}
\newcommand{\bq} { {\textbf{q}}}
\newcommand{\bp} { {\textbf{p}}}

\newcommand{\tnT}{{\operatorname{T}}}
\newcommand{\tnR}{{\operatorname{R}}}
\newcommand{\id}{{\operatorname{Id}}}
\newcommand{\Mod}{{\operatorname{Mod}}}
\newcommand{\fpw}[1]{{\llbracket #1\rrbracket}}

\makeatletter
\newcommand\boxstar{\mathrel{\boxcls@{\star}}}
\newcommand{\boxcls@}[1]{%
  \vphantom{\Box}
  \mathpalette\boxcls@@{#1}%
}
\newcommand{\boxcls@@}[2]{%
  \ooalign{$\m@th#1\Box$\cr
  \hidewidth\boxcls@fix{#1}\hbox{$\m@th#1#2\mkern 2.4mu$}\cr}}

\newcommand\boxcls@fix[1]{
  \ifx#1\displaystyle
    \raise.225ex
  \else
    \ifx#1\textstyle
      \raise.225ex
    \else
      \ifx#1\scriptstyle
        \raise.180ex
      \else
        \raise.150ex
      \fi
    \fi
  \fi
}
\makeatother

\title{Idempotence of microlocal kernels and $S^1$-equivariant Chiu-Tamarkin invariant}
\author{Bingyu Zhang}
\date{\today}

\begin{document}
\maketitle

\begin{abstract}
In this article, we present algebraic constructions about the Chiu-Tamarkin invariant motivated by the idempotence of microlocal kernels: (1) an $S^1$-equivariant Chiu-Tamarkin invariant; (2)  a graded commutative product on the non-equivariant Chiu-Tamarkin invariant; and we also give a natural explanation for the definition of the $\bZ/\ell$-equivariant Chiu-Tamarkin invariant using the idempotence. As applications, we: (1) construct a sequence of symplectic capacities $(\overline{c}_k)_{k\in \bN}$ using the $S^1$-equivariant Chiu-Tamarkin invariant, and prove that it coincides with the symplectic capacities $({c}_k)_{k\in \bN}$ we defined using the $\bZ/\ell$-equivariant Chiu-Tamarkin invariant under certain conditions; and (2) prove a Viterbo isomorphism that relates the Chiu-Tamarkin invariant of disk bundles with string topology of the base. In the appendix, we prove the existence of microlocal kernels for all open sets in a cotangent bundle under the setup of triangulated categories.
\end{abstract}

\section{Introduction}
The microlocal theory of sheaves, introduced and developed by Kashiwara and Schapira \cite{KS90}, has proven to be an important tool for studying symplectic and contact geometry, as demonstrated in pioneering works \cite{nadler2009constructible,tamarkin2013}.

To study quantitative symplectic geometry of $T^*X$ using sheaf theory, Tamarkin defines a category $\mathcal{D}(X)$ as a Verdier quotient of the derived category of sheaves $D(X\times \bR)$ by a microlocal condition in \cite{tamarkin2013}. We will call it the Tamarkin category. The Tamarkin category admits the following natural transformations between endofunctors: 
\[\tau_c: \id_{\cD(X)}\rightarrow \tnT_{c*},\quad c\geq 0,\] 
where $\tnT_{c*}$ is induced by the direct image functor of the translation map $\tnT_c(\bq,t)=(\bq,t+c)$ on $X\times \bR$. We then define the numerical invariant (\autoref{sheafenergy}) \[e(F)=\inf\{c \geq 0: \tau_c(F)=0\},\quad F \in \cD(X) .\] Tamarkin's work and subsequent works (for example, \cite{asano2017persistence,asanoike2022complete,GV2022viterboconj,guillermouviterbo_gammasupport}) show that $e(F)$ and its variants are effective tools for quantitative symplectic geometry. 

Chiu's proof of the contact non-squeezing theorem \cite{chiu2017} is a highlight of the theory of the Tamarkin category. He introduces ${\bZ/\ell}$-equivariant homological invariants $C_{T}^{\bZ/\ell}(U,\bK)$ for $T\geq 0$ and open sets $U$ in cotangent bundles. We will call it ${\bZ/\ell}$-equivariant Chiu-Tamarkin invariant. The main ingredients of his construction are: 
\begin{itemize}[fullwidth]
    \item For every open set $U \subset T^*X$, we can associate an object $P_U\in \cD(X^2)$ that we call the microlocal kernel of $U$. The microlocal kernel is {\it idempotent} in the sense $P_U\star P_U\cong P_U$, where $\star$ is the convolution in $\cD(X^2)$ (see \autoref{def: convolution}).
    \item A {\it ${\bZ/\ell}$-action} of $P_U^{\boxtimes \ell}$ that is given by the cyclic permutation of factors.
\end{itemize}
One feature of Chiu's proof for the contact non-squeezing theorem is that we do not care specifically about the value of $\ell$, only that it is sufficiently large. In \cite{Capacities2021}, motivated by the idea that the ${\bZ/\ell}$-actions approximate an $S^1$-action, we constructed a sequence of symplectic capacities $(c_k)_{k\in \bN}$ using the ${\bZ/\ell}$-equivariant Chiu-Tamarkin invariant that behaves like $S^1$-equivariant capacities constructed using the $S^1$-equivariant symplectic (co)homology (for example, the Gutt-Hutchings capacities \cite{gutthutchings2018capacities}). 

In this paper, we construct a genuine algebraic $S^1$-action to study the similarities between the $\bZ/\ell$-equivariant capacities and the $S^1$-equivariant capacities. The idempotence of $P_U$ plays a key role in the construction. We will also introduce a graded commutative product on the non-equivariant Chiu-Tamarkin invariant, as well as providing a natural explanation for the definition of the ${\bZ/\ell}$-equivariant Chiu-Tamarkin invariant. Subsequently, we will explain some geometric applications of our constructions concerning symplectic capacities and string topology.

\subsection{Algebraic structures of Chiu-Tamarkin invariants}The idempotence of the microlocal kernel, i.e. $P_U\star P_U\cong P_U$, implies that $P_U$ is a {\it monoid object} in the monoidal category $(\cD(X^2),\star)$. Motivated by this observation, we construct two algebraic structures based on the idempotence of the microlocal kernels: (1)An $S^1$-equivariant Chiu-Tamarkin invariant; (2)A graded commutative product on the non-equivariant Chiu-Tamarkin invariant.

\begin{enumerate}[fullwidth]
\item In \autoref{section: s1-theory}, we will define an object $C_T^{S^1}(U,\bK) $ in the derived category of mixed complexes $ D_{dg}(\bK[\epsilon]-\Mod)$, which we refer to as the {\em $S^1$-equivariant Chiu-Tamarkin invariant}, for a field $\bK$, $T\geq 0$, and an open set $U$ in cotangent bundles. We prove the following:
\begin{iThm}\label{intro-thm: prop of S^1-thy}The $S^1$-equivariant Chiu-Tamarkin invariant $C_T^{S^1}(U,\bK)$ satisfies the following properties:
\begin{enumerate}[fullwidth,label=$(${\alph*}$)$]
\item (\autoref{prop: s1 persistence}) The map $T\rightarrow  H^*C^{S^1}_T(U,\bK)$ is a persistence $\bK[u]$-module. 

\item Under the forgetful functor $For:D_{dg}(\bK[\epsilon]-\Mod)\rightarrow D_{dg}(\bK-\Mod)$, we have (\eqref{equation: forgetful S^1 to nonequivariant})
\[For(C_{T}^{S^1}(U,\bK))\cong C_{T}(U,\bK)\quad\in\quad D_{dg}(\bK-\Mod),\]
where $C_{T}(U,\bK)$ denotes $C_{T}^{\bZ/\ell}(U,\bK)$ with $\ell=1$.
\item We have a Gysin sequence (\eqref{Gysin long exact sequence})
    \[H^{p-2}C_{T}^{S^1}(U,\bK)\xrightarrow{\cdot u}H^{p}C_{T}^{S^1}(U,\bK)\rightarrow H^pC_{T}(U,\bK)\xrightarrow{+1}.\]

And we have \autoref{invariance1}:
\item Let $U_1\xhookrightarrow{i} U_2$ be an inclusion, there is a morphism $C^{S^1}_T(U_2,\bK) \xrightarrow{i^*} C^{S^1}_T(U_1,\bK)$, which is functorial with respect to inclusions of open sets.

\item For a compactly supported Hamiltonian homeomorphism $\varphi:T^*X  \rightarrow T^*X $ in the sense of \cite{Oh-MullerHOMEO}, there exists an isomorphism $\Phi^{S^1}_{T}:C^{S^1}_T(U,\bK) \xrightarrow{\cong} C^{S^1}_T\left(\varphi(U),\bK\right)$.
\end{enumerate}
\end{iThm}

\begin{iRMK*}The results (a) and (b) of \autoref{intro-thm: prop of S^1-thy} are also proven for $C_{T}^{\bZ/\ell}(U,\bK)$ in \autoref{section: Z/l CT complex} ((d) and (e) for $C_{T}^{\bZ/\ell}(U,\bK)$ are Theorem 2.5 of \cite{Capacities2021}). In particular, the ${\bZ/\ell}$-version of (b) explains how Chiu's construction (\cite[Definition 4.6]{chiu2017}) is naturally obtained. However, the Gysin sequence (c) is only proven for $C_{T}^{S^1}(U,\bK)$, and we don't know if similar Gysin sequences exist for $C_{T}^{\bZ/\ell}(U,\bK)$.   
\end{iRMK*}

Here, let me explain the idea of the construction. For an associative algebra $(A,m)$, we consider its Hoschchild chain complex
\[\begin{tikzcd}
\cdots \arrow[r, shift left=1] \arrow[r]  \arrow[r, shift right=1] & A^{\otimes \ell} \arrow[r, shift left=1] \arrow[r]  \arrow[r, shift right=1] & \cdots \arrow[r, shift left=1] \arrow[r]  \arrow[r, shift right=1] & A^{\otimes 3} \arrow[r, "m\otimes \id", shift left] \arrow[r, "\id\otimes m"', shift right] & A^{\otimes 2} \arrow[r, "m"] & A.
\end{tikzcd}\]
We can moreover consider the cyclic permutation on $A^{\otimes \ell}$ to obtain the cyclic chain complex. Then, the cyclic homology of $A$ is defined as the homotopy fixed point of the cyclic chain complex.

Now, since $P_U$ is an associative algebra object in the monoidal category $(\cD(X^2),\star)$, it is straightforward to construct a similar ``cyclic chain complex"
\[\begin{tikzcd}
\cdots \arrow[r, shift left=1] \arrow[r]  \arrow[r, shift right=1] & P_U^{\star \ell} \arrow[r, shift left=1] \arrow[r]  \arrow[r, shift right=1] & \cdots \arrow[r, shift left=1] \arrow[r]  \arrow[r, shift right=1] & P_U^{\star 3} \arrow[r, shift left] \arrow[r,  shift right] & P_U^{\star 2} \arrow[r] & P_U.
\end{tikzcd}\]
However, the commutative relations are only true in the homotopy commutative sense since $P_U$ is an associative algebra object in a triangulated Verdier quotient $\cD(X^2)$. Consequently, we lack sufficient data to define an algebraic $S^1$-action. The correction of this idea requires more technical discussion, which we postpone to \autoref{section: cyclic structure}. Upon constructing the correct algebraic $S^1$-action, we define the $S^1$-equivariant Chiu-Tamarkin invariant as the (linear dual of the) homotopy fixed point of the algebraic $S^1$-action.
\item 
Recall that $C_T(U,\bK)=C_{T}^{\bZ/\ell}(U,\bK)$ for $\ell=1$, which is the non-equivariant version of the Chiu-Tamarkin invariant. The idempotence helps us identify the non-equivariant Chiu-Tamarkin with the (shifted) endomorphism algebra of the microlocal kernel $P_U$ (see \autoref{non-equivariant CT and yoneda form}) when $X$ is orientable. Then, we can define a (shifted) cup product on $H^qC_T(U,\bK)$ using the (shifted) Yoneda product for orientable $X$ (\autoref{definition: cup product}). The following theorem follows by a version of the argument: the endomorphism algebra of an associative object is commutative.
\begin{iThm}[{\autoref{theorem: graded commutative}}]\label{intro theorem: graded commutative}The (shifted) cup product on $H^qC_T(U,\bK)$ is unital, associative and graded commutative.
\end{iThm}

\end{enumerate}

\subsection{Geometric applications}
The first geometric application of the $S^1$-equivariant Chiu-Tamarkin invariant concerns symplectic capacities. \autoref{intro-thm: prop of S^1-thy} allows us to define a sequence of functions $(\overline{c}_k)_{k\in \bN}$ using the $S^1$-equivariant Chiu-Tamarkin invariant (see \autoref{definition of capacities}) for orientable $X$:
\begin{iThm}[\autoref{capacity property symplectic} and \autoref{prop: c1=e}] The functions $\overline{c}_k:\text{Open}(T^*X)\rightarrow [0,\infty]$ are symplectic capacities. For all open sets $U$, we have $\overline{c}_1(U)= e(P_U)$.
\end{iThm}
The construction of $(\overline{c}_k)_{k\in \bN}$ is similar to the $\bZ/\ell$-version $({c}_k)_{k\in \bN}$ that we defined \cite{Capacities2021}. However, we can prove $\overline{c}_1(U)=e(P_U)$ using the Gysin sequence, while we do not know if ${c}_1(U)=e(P_U)$ is true in general due to the absence of Gysin sequences in the $\bZ/\ell$-theory.

Therefore, it is worth comparing $(\overline{c}_k)_{k\in \bN}$ (constructed using the $S^1$-theory) and $({c}_k)_{k\in \bN}$ (constructed using the $\bZ/\ell$-theory). In \autoref{section: Z coefficient}, we consider a certain class of open sets, which we call {\em open set admitting well-behaved microlocal kernels} (see \autoref{def: well-behaved microlocal kernels} for its precise meaning), and then we demonstrate that
\begin{iThm}[\autoref{c_k=cbar_k}]
For an open set $U$ that admitting well-behaved microlocal kernels, if $H^qC_T^{S^1}(U,\bZ)$ is finitely generated for all $T\geq 0$ and all $q$, then \[\overline{c}_k(U)=c_k(U).\]
\end{iThm}
In particular, for a convex toric domain $U=X_\Omega$, the conditions here are satisfied, and the combinatorial formula for $c_k(X_\Omega)$ (\cite[Theorem 3.7]{Capacities2021}) is valid for $\overline{c}_k(X_\Omega)$.

Our second geometry application is a Viterbo isomorphism that relates the Chiu-Tamarkin invariant of the disk bundle and the string topology of the base.
\begin{iThm}[{\autoref{Viterbo isomorphism statement} and \autoref{theorem: viterbo isomorphism-product}}]\label{intro thm: viterbo isomorphism}
For a compact manifold $X$ and $T\in [0,\infty]$, we have
\begin{equation*}
H^{q}C_{T}^{G}(D^*X,\bK) \cong H_{d-q}^{G}(\mathcal{L}_{\leq T}X,\bK)
\end{equation*}
for $G={\bZ/\ell,\,S^1}$, where $\mathcal{L}_{\leq T}X=\{\gamma\in \mathcal{L}X: L(\gamma)\leq T\}$ is the free loop space of loops that have length at most $T$.
In particular, when $\ell=1$, the isomorphism intertwines the cup product on the Chiu-Tamarkin cohomology and the filtered version of the loop product on the string topology.
\end{iThm}
Consequently, we can compare the Chiu-Tamarkin invariant with the symplectic homology in the case of disk bundles based on a filtered version of Viterbo isomorphism for symplectic homology \cite[Theorem 5.3]{Cieliebak-Hingston-Oancea}.
\begin{iCoro}\label{intro coro: SH=CT}For a compact connected Spin manifold $X$, we have an isomorphism that intertwines the cup product on the Chiu-Tamarkin cohomology and the pair-of-pants product on symplectic homology
\begin{equation*}
   H^{q}C_{T}(D^*X,\bZ) \cong H_{d-q}(\mathcal{L}_{\leq T}X,\bZ)\cong SH_{-q}^{(-\infty,T)}(\overline{D}^*X,\bZ).
\end{equation*}
 \end{iCoro}
Noticed that the filtration on the free loop space in \cite{Cieliebak-Hingston-Oancea} is given by the square root of the energy functional. However, it follows a result of Anosov \cite{Anosov}, the free loop space filtered by the length is homotopy equivalent to the free loop space filtered by the square root of energy. Also, the degree shifting $d$ comes from the triviality of a local system concentrated in degree $-d$ since $X$ is Spin.
\begin{iRMK*}The Viterbo isomorphism between symplectic homology and string topology was first proposed by Viterbo in \cite{viterbo-functorsI}, and was proven using generating function homology in \cite{viterbo-functorsII}. Later, Abbondandolo-Schwarz and Salamon-Weber proved the Viterbo isomorphism using different methods (see \cite{abbondandolo2006floer,salamon2006floer}). Kragh emphasized the role of the Spin structure of the base in \cite{kragh2018viterbo}. For a survey of the Viterbo isomorphism for symplectic cohomology, see \cite{abouzaid2015symplectic}. Here, we abuse the term Viterbo isomorphism in \autoref{intro thm: viterbo isomorphism} since we expect that the Chiu-Tamarkin cohomology and symplectic homology are isomorphic in general cases. Further discussion is provided below.

Lastly, we want to highlight that the proof of \autoref{intro thm: viterbo isomorphism} relies on an explicit formula for the microlocal kernel $P_{D^*X}$ (see \autoref{kernel of unit disk bundle}), which would have independent interests."
\end{iRMK*}

\subsection{Related topics and further directions} 
\begin{itemize}[fullwidth]
    \item The $C^0$-naturality of Chiu-Tamarkin invariants: First, the Chiu-Tamarkin invariant can be defined for all open sets in cotangent bundles. Second, the Chiu-Tamarkin invariants are invariant under Hamiltonian homeomorphisms in the sense of \cite{Oh-MullerHOMEO}. A straightforward application is that we can provide a conceptual proof for the Gromov non-squeezing theorem for Hamiltonian homeomorphisms. We hope that the Chiu-Tamarkin invariant can play a role in further studies of $C^0$-symplectic geometry.
    
    \item In many cases,  it has been proven that the sheaf invariants and Floer invariants are equivalent (for example, \cite{nadler2009constructible,GPS3,viterbo2019sheafquantizationoflagrangian}). Therefore, it is natural to expect an isomorphism between the Chiu-Tamarkin invariant and symplectic (co)homology. 
    
    We have two pieces of evidence: 1) Our capacity for convex toric domains is the same as the Gutt-Hutchings capacities defined using $S^1$-equivariant symplectic homology. 2) \autoref{intro coro: SH=CT} actually shows Chiu-Tamarkin cohomology is isomorphic to symplectic homology with respect to filtration. 

    To construct such an isomorphism, we can consider the following two different approaches: 1)The work of \cite{kuo2021wrapped} provides a wrapping formula for microlocal kernels using the GKS sheaf quantization. If we establish a connection between the (global section of) GKS sheaf quantization and the Hamiltonian Floer theory based on the same idea of \cite{viterbo2019sheafquantizationoflagrangian}, we can relate the cohomology of the Chiu-Tamarkin invariant and symplectic homology.  2)We will express the Chiu-Tamarkin invariants as the Hochschild/Cyclic homology of the Tamarkin category in a work in progress. Therefore, the isomorphism will follow from a filtered equivalence between the Tamarkin category and the wrapped Fukaya category (i.e. a filtered version of \cite{GPS3}), and a filtered version of \cite{GanatraS1CY}. The second approach requires additional ingredients: The definition of the filtered Fukaya category is not completely clear as well as a filtered equivalence with Tamarkin categories (see \cite{TPF,FilteredFukayaAmbrosioni,IK-NovikovTamarkinCategory} for some attempts and discussions). Furthermore, the filtered version of \cite{GanatraS1CY} has not been explored in the literature to date. However, the advantage of the second approach is its potential for easy generalization to Weinstein manifolds, based on \cite{NS20Weinstein}."
    \end{itemize}

\subsection{Organization and Conventions of the Paper}
This paper is structured as follows: In \autoref{section: preliminary}, we present basic definitions and results for microsupport, Tamarkin category, and microlocal kernels. In \autoref{section: Z/l CT complex}, we review the $\bZ/\ell$-equivariant Chiu-Tamarkin invariant and explain some new observations and constructions. The main part of the paper is \autoref{section: s1-theory}, where we begin with an algebraic preparation in the first two subsections, followed by the definition of the $S^1$-equivariant Chiu-Tamarkin invariant. The rest of \autoref{section: s1-theory} focuses on proving the basic properties we need, with the capacities studied at the end of the section. Finally, the last part focuses on computations for disk bundles, in particular the proof of our Viterbo isomorphism. In \autoref{appendix: existence kernel}, we prove the existence of microlocal kernels for all open sets in a cotangent bundle under the setup of triangulated categories.

Next, we introduce some notation: We take $\bN=\bZ_{> 0}$ and $\bN_0=\bZ_{\geq  0}$.

We use subscripts to represent elements in sets. For example, to emphasize $a\in A$, we use the notation $A_a$. For the Cartesian product $A^n$, we define $\Delta_{A^n}:A\rightarrow A^n$ as the diagonal map, and its image is also denoted by $\Delta_{A^n}$. Projection maps are always denoted by $\pi$. Usually, the subscript encodes the fiber of the projection. For example $\pi_Y=\pi_y: X_x\times Y_y \rightarrow X_x.$ We will use $s_t^n$ to denote the summation map $\bR^n\rightarrow \bR$ and its base change.

\subsection*{Acknowledgements}Most of the work is part of the author's thesis, and I would like to express my heartfelt gratitude to Stéphane Guillermou for his generous guidance and support. I also thank Claude Viterbo for his consistent encouragement throughout the thesis. Special thanks go to Alexandru Oancea for introducing the notion of pre-cyclic objects to me, which proved to be crucial for this work. Thank for Bernhard Keller for the helpful discussion on triangulated derivators. Thank for Vivek Shende for helpful suggestions for the exposition. I am grateful for the interest and feedback of Zhen Gao, Yuichi Ike, Tatsuki Kuwagaki, and Nicolas Vichery. 

This work is supported by the ANR projects MICROLOCAL (ANR-15CE40-0007-01), COSY (ANR-21-CE40-0002), and the Novo Nordisk Foundation grant NNF20OC0066298.

\section{Sheaves and Tamarkin category}\label{section: preliminary}
In this section, we will review the concepts and tools of sheaves that we will use. Let $\bK$ be a commutative ring, such as fields or $\bZ$. For a manifold $X$, which is assumed to be connected in this article, we use $\bq\in X$ to represent both points and local coordinates of $X$. Correspondingly, points and the canonical Darboux coordinate of $T^*X$ will be denoted by $(\bq,\bp)$. We denote $D(X)$ as the derived category of complexes of sheaves of $\bK$-modules over $X$. It is worth noting that we generally do not specify the boundedness of the complexes that we use. However, in most of our applications, the complexes are locally bounded, which means that their restrictions on relatively compact open sets are bounded.

\subsection{Microsupport of Sheaves}For a locally closed inclusion $i:Z\subset X$ and $F\in D(X)$, we set
\[F_Z=i_!i^{-1}F, \quad {\tnR }\Gamma_ZF=i_*i^!F.\]
\begin{Def}[{\cite[Definition 5.1.2]{KS90}}]For $F \in D(X)$, the microsupport of $F$ is defined by
\begin{equation*}
SS(F)= \overline{\left\lbrace(\bq,\bp)\in T^*X: \begin{aligned}
    &\text{There is a }C^1\text{-function }f\text{ near }\bq\text{ such that}\\
    &f(\bq)=0\text{, }df(\bq)=\bp\text{ and }\left
({\tnR }\Gamma_{\{f\geq 0\}}F\right)_\bq\neq 0.
\end{aligned} \right\rbrace}.    
\end{equation*}
 \end{Def}
By definition, $SS(F)$ is a closed subset of $T^*X$, conic with respect to the $\mathbb{R}_{>0}$-action along fibers. It is convenient to define $\dot{T}^*X = T^*X \setminus 0_X$ and $\dot{SS}(F) = SS(F) \cap \dot{T}^*X$.

There is a triangulated inequality for the microsupport: for a distinguished triangle $A \rightarrow B \rightarrow C \xrightarrow{+1}$, we have $SS(A) \subset SS(B) \cup SS(C)$.

Regarding the necessary microsupport estimates under 6-operations, we refer to \cite[Theorem 5.4]{KS90} and \cite[Corollary 3.4]{tamarkin2013}.

Let $X_1, X_2, X_3$ be three manifolds. Recall that $\pi_X: X \times Y \rightarrow Y$ is a projection with the fiber $X$ for any arbitrary $Y$. We denote $s_t^n$ the summation map $\bR^n\rightarrow \bR$, we also abuse the notation $s_t^n$ to represent $\id_X\times s_t^n$ for all $X$.

\begin{Def}\label{def: convolution}For $F\in D(X_1\times X_2\times \bR_{t_1})$ and $G\in D(X_2\times X_3\times \bR_{t_2})$, the convolution is defined as
\[F\star G\coloneqq {\tnR }s_{t!}^2{\tnR }\pi_{X_2!}(\pi_{(X_3,t_2)}^{-1}F\dotimes \pi_{(X_1,t_1)}^{-1}G)\in D(X_1\times X_3\times \bR_t) .\]
In particular, when $X_2$ is a point, we use the notation $F_1 \boxstar F_2$ to emphasize.
\end{Def}
For $0\in \bR$ and $F\in D(X\times \bR)$, we have $ \bK_{0}\star F \cong F$, where $\bK_0\coloneqq \bK_{\{0\}}$. So, the functor $ \bK_{0}\star$ serves as the identity functor. Additionally, $\star$ satisfies the following monoidal identities:
\begin{align}\label{monoidal identities}
\begin{aligned}
&(F_1\star F_2)\star F_3 \cong F_1\star (F_2\star F_3),\quad 
& F_1 \star F_2 \cong F_2 \star F_1.
\end{aligned}
\end{align}
Here, commutative identities are induced by the identification $X_1 \times X_3 \cong X_3 \times X_1$. Their proof relies solely on the proper base change and the projection formula. Therefore, the isomorphisms are natural with respect to all sheaves involved. Consequently, the category $D(X \times X \times \bR)$ forms a symmetric monoidal category under convolution, with the unit $\mathbb{K}_{\Delta_{X^2} \times \{0\}}$. It acts on $D(X \times Y \times \bR)$ by convolution.

\subsection{Tamarkin category}\label{section Tamarkin cat}To collect numerical information, we should lift exact Lagrangians in the cotangent bundle to Legendrians in the $1$-jet bundle. Therefore, instead of considering sheaves over $X$, we should consider sheaves over $X\times\bR$, since
\[T^*(X\times \bR)=\bR_+ J^1X \sqcup J^1X\times \{0\} \sqcup \bR_{-}J^1X.\]
Here, $\bR_{\pm} J^1X$ denotes the two copies of the symplectization of $J^1X$ in $T^*(X\times\bR)$, where we consider the $1$-jet space $J^1X = T^*X\times \mathbb{R}_t$ as a contact manifold equipped with the contact form $\alpha = dt + \bp d\bq$. If we write $T^*(X\times \bR)=T^*X\times \bR_t\times \bR_{\tau}$, and the symplectic form is $\omega = d\bp\wedge d\bq + d\tau\wedge dt$, then we can write $\bR_+ J^1X=\{\tau>0\}$. The symplectization map can be represented by the following map $q$. The composition of $q$ with the Lagrangian projection is the symplectic reduction of $T^*X\times T^*_{\tau>0}\bR_t$ with respect to the hypersurface ${\tau=1}$, denoted by $\rho$. Explicitly, we have
\begin{align*}
    q:&\, \bR_+ J^1X=\{\tau>0\} \rightarrow J^1X,\, (\bq,\bp,t,\tau) \mapsto (\bq,\bp/\tau,t),\\
     \rho:\,&\bR_+ J^1X=\{\tau>0\} \rightarrow T^*X,\, (\bq,\bp,t,\tau) \mapsto (\bq,\bp/\tau),
\end{align*}
and the following commutative diagram
\begin{equation*}
 \begin{tikzcd}
\bR_+ J^1X=\{\tau>0\} \arrow[r,"q"] \arrow[rr, "\rho", bend right] & J^1X \arrow[r] & T^*X.
\end{tikzcd}    
\end{equation*}

Recall that the microsupport is a closed conic subset of cotangent bundles. As $q$ is a $\bR_{>0}$-invariant quotient map, we can address the conicity issue by introducing the following two sets, which are not necessarily conic: For sheaves $F\in D(X\times \bR_t)$, we define
\begin{align}\label{mu supp}
\begin{aligned}
        {\mu}s_L(F)=q\left(SS(F)\cap \{\tau> 0\}\right)\subset J^1X,\\
    {\mu}s(F)=\rho\left(SS(F)\cap \{\tau> 0\}\right)\subset T^*X.
\end{aligned}
\end{align}
We refer to $\mu s_L(F)$ as the Legendre microsupport and $\mu s(F)$ as the sectional microsupport.

Meanwhile, the information for $\tau \leq 0$ will not be taken into account in ${\mu}s_L(F)$ and ${\mu}s(F)$. To address this, we have two options. First, we can categorically localize sheaves that are microsupported in $\{\tau \leq 0\}$. Alternatively, we can utilize the concept of $\gamma$-sheaves and the microlocal cut-off lemma (\cite[Proposition 5.2.3]{KS90}) to cut off the negative part of the microsupport. Using either of these methods, we can ensure $SS(F)\subset \{\tau \geq 0\}$. In this case, the microlocal Morse lemma \cite[Theorem 5.4.5(ii)(c)]{KS90} guarantees that we do not lose significant information. Now we will explain both approaches.

Consider the full subcategory $\mathcal{C}=\left\{F: SS(F)\subset \{\tau \leq 0\}\right\}$ of $D(X\times\bR)$. The triangulated inequality of microsupport shows that $\mathcal{C}$ is a thick triangulated subcategory of $D(X\times\bR)$.

\begin{Def}\label{tamarkincategory} We define the Tamarkin category as the following Verdier localization:
\[\mathcal{D}(X)= D(X\times \bR) /\left\{F: SS(F)\subset \{\tau \leq 0\}\right\}.\]\end{Def}
As a Verdier localization, the Tamarkin category $\mathcal{D}(X)$ is itself a triangulated category.

By the definition of localization, if $F$ and $F'$ are two representatives in $D(X^2\times \bR)$ for an object $[F]$ in $\mathcal{D}(X)$, there exists a distinguished triangle $F \rightarrow F' \rightarrow G \xrightarrow{+1}$ such that $SS(G)\subset \{\tau \leq 0\}$. Therefore, it makes sense to define
\[{\mu}s_L([F]) \coloneqq {\mu}s_L(F),\qquad {\mu}s([F]) \coloneqq {\mu}s(F).\]
However, computing $\HOM$ in $\mathcal{D}(X)$ is not straightforward. To address this issue, Tamarkin proved the following theorem:
\begin{Thm}[\cite{tamarkin2013}]\label{Tamarkin projector}
The functors $F \mapsto\bK_{[0,\infty)}\star F $ and $F \mapsto \bK_{(0,\infty)}[1]\star F$ on $D(X \times\bR_t)$, along with the excision triangle
$\bK_{[0,\infty)} \rightarrow \bK_{0} \rightarrow  \bK_{(0,\infty)}[1]\xrightarrow{+1}$, 
provide a left semi-orthogonal decomposition of $D(X\times\bR_t)$ associated with the subcategory $\left\{F: SS(F)\subset \{\tau \leq 0\}\right\}$.

 In other words, for $F\in D(X \times\bR_t)$, we have a distinguished triangle
\begin{equation}
\bK_{[0,\infty)}\star F\rightarrow F \rightarrow \bK_{(0,\infty)}[1]\star F \xrightarrow{+1}, \label{Tamarkinorthgonalcomplement}
\end{equation}
where $\bK_{[0,\infty)}\star F\in {^\perp} \left\{F: SS(F)\subset \{\tau \leq 0\}\right\}$, $\bK_{(0,\infty)}[1]\star  F\in \left\{F: SS(F)\subset \{\tau \leq 0\}\right\}$.
\end{Thm}
Therefore, with the help of {\cite[Chapter 4 and Exercise 10.15]{KS2006}}, we establish that the functor
\[\cD(X)\rightarrow D(X^2\times \bR), [F]\rightarrow F\star \bK_{[0,\infty)}\]
is a well-defined fully faithful functor known as the Tamarkin projector. Here, $F\star \bK_{[0,\infty)}$ serves as a canonical representative for $[F]\in \cD(X)$. The essential image of this functor is the left orthogonal complement ${^\perp} \left\{F: SS(F)\subset \{\tau \leq 0\}\right\}$.

In practice, it is often convenient to identify $\cD(X)$ with ${^\perp} \left\{F: SS(F)\subset \{\tau \leq 0\}\right\}$ for computational purposes and related constructions. This allows us to directly compute $\HOM$ in $D(X^2\times \bR)$ for objects in $\cD(X)$ with their canonical representatives. Throughout this paper, we will primarily focus on the left orthogonal complement model without explicitly mentioning it. Specifically, we can say a sheaf $F\in D(X^2\times \bR)$ is an object of $\cD(X)$ if and only if $F\cong F\star \bK_{[0,\infty)}$. Consequently, the convolution functor $\bK_{\Delta_{X^2}\times [0,\infty)}\star$ induces the identity functor of the Tamarkin category $\mathcal{D}(X)$.

For the canonical representative $F\cong F\star \bK_{ [0,\infty)}$ of $[F]\in \mathcal{D}(X)$, one can show $SS(F)\subset \{\tau \geq 0\}$ using functorial estimates of microsupport, see \cite[Proposition 3.17]{GS2014}. It is also proven by \cite[Lemma 3.7]{tamarkin2013} that if $\mu s(F)$ is compact, then $SS(F)\cap \{\tau \leq 0\}=0_{X\times \bR}$.

On the other hand, we consider the product space $X\times \bR_\gamma$, where $\bR_\gamma$ denotes $\bR$ equipped with the $\gamma$-topology as defined in \cite[Section 3.5]{KS90} for $\gamma=(-\infty,0]$. According to the microlocal cut-off lemma (Proposition 5.2.3 in {\it loc.cit.}), we have the equivalence between $D(X\times \bR_\gamma)$ and the category $\{F\in D(X\times \bR): SS(F)\subset {\tau \geq 0}\}$ as well as $\{F\in D(X\times \bR): \bK_{[0,\infty)}\star_{np}F\cong F\}$, where $\bK_{[0,\infty)}\star_{np}F\coloneqq \tnR s_{t*}^2( \bK_{[0,\infty)}\boxtimes F)$. Objects in these equivalent categories are referred to as $\gamma$-sheaves. Consequently, $\cD(X)$ is equivalent to a full subcategory of $D(X\times \bR_\gamma)$.

Both approaches have advantages in applications. In this article, our main focus is on Tamarkin's approach. However, the microlocal cut-off approach is very useful when considering the quantization of the Reeb flow in the next subsection.

If $X$ admits a $G$-action and we equip the trivial action on $\bR$, we can define the equivariant Tamarkin category $\cD_G(X)\subset D_G(X\times \bR)$, where $D_G$ represents the equivariant derived category as defined in \cite{BernsteinLunts}. The construction follows the same principles since the microsupport of an equivariant sheaf is defined as the microsupport of the corresponding non-equivariant sheaf. We can also define the Tamarkin projector using the equivariant 6-operations, and the previous discussion applies to equivariant $\gamma$-sheaves as well.

\subsection{Quantization for the Reeb flow of the 1-jet }\label{section: reeb action}Consider, the translation map $\text{T}_c:X\times \bR\rightarrow X\times \bR$, $(x,t)\mapsto (x,t+c)$. Then the push forward map $\text{T}_{c*}: D(X\times \bR)\rightarrow D(X\times \bR)$ is a family of endfunctors. 

For $F\in D(X\times \bR_\gamma)$, we have $\tnT_{c*}\cong -\star_{np} \bK_{{c}} \cong -\star_{np} \bK_{[c,\infty)}$. Therefore, we obtain a family of natural transformations $\tau_c: \id\rightarrow \tnT_{c*}$, for $c\geq 0$, induced by the restriction map $\bK_{[0,\infty)}\rightarrow \bK_{[c,\infty)}$ in $D(X\times \bR_\gamma)$. For a $\gamma$-open set $U=U+\gamma$, the natural morphism $\tau_c(F)$ is induced by
\[\tnR \Gamma(U,F)\rightarrow \tnR \Gamma(U+c,F)\cong \tnR \Gamma(U,\tnT_{c*}F).\]
Therefore, $\tnT_{c*}$ and $\tau_c$ commute with the 6-operations and adjunctions on $D(X\times \bR_\gamma)$. For example, if $f:X\rightarrow Y$, consider ${f}_\bR=f\times\id_{\bR}: X\times \bR\rightarrow Y\times \bR$. Then ${f}_\bR$ is a continuous map on $X\times \bR_\gamma$, and we have
\[\tau_c(\tnR {f}_{\bR*}F)=\tnR {f}_{\bR*}(\tau_c(F)): \tnR {f}_{\bR*}F\rightarrow  \tnT_{c*}\tnR {f}_{\bR*}\cong \tnR {f}_{\bR*}\tnT_{c*}F. \]
For $F\in \cD(X)$, we also have $\tnT_{c*}\cong \bK_{\{c\}}\star_{np}-  \cong \bK_{\{c\}}\star-  \cong \bK_{[c,\infty)}\star- $. Then the natural transformation $\tau_c$ restricts to $\cD(X)$, and it also commutes with the 6-operations and adjunctions. In particular, $\cD(X)$ is a triangulated persistent category, as defined in \cite{TPC}. The same discussion applies if $X$ possesses a group action and $f$ is an equivariant map.

Geometrically, the functor $\tnT_{c*}$ quantizes the Reeb flow of the canonical contact form on $J^1X$. Moreover, it suggests that the additional variable $\bR_t$ associated with objects in the Tamarkin category $\cD(X)$ represents a form of energy (or action), and the natural transformation $\tau_c$ ``moves" Reeb chords in a homological manner. With this perspective, we introduce our first numerical invariant of sheaves.
\begin{Def}\label{sheafenergy}For $F\in \mathcal{D}(X)$, the sheaf energy is defined to be: 
\[e(F)=\inf\{c \geq 0: \tau_c(F)=0\}.\]
\end{Def}

\subsection{Tamarkin categories of subsets and microlocal projectors}In this subsection, we will study the categories of sheaves microsupported in an open set $U\subset T^*X$. Subsequently, we will construct the kernels of the projectors associated with these categories.

For the {\em closed} subset $Z=T^*X\setminus U \subset T^*X$, we define $\mathcal{D}_Z(X)$ as the full subcategory of $\mathcal{D}(X)$ consisting of sheaves satisfying $\mu s(F) \subset Z$. Additionally, we define $\mathcal{D}_{U}(X)$ as the Verdier localization $\mathcal{D}_{U}(X)=\cD(X)/\cD_Z(X)$.

Now we have a diagram of triangulated categories
\begin{equation}
 \mathcal{D}_{Z}(X) 
\hookrightarrow \mathcal{D}(X) \rightarrow \mathcal{D}_{U}(X) \label{inclusion diagram of categories supported in U}  
\end{equation}
\begin{Def}\label{defadmissibledomains}We say that $U$ is {\em $\bK$-admissible} if there exists a distinguished triangle 
\[ P_U\rightarrow {\bK}_{\Delta_{X^2} \times[0,\infty)} \rightarrow Q_U \xrightarrow{+1},\]
in $\cD(X^2)$ such that the convolution functor $\star P_U $ induces the left adjoint of the quotient functor $\mathcal{D}(X) \rightarrow \mathcal{D}_{U}(X)$, and $\star Q_U$
induces the left adjoint of the natural inclusion $\mathcal{D}_{Z}(X) 
\hookrightarrow \mathcal{D}(X)$.

The pair of sheaves $(P_U,Q_U)$, together with the distinguished triangle, provides an orthogonal decomposition of $\mathcal{D}(X)$. In particular, we can identify $\mathcal{D}_{U}(X)$ as the left orthogonal complement of $ \mathcal{D}_{Z}(X) $ in $\mathcal{D}(X)$.

We call the pair $(P_U,Q_U)$ as the {\em microlocal kernels} associated with $U$, and the distinguished triangle as the defining triangle of $U$. We say that $U$ is {\em admissible} if $U$ is $\bZ$-admissible.
\end{Def}

\begin{RMK}\label{remark: one kernel is enough}This definition differs slightly from the one given in \cite{Capacities2021}. However, with the aid of \autoref{all open sets admissible}, we can demonstrate the equivalence of two definitions. Furthermore, the existence of either $P_U$ or $Q_U$ is sufficient to ensure the existence of the other.
\end{RMK}

Chiu has proven that balls are admissible. This idea can be directly generalized to bounded open sets (\cite[Subsection 2.3]{zhangthesis}). In \autoref{appendix: existence kernel}, we will prove the following result:
\begin{Thm}\label{all open sets admissible}
All open sets in a cotangent bundle are admissible.
\end{Thm}
\begin{RMK}\label{remark: kernel in infinity category}Actually, Kuo's results in \cite{kuo2021wrapped} are sufficient to provide a proof of this fact within the framework of the stable $\infty$-category. This was missed in \cite{Capacities2021} due to a gap in the author's knowledge. Therefore, the main point of the proof presented in \autoref{appendix: existence kernel} is that we only rely on techniques from triangulated categories (and triangulated derivators).
\end{RMK}
Now, we state the basic properties of microlocal kernels that are crucial for our later discussions.
\begin{Lemma}[{\cite[Lemma 2.3]{Capacities2021}}]\label{functorial lemma}Suppose $U_1 \subset U_2 $ is a inclusion of two open subsets in $T^*X$, and their defining triangles are
    \[ P_{U_i}\xrightarrow{a_i} {\bK}_{\Delta_{X^2} \times[0,\infty)} \xrightarrow{b_i} Q_{U_i} \xrightarrow{+1},\qquad i=1,2.\]
\begin{enumerate}[fullwidth]\item We have $Q_{U_2}\star P_{U_1}\cong 0$, and the natural morphism
\[a_2\star P_{U_1} =[P_{U_2} \star P_{U_1}\rightarrow P_{U_1}],\]
is an isomorphism. In particular, we have $P_{U}\star P_{U}\cong P_{U}$ and $P_{U}\star Q_{U}\cong Q_{U}\star P_{U}\cong 0$ for all open sets.
\item  For any open set $U$ and for all $F,G \in D(X^2 \times \bR)$, we have the isomorphism:\[\HOM_{D(X^2\times \bR)}(F\star P_{U} ,G\star P_{U})\rightarrow  \HOM_{D(X^2\times \bR)}(F\star P_{U},G).\]
\item For all $c\geq 0$, we have $\RHOM(P_{U_1},\tnT_{c*}(Q_{U_2}))\cong 0$ and
\begin{equation}
  \RHOM(P_{U_1},\tnT_{c*}(a_2)):\,\RHOM(P_{U_1},\tnT_{c*}(P_{U_2})) \cong\RHOM(P_{U_1},{\bK}_{\Delta_{X^2} \times[c,\infty)}) .\end{equation}
    \end{enumerate}
\end{Lemma}
\begin{Coro}[{Idempotence of microlocal kernels}]\label{idempotence}For an open subset $U\subset T^*X$ and all $\ell\in \bN$, we have isomorphisms
\[P_U^{\star \ell }\xrightarrow{\cong} P_U, \qquad Q_U\xrightarrow{\cong} Q_U^{\star \ell }.\]
\end{Coro}
\begin{Prop}[{Functorial of microlocal kernels, see \cite[Lemma 2.4]{Capacities2021}}]\label{functorial}For any inclusion $U_1 \subset U_2 \subset T^*X$ between open subsets and their defining triangles are
    \[ P_{U_i}\xrightarrow{a_i} {\bK}_{\Delta_{X^2} \times[0,\infty)} \xrightarrow{b_i} Q_{U_i} \xrightarrow{+1},\]then we have a morphism between the defining triangles:
    \begin{equation*}
        \begin{tikzcd}
P_{U_1} \arrow[d, "a"] \arrow[r, "a_1"] & {{\bK}_{\Delta_{X^2}\times[0,\infty)}} \arrow[r, "b_1"] \arrow[d, "\id"] & Q_{U_1} \arrow[d, "b"] \arrow[r, "+1"] & {} \\
P_{U_2} \arrow[r, "a_2"]                & {{\bK}_{\Delta_{X^2}\times[0,\infty)}} \arrow[r, "b_2"]                  & Q_{U_2} \arrow[r, "+1"]                &  { .}
\end{tikzcd}
    \end{equation*}
These morphisms $a,b$ are natural with respect to inclusions of open sets. In particular, when $U_1=U_2$ (but $P_{U_1}$ and $P_{U_2}$ are a priori not the same), the morphism of the defining triangles is an isomorphism of distinguished triangles. \end{Prop}

\section{\texorpdfstring{$\bZ/\ell$-equivariant Chiu-Tamarkin invariant}{}}\label{section: Z/l CT complex}
In this section, we study the $\bZ/\ell$-equivariant Chiu-Tamarkin invariant for open sets $U\subset T^*X$. We fix a defining triangle of $U$:
\[ P_U\rightarrow {\bK}_{\Delta_{X^2} \times[0,\infty)} \rightarrow Q_U \xrightarrow{+1}.\]

\subsection{Non-equivariant Chiu-Tamarkin invariant}Consider the following maps:
\[\bR \xleftarrow{\pi_X} X\times \bR \xrightarrow{\Delta_{X^2}} X^2\times \bR.\]
We define 
\begin{equation}
    F_1(U,\bK)\coloneqq \tnR\pi_{X!}\Delta_{X^2}^{-1}P_U,\quad F_1^{out}(U,\bK)\coloneqq \tnR\pi_{X!}\Delta_{X^2}^{-1}Q_U.
\end{equation}
\begin{Def}\label{definition of non-eq CT}For an open set $U$ and $T\geq 0$, we define an object of $D({\bK}-\Mod)$ that we call the non-equivariant Chiu-Tamarkin invariant  
\begin{align*}
    C_{T}(U,\bK)
    &=\RHOM\left((F_{1}(U,\bK))_T, \bK[-d]\right).
\end{align*}
 
We define the positive Chiu-Tamarkin invariant, also in $D({\bK}-\Mod)$,  
\begin{align*}
    C^{out}_{T}(U,\bK)&=\RHOM\left((F^{out}_{1}(U,\bK))_T,\bK[-d]  \right).
\end{align*}
\end{Def}
The definition also appears in \cite[Section 4.7]{zhang2020quantitative}.

Combining the defining triangle and the definition, we have the following tautological triangle 
\begin{equation}\label{tautological triangle: non eq}
C_{T}^{out}(U,\bK)\rightarrow \tnR \Gamma(X,\omega_X)\rightarrow C_{T}(U,\bK)  \xrightarrow{+1},
\end{equation}
where $\omega_X$ is the dualizing sheaf of $X$.

As we have a colimit formula for microlocal kernels (\autoref{coro: colimit formula for kernel in general}), and $\tnR\pi_{X!}\Delta_{X^2}^{-1}$ is a cocontinuous functor, we obtain the following limit formula for the Chiu-Tamarkin invariant:
\begin{Prop}\label{coro: colimit formula}For an exhaustion of open sets $(U_n)_{n\in \bN}$ of $U$, i.e. $\overline{U_n}\subset U_{n+1}$ and $U=\bigcup_n U_n$, we have
\[C_T(U,\bK)=\textnormal{holim}_nC_T(U_n,\bK), \qquad C_T^{out}(U,\bK)=\textnormal{holim}_nC_T^{out}(U_n,\bK).\]    
\end{Prop}
Recall that symplectic homology is defined directly for Liouville domains \cite{CFH,CFHW,viterbo-functorsI,OanceaSHsurvey} or Liouville sectors \cite{GPS2}, and then the definition for general Liouville manifolds is obtained as a limit of symplectic homology of domains that exhaust it \cite{seidel2006biased}. This property allows us to draw an analogy with sheaves.

\begin{Prop}\label{non-equivariant CT and yoneda form}If $X$ is orientable, we have 
\[C_{T}(U,\bK)\cong \RHOM(P_{U},\tnT_{T*}(P_{U})).\]    
\end{Prop}
\begin{proof}Using adjunction, we have the isomorphism 
\[N: C_{T}(U,\bK)
    \cong \RHOM\left(P_U, \Delta_{X^2*}\pi_{X}^{!}\bK_{\{T\}}[-d]\right).\]
However, when $X$ is orientable, we have
$\Delta_{X^2*}\pi_{X}^{!}\bK_{\{T\}}[-d]\cong {\bK}_{\Delta_{X^2} \times \{T\}}$. Then the result follows from \autoref{functorial lemma}.
\end{proof}

\subsection{\texorpdfstring{$\bZ/\ell$-equivariant Chiu-Tamarkin invariant}{}}Now, let me explain how to establish the $\bZ/\ell$-naturality of the Chiu-Tamarkin invariant(\cite{chiu2017,Capacities2021}). We refer to \cite{BernsteinLunts} for the definition of equivariant derived category and 6-functors formalism. 

The manifold $(X^2\times \bR_t)^{\ell}$ admits a $\bZ/\ell$-action induced by the cyclic permutation of factors. Applying the Steenrod's construction \cite[Section 2.2]{lonergan_2021}, we can lift the object $P_U^{\dboxtimes \ell}$ from $D((X^2\times \bR_t)^\ell)$ to an object $St_D(P_U)$ in the equivariant derived category $D_{\bZ/\ell}((X^2\times \bR_t)^\ell)$. We will still denote this lifting by $P_U^{\dboxtimes \ell}$. Consequently, we have $P_U^{\boxstar \ell}=\tnR s_{t!}^\ell P_U^{\dboxtimes \ell} \in D_{\bZ/\ell}((X^2)^\ell \times \bR_t)$.

Consider the $\bZ/\ell$-equivariant maps
\begin{align*}
   \pi_{\underline{\bq}}&: X^\ell \times \bR\rightarrow \bR,\\
    \Tilde{\Delta}_X&: X^\ell \times \bR\rightarrow X^{2\ell} \times \bR,\\
    \Tilde{\Delta}_X(\bq_1,\dots,\bq_\ell,t)&=(\bq_\ell,\bq_1,\bq_1,\dots,\bq_{\ell-1},\bq_{\ell-1},\bq_\ell,t),\\
    i_T&: \pt\rightarrow \bR.
\end{align*}
There exists an adjoint pair $(\alpha_{\ell,T,X},\beta_{\ell,T,X})$:
\begin{equation*}
 \begin{tikzcd}
F\in D_{\bZ/\ell}((X^2\times \bR_t)^\ell) \arrow[rr, "\alpha_{\ell,T,X}", shift left] &  & D_{\bZ/\ell}(\pt) \ni G, \arrow[ll, "\beta_{\ell,T,X}"]
\end{tikzcd}   
\end{equation*}
defined by:
\begin{equation}\label{definition of adjoint loop functor}
\alpha_{\ell,T,X}(F)=i_T^{-1}{\tnR }\pi_{\underline{\bq}!}   \Tilde{\Delta}_X^{-1}{\tnR }s_{t!}^\ell\left(F\right),\quad\beta_{\ell,T,X}(G)=s_t^{\ell!}\Tilde{\Delta}_{X*} \pi_{\underline{\bq}}^!i_{T*} G.
\end{equation}
Now, we define a functor 
\begin{equation}\label{definition of F}
    F_{\ell,X}={\tnR }\pi_{\underline{\bq}!}   \Tilde{\Delta}_X^{-1}{\tnR }s_{t!}^\ell: D_{\bZ/\ell}((X^2\times \bR_t)^\ell) \rightarrow D_{\bZ/\ell}(\bR).
\end{equation}
Then $\alpha_{\ell,T,X}=i_T^{-1} F_{\ell,X}$.

Similarly, we define another adjoint pair $(\alpha'_{\ell,T,X},\beta'_{\ell,T,X})$:
\begin{equation*}
  \begin{tikzcd}
F\in D_{\bZ/\ell}((X^2)^\ell\times \bR_t) \arrow[rr, "\alpha'_{\ell,T,X}", shift left] &  & D_{\bZ/\ell}(\pt) \ni G, \arrow[ll, "\beta'_{\ell,T,X}"]
\end{tikzcd}  
\end{equation*}
\begin{equation}\label{definition of adjoint loop functor2}\alpha'_{\ell,T,X}(F)=i_T^{-1}{\tnR }\pi_{\underline{\bq}!}\Tilde{\Delta}_X^{-1}\left(F\right)\quad \beta'_{\ell,T,X}(G)=\Tilde{\Delta}_{X*} \pi_{\underline{\bq}}^!i_{T*} G.
\end{equation}
We also define a functor 
\begin{equation}\label{definition of F'}
    F'_{\ell,X}={\tnR }\pi_{\underline{\bq}!}  \Tilde{\Delta}_X^{-1}: D_{\bZ/\ell}(X^{2\ell}\times \bR_t) \rightarrow D_{\bZ/\ell}(\bR).
\end{equation}
Then $\alpha'_{\ell,T,X}=i_T^{-1} F'_{\ell,X}$.

We denote $(\alpha_{\ell,T,X},\beta_{\ell,T,X})=(\alpha_T,\beta_T)$ and $(\alpha'_{\ell,T,X},\beta'_{\ell,T,X})=(\alpha'_T,\beta'_T)$ for simplicity. We will frequently use $\alpha_{\ell,T,X}, \beta_{\ell,T,X}$ ($\alpha'_{\ell,T,X}, \beta'_{\ell,T,X}$), and $ F_{\ell,X}$ ($ F'_{\ell,X}$) in the non-equivariant categories. We denote them by the same notation later.
\begin{Def}\label{definition of CT}With the notation above, we define an object of $D_{\bZ/\ell}(\pt)\simeq D({\bK}[{\bZ/\ell}]-\Mod)$ that we call the $\bZ/\ell$-equivariant Chiu-Tamarkin invariant 
\begin{align*}
    C_{T}^{\bZ/\ell}(U,\bK)&=\RHOM_{\bZ/\ell}\left(\alpha_{\ell,T,X}(P_U^{\dboxtimes \ell}),\bK[-d]  \right)\\
    &=\RHOM_{\bZ/\ell}\left((F_{\ell}(U,\bK))_T, \bK[-d]\right)\\
    &\cong\RHOM_{\bZ/\ell}\left(P_U^{\dboxtimes \ell},\beta_{\ell,T,X}\bK[-d]   \right), 
\end{align*}
where $F_{\ell}(U,\bK)=F_{\ell,X}(P_U^{\dboxtimes \ell})=F'_{\ell,X}(P_U^{\scriptstyle\boxstar \ell})$.

We define the positive $\bZ/\ell$-equivariant Chiu-Tamarkin invariant, also in $D_{\bZ/\ell}(\pt)$ 
\begin{align*}
    C^{\bZ/\ell,out}_{T}(U,\bK)&=\RHOM_{\bZ/\ell}\left(\alpha_{\ell,T,X}(Q_U^{\dboxtimes \ell}),\bK[-d]  \right)\\
    &=\RHOM_{\bZ/\ell}\left((F^{out}_{\ell}(U,\bK))_T, \bK[-d]\right)\\
    &\cong\RHOM_{\bZ/\ell}\left(Q_U^{\dboxtimes \ell},\beta_{\ell,T,X}\bK[-d]   \right), 
\end{align*}
where $F_{\ell}^{out}(U,\bK)=F_{\ell,X}(Q_U^{\dboxtimes \ell})=F'_{\ell,X}(Q_U^{\boxstar \ell})$.
\end{Def}
When $\ell=1$, i.e. the non-equivariant case, we recover $C_{T}(U,\bK)= C^{\bZ/1}_{T}(U,\bK)$ and $C^{out}_{T}(U,\bK)= C^{\bZ/1,out}_{T}(U,\bK)$. 

The cohomology groups $H^*C^{\bZ/\ell}_{T}(U,\bK)$ and $H^*C^{\bZ/\ell,out}_{T}(U,\bK)$ are graded modules over the $\bK$- algebra $\text{Ext}_{\bZ/\ell}^*(\bK,\bK)\cong H^*(B\bZ/\ell,\bK)$ via the Yoneda product.

Now, we can explain the first application of the idempotence (\autoref{idempotence}):
\begin{Prop}\label{cyclicstructure step1}For an open set $U$, we have isomorphisms in the non-equivariant derived category $D(\bR)$:
\[F_{\ell}(U,\bK)\cong F_{1}(U,\bK)\qquad F^{out}_{1}(U,\bK)\cong F^{out}_{\ell}(U,\bK).\]
\end{Prop}
\begin{proof}Only prove the first one, the second one is proven in the same way. As \autoref{idempotence} proven, we have $P_U^{\star \ell}\cong P_U$. Then we only need to show that $F_{\ell, X}(P_U^{\dboxtimes \ell})  {\cong} F_{1,X}(P_U^{\star \ell})$. Consider
\[d: X^{\ell+1}\times\bR_t\rightarrow X^{2\ell}\times\bR_t,\quad d(\bq_0,\dots,\bq_{\ell},t)=(\bq_0,\bq_1,\bq_1,\dots,\bq_{\ell-1},\bq_{\ell-1},\bq_{\ell},t).\]
We have
\begin{equation}\label{Long expression of convolution}
H_1\star H_2\star \cdots H_\ell \cong \tnR \pi_{(\bq_1,\dots,\bq_{\ell-1})!}d^{-1}\tnR s_{t!}^\ell(H_1\dboxtimes H_2\dboxtimes \cdots \dboxtimes H_\ell).
\end{equation}
On the other hand, recall the proof of $P_U^{\star \ell}\xrightarrow{\cong} P_U$, which is given by $P_U^{\star \ell}\xrightarrow{\cong} \bK_{\Delta_{X^2}\times [0,\infty)}^{\star (\ell -1)} \star P_U$. Then the isomorphism is obtained by applying $G=\tnR \pi_{(\bq_1,\dots,\bq_{\ell-1})!}d^{-1}\tnR s_{t!}^\ell$ to the following morphism
\[P_U^{\dboxtimes \ell}\rightarrow \bK_{\Delta_{X^2}\times [0,\infty)}^{\dboxtimes (\ell -1)} \dboxtimes P_U.\]
But the projection formula and the base change formula provide a natural isomorphism of functors $F_{1,X}\circ G\cong F_{\ell,X}$.
Then the desired isomorphism follows. 
\end{proof}
Consequently, under the forget functor $For_{\bZ/\ell}:D_{\bZ/\ell}\rightarrow D$, we have
\begin{equation}\label{eq: equivariant lifting}
For_{\bZ/\ell}( C^{\bZ/\ell}_{T}(U,\bK))\cong  C_{T}(U,\bK),\qquad For_{\bZ/\ell}( C^{\bZ/\ell,out}_{T}(U,\bK))\cong C^{out}_{T}(U,\bK)).    \end{equation}
This implies that, as expected, the $\bZ/\ell$-equivariant Chiu-Tamarkin invariant is the equivariant lifting of the non-equivariant Chiu-Tamarkin invariant.

We also anticipate an equivariant version of the tautological triangle \eqref{tautological triangle: non eq}. However, since Steenrod's construction is not a triangulated functor, this is not immediately apparent.
\begin{Prop}For an open set $U$, we have a distinguished triangle in the equivariant derived category:
\begin{equation}\label{pre Tautological exact triangle}
F_\ell(U,\bK)\rightarrow \tnR \Gamma_c(X,\bK)\rightarrow F^{out}_\ell(U,\bK)\xrightarrow{+1}.
\end{equation}
Then a distinguished triangle for the Chiu-Tamarkin invariants follows:
\begin{equation}\label{Tautological exact triangle}
C^{\bZ/\ell,out}_{T}(U,\bK)\rightarrow \tnR \Gamma(X,\omega_X^{\bZ/\ell})\rightarrow C^{\bZ/\ell}_{T}(U,\bK)  \xrightarrow{+1}.
\end{equation}
We call them the {\em tautological exact triangles} for the Chiu-Tamarkin invariant.
\end{Prop}
\begin{proof}Take the defining triangle of $U$, say:
\[P_U\xrightarrow{a} \bK_{\Delta_{X^2}\times [0,\infty)} \xrightarrow{b} Q_U\xrightarrow{+1},\]
then we have $ba=0$. 

By taking Steenrod operation and $\tnR s_{t!}^\ell$, we have two morphisms in the equivariant derived categories: \[P_U^{\boxstar\ell }\xrightarrow{a^{\boxstar\ell }} \bK_{\Delta_{X^2}\times [0,\infty)}^{\boxstar\ell } \xrightarrow{b^{\boxstar\ell }} Q_U^{\boxstar\ell },\]
and their composition is $0$.

Let us take the cone of $a^{\boxstar\ell }$ in the equivariant category, then we have a distinguished triangle in the equivariant category:
\begin{equation}\label{dt of a ell}
P_U^{\boxstar\ell }\xrightarrow{a^{\boxstar\ell }} \bK_{\Delta_{X^2}\times [0,\infty)}^{\boxstar\ell } \xrightarrow{c^\ell} \mathcal{C}_\ell \xrightarrow{+1}.   \end{equation}Since $b^{\boxstar\ell }a^{\boxstar\ell }=0$, by applying the cohomological functor $\HOM_{\bZ/\ell}(-,Q_U^{\boxstar\ell })$, we have a morphism $\psi:\mathcal{C}_\ell \rightarrow Q_U^{\boxstar\ell }$ that fits into the commutative diagram in the equivariant category:
\begin{equation*}
 \begin{tikzcd}
P_U^{\boxstar\ell } \arrow[r, "a^{\boxstar\ell }"] & {\bK_{\Delta_{X^2}\times [0,\infty)}^{\boxstar\ell } } \arrow[r, "c^\ell"] \arrow[rd, "b^{\boxstar\ell }"] & \mathcal{C}_\ell \arrow[d, "\psi", dashed] \arrow[r, "+1"] & {} \\
                                                       &                                                                                                                & Q_U^{\boxstar\ell }.                             &   
\end{tikzcd}   
\end{equation*}
However, we do not know if $\psi$ is an isomorphism. But we can show that $F'_\ell(\psi)$ (see \eqref{definition of F'}) is an isomorphism in the equivariant category. Then the distinguished triangle \eqref{pre Tautological exact triangle} follows.

We will argue that $F'_\ell(\psi)$ is an isomorphism in the following steps. Let us consider the forgetful functor $For_{\bZ/\ell}: D_{\bZ/\ell}\rightarrow D$. We will show that {\em there exists an isomorphism $\phi: For_{\bZ/\ell}(F'_\ell(\mathcal{C}_\ell))\rightarrow For_{\bZ/\ell}(F'_\ell(Q_U^{\boxstar\ell }))$ such that $For(F'_\ell(\psi))=\phi$ in the non-equivariant derived category $D$. }However, one can prove that $For_{\bZ/\ell}$ is a conservative functor. Then $F'_\ell(\psi)$ is an isomorphism in the equivariant derived category $D_{\bZ/\ell}$.

Therefore, we will only consider the non-equivariant derived category and then we omit the functor $For_{\bZ/\ell}$ in the following.

Using the morphism $P_U^{\dboxtimes \ell}\rightarrow \bK_{\Delta_{X^2}\times [0,\infty)}^{\dboxtimes \ell -1} \dboxtimes P_U$, we embed \eqref{dt of a ell} into the following diagram:
\begin{equation*}
    \begin{tikzcd}
P_U^{\boxstar\ell } \arrow[r, "a^{\boxstar\ell }"] \arrow[d]                        & {\bK_{\Delta_{X^2}\times [0,\infty)}^{\boxstar\ell }} \arrow[r, "c^\ell"] \arrow[d, "="] & \mathcal{C}_\ell \arrow[r, "+1"]                                                          & {} \\
{\bK_{\Delta_{X^2}\times [0,\infty)}^{\boxstar \ell -1} \boxstar P_U} \arrow[r] & {\bK_{\Delta_{X^2}\times [0,\infty)}^{\boxstar\ell }} \arrow[r]                          & {\bK_{\Delta_{X^2}\times [0,\infty)}^{\boxstar \ell -1} \boxstar Q_U} \arrow[r, "+1"] & {}
\end{tikzcd}  
\end{equation*}
Applying $G'=\tnR \pi_{(\bq_1,\dots,\bq_{\ell-1})!}d^{-1}$ (see \eqref{Long expression of convolution}) to the diagram, we obtain a diagram below. To be precise, we explicitly give the morphisms for distinguished triangles.
\begin{equation*}
 \begin{tikzcd}
P_U^{\star\ell } \arrow[r, "a^{\star\ell }"] \arrow[d, "\cong"] & {\bK_{\Delta_{X^2}\times [0,\infty)}} \arrow[r, "G'(c^\ell)"] \arrow[d, "="] & G'(\mathcal{C}_\ell) \arrow[d, dashed] \arrow[r, "e"] & P_U^{\star\ell }[1] \arrow[d, "\cong"] \\
P_U \arrow[r, "a"]                                              & {\bK_{\Delta_{X^2}\times [0,\infty)}} \arrow[r, "b"]                        & Q_U \arrow[r, "d"]                 & P_U[1]
\end{tikzcd}   
\end{equation*}
So TR3 shows that there exists an isomorphism $G'(\mathcal{C}_\ell)\rightarrow Q_U$ which fills the diagram into a commutative diagram. Then we take $\psi'$ to be the composition:
\[G'(\mathcal{C}_\ell)\rightarrow  Q_U \rightarrow Q_U^{\star \ell},\]
which is an isomorphism.
Now, we have the commutative diagram
\begin{equation*}
 \begin{tikzcd}
{\bK_{\Delta_{X^2}\times [0,\infty)}} \arrow[r, "G'(c^\ell)"] \arrow[d, "b"] & G'(\mathcal{C}_\ell) \arrow[d, "\psi'"] \arrow[ld] \\
Q_U \arrow[r]                                                               & Q_U^{\star\ell}.              \end{tikzcd}   
\end{equation*}
So we obtain one factorization $b^{\star \ell}=\psi'G'(c^\ell)$.

On the other hand, applying $G'$ to the factorization $b^{\boxstar\ell}=\psi c^\ell$, we obtain another factorization $b^{\star \ell}=G'(\psi) G'(c^\ell)$. Consequently, $(\psi'-G'(\psi))G'(c^\ell)=0$ determines an element $u\in \HOM(P_U^{\star \ell}[1],Q_U^{\star \ell})$ such that \[\psi'-G'(\psi)=ue.\]
However, we have
\[\HOM(P_U^{\star \ell}[1],Q_U^{\star \ell}) \cong \HOM(P_U[1],Q_U)\cong 0.\]
Then we have 
\[\psi'=G'(\psi).\]
Consequently, we have
\[\phi=F_1(\psi')=F_1(G'(\psi))=F_\ell'(\psi),\]
and $\phi=F_1(\psi')$ is an isomorphism since $\psi'$ is an isomorphism.\end{proof}

\subsection{Persistence module}\label{subsection Z/l persistence}A persistence module $M$ is a functor from $(\bR, \leq)$ to the category of $R$-modules. Alternatively, we can describe it as a family of $R$-modules $M_T$ and a family of morphisms $M(T\leq T'):M_T\rightarrow M_{T'}$ satisfying $M(T\leq T'')=M(T'\leq T'')\circ M(T\leq T')$ and the identity property $M(T= T)=\id_{M_T}$. For more information on persistence modules, we refer to \cite{2020topological_persistence}.

Here, we will study a family of morphisms $\{M(T,T')\}_{T'\geq T}$ (where we set $c=T'-T\geq 0$) in the equivariant derived category $D_{Z/\ell}(\pt)$:
\[M(T,T+c):C_{\ell,T}(U,\bK)\rightarrow C_{\ell,T+c}(U,\bK)\]
functorially with respect to $T$. Subsequently, we take the cohomology $H^q$ or $H^*$ to construct persistence modules over $R=\mathbb{K}$ or $R=\text{Ext}_{\bZ/\ell}^*(\bK,\bK)$.

We first focus on the case $\ell=1$. Since $P_U\in \cD(X^2)$, there exists $\tau_c(P_U): P_U\rightarrow \tnT_{c*} P_U$ for $c\geq 0$. Applying $\alpha_{1,X,T}\circ \tnT_{-c*}$, we have 
\[F_{1}(U,\bK)_{T+c} \cong \alpha_{1,X,T}(\tnT_{-c*}P_U)  \rightarrow F_{1}(U,\bK)_{T}.\]
Then, we have a family of morphisms for $c\geq0$,
\begin{equation}\label{persistent struc non-eq}
C_{T}(U,\bK)\rightarrow C_{T+c}(U,\bK).    
\end{equation}
It is functorial with respect $T$ since $\tau_c\circ \tau_d = \tau_{c+d}$ for $c,d\geq 0$.

For $\ell \geq 2$, we apply the same idea. Since $P_U\in \mathcal{D}(X^2)$, we have that $P_U^{\boxstar \ell}\in \cD_{\bZ/\ell}(X^{2\ell})$. Therefore, for all $c\geq 0$, we have $\tau_{c}(P_U^{\boxstar  \ell}):P_U^{\boxstar \ell}\rightarrow \tnT_{c*}P_U^{\boxstar \ell}$ in $\cD_{\bZ/\ell}(X^{2\ell})$ induced by $\tau_{c/\ell}(P_U): P_U\rightarrow \tnT_{c/\ell*} P_U$.
So, we obtain an equivariant morphism
\[F_{\ell}(U,\bK)_{T+c}=\alpha'_{\ell,X,T+c}(P_U ^{\boxstar \ell})\cong  \alpha'_{\ell,X,T}(\tnT_{-c*}P_U^{\boxstar \ell}) \rightarrow \alpha'_{\ell,X,T}(P_U^{\boxstar \ell})\cong F_{\ell}(U,\bK)_{T}.\]
Then we have a family of morphisms functorial with respect $T\geq 0$:
\begin{equation}
C_{\ell,T}(U,\bK)\rightarrow C_{\ell,T+c}(U,\bK).    
\end{equation}
The parameter $T$ indicates the action of Reeb orbits, and the translations $\tnT_{c*}$ quantize the Reeb flow. Consequently, we can define various Chiu-Tamarkin invariants for different action windows by employing different cut-offs of $T$.
\begin{Def}When $T=\infty$, we define \begin{equation}
    H^*C^{\bZ/\ell}_{\infty}(U,\bK)\coloneqq \varinjlim_{T\geq 0}H^*C^{\bZ/\ell}_{T}(U,\bK).
\end{equation}
When we have an action window $(T,T']$, we define
\begin{align}
\begin{aligned}
    C^{\bZ/\ell}_{(T,T']}(U,\bK)\coloneqq& \RHOM_{{\bZ/\ell}}\left(F_\ell(U,\bK),\bK_{[T,T')}[1-d]\right)\\
   \cong   & cone (C^{\bZ/\ell}_{T}(U,\bK)\rightarrow C^{\bZ/\ell}_{T'}(U,\bK)) & 
\end{aligned}
\end{align}
\end{Def}
Immediately, we have a distinguished triangle for $0\leq T<T'$, say:
\begin{equation}\label{Action triangle of Chiu-Tamarkin complex}
 C^{\bZ/\ell}_{T}(U,\bK)\rightarrow C^{\bZ/\ell}_{T'}(U,\bK) \rightarrow C^{\bZ/\ell}_{(T,T']}(U,\bK)\xrightarrow{+1}.    
\end{equation}
We refer to it as the {\em action exact triangle} of the Chiu-Tamarkin invariant. The discussions here also apply to positive versions of Chiu-Tamarkin invariants.
 
Let us compute an example when $U=T^*X$. 
Recall $P_{T^*X}=\bK_{\Delta_{X^2}\times [0,\infty)}$. Therefore, we have $\Tilde{\Delta}^{-1}\left(P_{T^*X}^{\boxstar \ell}\right)= \bK_{\Delta_{X^{\ell}}\times [0,\infty) }$, and
\[F_{\ell}(T^*X,\bK)= {\tnR }\pi_{\underline{\bq}!}(\bK_{\Delta_{X^{\ell}}\times [0,\infty) })=E_{[0,\infty)}, \]
where $E=R\Gamma_c(\Delta_{X^{\ell}},\bK)$ and ${\bZ/\ell}$ acts on $E={\tnR }\Gamma_c(\Delta_{X^{\ell}},\bK)\cong {\tnR }\Gamma_c(X,\bK)$ trivially. 

Since ${\bZ/\ell}$ acts on $E$ trivially, we have, for $T\geq 0$,
\begin{equation}\label{F of the cotangent bundle}
C_{T}^{\bZ/\ell}(T^*X,\bK)\cong  \RHOM_{\bZ/\ell}(\bK,\bK)\dotimes {\tnR }\Gamma(X,or_{X}). 
\end{equation}Since $Q_{T^*X}\cong 0$, we have $F^{out}_{\ell}(T^*X,\bK)=0$ and $C_{T}^{{\bZ/\ell},out}(T^*X,\bK)=0$. Also, for $0\leq T<T' $, we have $C_{(T,T']}^{{\bZ/\ell},out}(T^*X,\bK)=0$.

\subsection{Product}\label{sec: cup product}
In this subsection, we assume that $X$ is orientable. As we proved in \autoref{non-equivariant CT and yoneda form}, we have that
\[C_{T}(U,\bK)\cong \RHOM(P_{U},\tnT_{T*}(P_{U})).\]
Using this formula, we observe that the non-equivariant Chiu-Tamarkin invariant is nothing but the endomorphism of $P_U$ in the triangulated persistent category $\cD(X^2)$. Notably, there exists a shifted Yoneda product on $\RHOM(P_{U},\tnT_{T*}(P_{U}))$, which we will further elaborate on in this subsection.
\begin{Def}For $\alpha \in\textnormal{Ext}^a(P_U,\tnT_{A*}P_U) $, $\beta \in\textnormal{Ext}^b(P_U,\tnT_{B*}P_U)$, we define the shifted Yoneda product $\alpha \bullet \beta\in \textnormal{Ext}^{a+b}(P_U,\tnT_{(A+B)*}P_U)$ to be the composition:
\[ P_U\xrightarrow{\beta} {\tnT_{B*}P_U} \xrightarrow{\tnT_{A*}\alpha} \tnT_{(A+B)*}P_U. \]
\end{Def}
The shifted Yoneda product is a shifted version of the usual Yoneda product on $\textnormal{Ext}^a(P_U,P_U)$. This product has also been discussed in \cite[Lemma 6.4.4]{TPC}. We will show that it is unital and associative. Moreover, since $P_U$ is a bialgebra in the symmetric monoidal category $(\mathcal{D}(X^2),\star)$ by the idempotence (\autoref{idempotence}), we have:
\begin{Thm}\label{theorem: graded commutative}The shifted Yoneda product is unital, associative, and graded commutative.
\end{Thm}
\begin{proof}The identity $\textnormal{Id}_{P_U}$ is the unit of the shifted Yoneda product. The associativity follows from the associativity of the usual Yoneda product and the functor isomorphism $\tnT_{(A+B+C)*} \cong \tnT_{A*}\circ \tnT_{B*}\circ \tnT_{C*}$. Let us prove the graded commutativity. 

We use \autoref{idempotence} for $\ell=2$, and denote the isomorphism by $\nu: P_U\star P_U \xrightarrow{\cong} P_U$. Now, we can define another product using $\nu$. 

The convolution $\star$ is a bifunctor, so we have 
\[\alpha \star \beta : P_U\star P_U \rightarrow \tnT_{A*}P_U \star \tnT_{B*}P_U\cong \tnT_{(A+B)*}(P_U\star P_U) \]
for $\alpha \in\textnormal{Ext}^a(P_U,\tnT_{A*}P_U) $, $\beta \in\textnormal{Ext}^b(P_U,\tnT_{B*}P_U)$. We define
\[\alpha \underline{\star} \beta =\tnT_{(A+B)*}( \nu )\circ (\alpha \star \beta) \circ \nu^{-1}. \] 
We will apply the Eckmann-Hilton argument below to the shifted Yoneda product and the convolution to prove that
\[\text{(1): }\alpha \underline{\star} \beta=\alpha \bullet \beta,\qquad\text{(2): }\alpha \bullet \beta \text{ is graded commutative}.\]
To apply the Eckmann-Hilton argument we first notice the following identities:
\begin{itemize}
    \item $\textnormal{Id}_{P_U}\underline{\star} \alpha =\alpha \underline{\star} \textnormal{Id}_{P_U}=\alpha$,
    \item For $\alpha_i \in \textnormal{Ext}^{a_i}(P_U,\tnT_{A_i*}P_U)$, $i=1,2,3,4$, we have
    \[(\alpha_1\bullet \alpha_2)\underline{\star} (\alpha_3\bullet \alpha_4)=(-1)^{a_2a_3}(\alpha_1\underline{\star} \alpha_3)\bullet (\alpha_2\underline{\star} \alpha_4).\]
\end{itemize}
To prove them, let us notice that $\star$ is a bifunctor, and that the graded shifting $[1]$ defines isomorphisms $\textnormal{Ext}^*(F,G[1])\cong \textnormal{Ext}^*(F,G)[1]\cong \textnormal{Ext}^*(F[-1],G)$ such that the following diagram is anti-commutative for all $F,G$:
\begin{equation*}
  \begin{tikzcd}
{\textnormal{Ext}^*(F[1],G[1])} \arrow[d] \arrow[r] & {\textnormal{Ext}^*(F,G[1])[-1]} \arrow[d] \\
{\textnormal{Ext}^*(F[1],G)[1]} \arrow[r]     & {\textnormal{Ext}^*(F,G)}.           
\end{tikzcd}  
\end{equation*}
Next, the Eckmann-Hilton argument uses the two above identities to conclude as follows:
\begin{gather*}
  \alpha\underline{\star} \beta = (\alpha\bullet\textnormal{Id}_{P_U})\underline{\star}(\textnormal{Id}_{P_U}\bullet \beta)=(\alpha\underline{\star}\textnormal{Id}_{P_U})\bullet(\textnormal{Id}_{P_U}\underline{\star} \beta)=\alpha\bullet \beta\\
  \alpha\bullet \beta =\alpha\underline{\star} \beta = (\textnormal{Id}_{P_U}\bullet\alpha)\underline{\star}(\beta\bullet\textnormal{Id}_{P_U} )=(-1)^{ab}(\textnormal{Id}_{P_U}\underline{\star} \beta)\bullet(\alpha \underline{\star} \textnormal{Id}_{P_U})=(-1)^{ab}\beta\bullet \alpha.\qedhere
\end{gather*}
\end{proof}
Now, since $X$ is orientable, we set the cohomology of the isomorphism in \autoref{non-equivariant CT and yoneda form} by:
\begin{align}\label{algebra isomorphisms formula}
\begin{aligned}
    \Theta:\textnormal{Ext}^*(P_U,\tnT_{T*}P_U) &\xrightarrow{\cong} \textnormal{Ext}^*(P_U,\bK_{\Delta_{X^2}\times \{T\}}) {\cong} H^*C_{T}(U,\bK)\\
    [ P_U\xrightarrow{\alpha} \tnT_{T*}P_U] &\mapsto \widetilde{\alpha}=[ P_U\xrightarrow{\alpha} \tnT_{T*}P_U \rightarrow \bK_{\Delta_{X^2}\times \{T\}} ]\mapsto N(\widetilde{\alpha}).
\end{aligned}
\end{align}

We would like to express $\widetilde{\alpha\bullet\beta}=\widetilde{\alpha\underline{\star}\beta}$ in a different form in the non-equivariant Chiu-Tamarkin invariant.

\begin{Def}\label{definition: cup product}
For ${\alpha}\in H^aC_{A}(U,\bK)$ and ${\beta} \in H^bC_{B}(U,\bK)$, we define the cup product $\alpha \cup \beta$ in such a way that the isomorphism of $\bK$-modules $\Theta$ in \eqref{algebra isomorphisms formula} becomes an isomorphism of $\bK$-algebras.
\end{Def}
The cup product we define here is closely related to the usual cup product on the cohomology ring. In fact, we have $H^*C_{T}(U,\bK)\cong \textnormal{Ext}^*(\bK_X,\bK_X) \cong H^*(X,\bK)$ for $T\geq 0$. Furthermore, the shifted Yoneda product is given by the Yoneda product on $\textnormal{Ext}^*(\bK_X,\bK_X)$, which in turn corresponds to the cup product on $H^*(X,\bK)$. Therefore, it is appropriate to refer to our defined product as a cup product.

On the other hand, the usual cup product on $H^*(X,\bK)$ can be defined by the diagonal pullback of the cross product of classes. A similar idea can also be applied here, and we will provide an alternative way to describe the cup product.

For classes $\alpha, \beta\in \textnormal{Ext}^*(P_U,\bK_{\Delta_{X^2}\times \{T\}})$, consider the external tensor
\[\alpha\dboxtimes \beta : P_U^{\dboxtimes 2}\rightarrow \bK_{\Delta_{X^2}\times \{A\} \times \Delta_{X^2}\times \{B\}}[a+b].\]
We apply the functor $\Tilde{\Delta}_X^{-1}$, where $\Tilde{\Delta}_X: X^2 \times \bR^2 \rightarrow X^4\times \bR^2$ is the twisted diagonal for $\ell=2$, to obtain
\[\Tilde{\Delta}_X^{-1}(\alpha\dboxtimes \beta) : \Tilde{\Delta}_X^{-1}(P_U^{\dboxtimes 2})\rightarrow \bK_{\Delta_{X^2}\times \{(A,B)\} }[a+b].\]
Next, we apply the Gysin morphism associated with the pair $(X^2,\Delta_{X^2})$. In fact, the Gysin morphism is given by composition with the Euler class, which is given by the following isomorphisms:
\begin{align*}
      \textnormal{Ext}^d(\bK_{\Delta_{X^2}\times  \{(A,B)\}}, \bK_{X^2\times \{(A,B)\}})
\cong \textnormal{Ext}^d(\bK_{\Delta_{X^2}} , \bK_{X^2})\cong  {H}^d_{\Delta_{X^2}}({X^2},\bK)
\cong  H^0(X,\bK).
\end{align*}
The last isomorphism comes from the excision and the Thom isomorphism. Then $1\in H^0(X,\bK)$ corresponds to the Euler class $e:\bK_{\Delta_{X^2} \times \{(A,B)\}}\rightarrow \bK_{X^2\times \{(A,B)\}}[d]$. We compose the Euler class to obtain 
\[e\circ \Tilde{\Delta}_X^{-1}(\alpha\dboxtimes \beta) \in \textnormal{Ext}^{a+b+d}(\Tilde{\Delta}_X^{-1}(P_U^{\dboxtimes 2}),\bK_{X^2\times \{(A,B)\}}).\]
Then we apply $\tnR s^2_{t!}$ to $e\circ \Tilde{\Delta}_X^{-1}(\alpha\dboxtimes \beta)$ to obtain 
\[\tnR s^2_{t!}(e\circ \Tilde{\Delta}_X^{-1}(\alpha\dboxtimes \beta))\in \textnormal{Ext}^{a+b+d}(\tnR s^2_{t!}\Tilde{\Delta}_X^{-1}P_U^{\dboxtimes 2},\bK_{X^2\times \{A+B\}}).\]
On the other hand, since $X^2$ is orientable, we have $\bK_{X^2\times \{A+B\}}\cong \pi_{X^2}^!\bK_{\{A+B\}}[-2d] $, so \autoref{cyclicstructure step1} shows that
\[ \textnormal{Ext}^{a+b+d}(\tnR s^2_{t!}\Tilde{\Delta}_X^{-1}P_U^{\dboxtimes 2},\bK_{X^2\times \{A+B\}})\cong \textnormal{Ext}^{a+b}(F_2(U,\bK)_{A+B},\bK[-d])\cong H^{a+b}C_{A+B}(U,\bK) .\]
Then we have 
\[\tnR s^2_{t!}(e\circ \Tilde{\Delta}_X^{-1}(\alpha\dboxtimes \beta))\in H^{a+b}C_{A+B}(U,\bK),\]
and the process verifies that 
\begin{equation}\label{formula: cup product}
    \tnR s^2_{t!}(e\circ \Tilde{\Delta}_X^{-1}(\alpha\dboxtimes \beta))=\alpha\cup \beta.
\end{equation}

\section{\texorpdfstring{$S^1$-equivariant Chiu-Tamarkin invariant}{}}\label{section: s1-theory}
So far, we have defined the $\bZ/\ell$-equivariant Chiu-Tamarkin invariant for all $\ell$. It is natural to know in what sense we have an $S^1$-equivariant theory. In \cite{Capacities2021}, we construct a sequence of symplectic capacities using a certain idea of ``numerical approximation" of $S^1$ through all the $\bZ/\ell$-equivariant Chiu-Tamarkin invariants. However, we will now explain how to construct a genuine algebraic $S^1$-equivariant theory. To achieve this, we will utilize the structure maps of microlocal kernels to establish an algebraic $S^1$ action, i.e. a mixed complex \cite{kassel1987cyclic}. Subsequently, we define an $S^1$-equivariant Chiu-Tamarkin invariant.

Let us begin with a discussion of the motivations behind this construction. For a cyclic object in a category $\mathcal{C}$, we mean a simplicial object in $\mathcal{C}$ together with cyclic permutations on $n$-simplex, subject to certain compatibility conditions. For pre-simplicial/cyclic objects, also known as semi-simplicial/cyclic objects, we simply drop degeneracy maps from the definition of simplicial/cyclic objects. For further details, we refer to \cite{Jones1987} and \cite[Chapter 6]{LodayCyclic}. We set $[n]=\{1,\dots,n\}$ for $n\in \bN$ and $[n]_0=\{0,1,\dots,n\}$ for $n\in \bN_{0}$.

The proof of \autoref{cyclicstructure step1} allows us to define a pre-cyclic object in $D(\bR)$ using $F_n(U,\bK)$ for all $n\in \bN$. Precisely, the morphism \[P_U^{\dboxtimes n}\rightarrow \bK_{\Delta_{X^2}\times [0,\infty)} \dboxtimes P_U^{\dboxtimes n -1},\]
together with cyclic permutations, induce morphisms
\begin{align*}
    d_i:& F_{n+1}(U,\bK) \rightarrow F_{n}(U,\bK),\\
    t_n:&F_{n+1}(U,\bK) \rightarrow F_{n+1}(U,\bK),
\end{align*}
for $i\in [n]$. One can verify that $(F_{n+1}(U,\bK),d_i,t_n)_{n\in \bN_0}$ forms a pre-cyclic object in $D(\bR)$. Similarly, we can define a pre-cocyclic object of $D(\bR)$ using $Q_U$. Taking the stalk at $T$, we will obtain a pre-(co)cyclic object of $D(\bK-\Mod)$. 

However, a pre-(co)cyclic object is not enough to define an algebraic $S^1$-action since commutative diagrams in the triangulated derived category do not necessarily come from a homotopy coherent diagram. So, our strategy is lifting microlocal kernels and structure maps to chain level. 

We will proceed in the following: In \autoref{section: cyclic structure}, we will functorially construct a pre-cocyclic object from a chain map $\bK_{\Delta_{X^2}\times \{0\}}\rightarrow Q$. In \autoref{subsection: s1-theory}, we apply the functor to microlocal kernels. To do so, it is necessary to use the kernel $Q_U$ because of the direction of the arrows. Therefore, we first define the positive $S^1$-equivariant Chiu-Tamarkin invariant $F^{S^1,out}_\bullet(U,\bK)_T$. Subsequently, motivated by the tautological triangle \eqref{Tautological exact triangle}, we define the $S^1$-equivariant Chiu-Tamarkin invariant $F^{S^1}_\bullet(U,\bK)_T$ by taking cocone for a certain morphism. 

After introducing the definition, we will explain some constructions and prove some properties in the remaining part of this section.

\subsection{\texorpdfstring{Cyclic complex associated with $\bK_{\Delta_{X^2}\times \{0\}}\rightarrow Q$}{}}\label{section: cyclic structure} 
Let us start with a discussion on chain-level categories. In the following, we require some basic concepts of model categories. We refer to \cite[Chapter 2]{Cisinski-homotopical-algebra} for a quick introduction.

Consider the Grothendieck abelian category $C(X^2\times \bR)$ of complexes of sheaves on $X^2\times \bR$. We can equip it with an injective model structure such that fibrant objects of the model structure are K-injective and component-wise injective chain complexes of sheaves over $X^2\times \mathbb{R}$, and all objects are cofibrant. Now, for $\bK_{\Delta_{X^2}\times \{0\}}$, which is treated as a chain complex of sheaves concentrated at degree $0$, we consider the slice category $ {^{\triangleright}C(X^2\times\bR)}\coloneqq \bK_{\Delta_{X^2}\times \{0\}} / C(X^2\times\bR)$. It is known that ${^{\triangleright}C(X^2\times \mathbb{R})}$ also admits a model structure, in which a chain map $\bK_{\Delta_{X^2}\times \{0\}}\rightarrow Q$ is fibrant if and only if $Q$ is fibrant in $C(X^2\times\bR)$ \cite[Proposition 2.2.6]{Cisinski-homotopical-algebra}.
\begin{RMK}Noticed that here the direction of arrows for the slice category is crucial because we ONLY have an injective model structure on $C(X^2\times\bR)$. It forces us to consider the slice category $\bK_{\Delta_{X^2}\times \{0\}} / C(X^2\times\bR)$, nor the slice category $C(X^2\times\bR)/\bK_{\Delta_{X^2}\times \{0\}} $, since we do not have meaningful model structure for the later. So far, we still do not know how to figure out a chain-level construction if we start from $C(X^2\times\bR)/\bK_{\Delta_{X^2}\times \{0\}}$.\end{RMK}
Now, let us recall some morphisms and define something new: For $n\in \bN$, we set
\begin{align*}
   \pi_{n}=\pi_{\underline{\bq}}&: X^n \times \bR\rightarrow \bR,\\
    \Tilde{\Delta}_n=\Tilde{\Delta}_X&: X^n \times \bR\rightarrow (X^{2})^n \times \bR,\\
    \Tilde{\Delta}_n(\bq_1,\dots,\bq_n,t)&=(\bq_n,\bq_1,\bq_1,\dots,\bq_{n-1},\bq_{n-1},\bq_n,t).
\end{align*}
Here, we will only emphasize the index $n$ as usually $X$ is fixed. Points in the $i$-th copy of $X^2$ are denoted as $(\bq_i^1,\bq_i^2)$, and the definition of $\widetilde{\Delta}_n$ implies that $\bq_i^2=\bq_{i+1}^1=\bq_i$.

Besides, we define the partial diagonals $\delta_i:X^n\times\bR\rightarrow X^{n+1}\times\bR$, 
\begin{align*}
\delta_i(\bq_1,\dots,\bq_n,t)=&(\bq_1,\dots,\bq_{i+1},\bq_{i+1},\dots,\bq_n,t), \, i\in [n-1]_0,\\
   \delta_n(\bq_1,\dots,\bq_n,t)=&(\bq_1,\bq_2,\dots,\bq_{n},\bq_1,t), 
\end{align*}
and the cyclic permutation
\[\tau_n:X^{n+1}\times\bR\rightarrow X^{n+1}\times\bR,\, (\bq_1,\bq_2,\dots,\bq_n,\bq_{n+1},t)\mapsto (\bq_2,\bq_3,\dots,\bq_{n+1},\bq_1,t).\]
Then one has $\delta_j\delta_i=\delta_i\delta_{j-1}$ for $i<j$, $\tau_n\delta_i=\delta_{i-1}\tau_{n-1}$ for $i\in [n]$ and $\tau_n\delta_0=\delta_{n}$. So $(X^{n+1}\times \bR,\delta_i,\tau_n)$ form a pre-cocyclic space, i.e. a pre-cocyclic object in the category of topological spaces \cite[Example 1.2]{Jones1987}.

\begin{Lemma}\label{lemma 3.19}For a field $\bK$, $n\in \bN$ and $Q \in C(X^2\times \bR)$, with the above notation, we have isomorphisms of complexes of sheaves over $\bR$:
\[\pi_{n !}(\Tilde{\Delta}_{n}^{-1}s_{t!}^n(Q^{\boxtimes n}))   \cong \pi_{(n+1)!}(\Tilde{\Delta}_{n+1}^{-1}(s_{t!}^{n+1}([i]Q^{\boxtimes n+1}) )) ,\]
where $[i]Q^{\boxtimes n+1}=Q^{\boxtimes i+1}\boxtimes \bK_{\Delta_{X^2}\times \{0\}} \boxtimes Q^{\boxtimes n-i-1}$ for $i\in [n-1]_0$ and $[n]Q^{\boxtimes n+1}=\bK_{\Delta_{X^2}\times \{0\}} \boxtimes Q^{\boxtimes n}$. \end{Lemma}
\begin{proof}At the beginning, we assume $i=n$. Recall that we are working over a field $\bK$, so there is no need to derive the tensor product. Moreover, the projection formula holds in the non-derived sense over fields. Additionally, the proper base change theorem is always true in the non-derived sense. Therefore, we can utilize the Künneth isomorphism in the non-derived sense over fields.

Take $t'=t_0$, $t''=t_1+\cdots+t_{n}$, and $t=t'+t''$. Then we have a decomposition of $s_t^{n+1}=s_t^2\circ (\id_{t_0}\times  s_t^n)$. In this way, the Künneth isomorphism and the projection formula show that
\[s_{t!}^{n+1}(\bK_{\Delta_{X^2}\times \{0\}} \boxtimes Q^{\boxtimes n})\cong s_{t!}^2(\bK_{\Delta_{X^2}\times \{0\}}\boxtimes s_{t!}^n Q^{\boxtimes n} )\cong \bK_{\Delta_{X^2}}\boxtimes s_{t!}^n Q^{\boxtimes n} .\]
Now, we have
\[\Tilde{\Delta}_{n+1}^{-1}( s_{t!}^{n+1}([0]Q^{\boxtimes n+1}) )\cong \Tilde{\Delta}_{n+1}^{-1}(\bK_{\Delta_{X^2}}\boxtimes s_{t!}^n Q^{\boxtimes n}) \cong ((d')^{-1}\bK_{\Delta_{X^2}})\otimes (d^{-1}s_{t!}^{n}Q^{\boxtimes n}),\]
where $d(\bq_1,\bq_2,\dots,\bq_{n+1},t)=(\bq_1,\bq_2,\bq_2,\dots,\bq_{n},\bq_{n},\bq_{n+1},t)$ and $d'(\bq_1,\bq_2,\dots,\bq_{n+1},t)=(\bq_{n+1},\bq_1)$.

Notice that the isomorphism
$(d')^{-1}\bK_{\Delta_{X^2}}=\delta_{n!}\bK_{X^n\times \bR}$ can be thought of as an isomorphism of complexes. Then the projection formula for $\delta_n$ implies the isomorphism of complexes
\begin{align*}     \delta_{n!}\bK_{X^n\times \bR} \otimes(d^{-1}s_{t!}^{n}Q^{\boxtimes n})
\cong \delta_{n!}( (d\delta_n)^{-1}s_{t!}^{n}Q^{\boxtimes n} )
\cong \delta_{n!}( \Tilde{\Delta}_{n}^{-1}s_{t!}^{n}Q^{\boxtimes n} ).
\end{align*}
Therefore, we conclude by $\pi_{(n+1)!}=\pi_{n!}\delta_{n!}$ in the non-derived sense.

For $i\in [n-1]_0$, we can apply $\tau_n$ several times and use the relations $\tau_n\delta_i=\delta_{i-1}\tau_{n-1}$ for $i\in [n]$, and $\tau_n\delta_0=\delta_{n}$. Then we conclude using the isomorphisms $\tau_n^{-1}\Tilde{\Delta}_{n}^{-1}s_{t!}^{n}Q^{\boxtimes n} \cong \Tilde{\Delta}_{n}^{-1}s_{t!}^{n}Q^{\boxtimes n} $, and $\pi_{(n+1)!}=\pi_{n+1!}\tau_{n!}$\end{proof}

To simplify the notation, we set $F_{n}^{nd}(Q)=\pi_{n !}(\Tilde{\Delta}_{n}^{-1}s_{t!}^n(Q^{\boxtimes n}))\in C(\bR)$, where the superscript {\em nd} indicates that the chain complex is not derived.

Now, we consider a further situation. We take an object in ${^{\triangleright}C(X^2\times\bR)}$, i.e., a chain map: $\bK_{\Delta_{X^2}\times \{0\}}\rightarrow Q$. Then for each $i\in [n]$, we can define a morphism using \autoref{lemma 3.19}:
\begin{align*}
 d_i:&F_{n}^{nd}(Q)\cong \pi_{(n+1)!}(\Tilde{\Delta}_{n+1}^{-1}(s_{t!}^{n+1}([i]Q^{\boxtimes n+1}) ))  \\ &\rightarrow \pi_{(n+1)!}(\Tilde{\Delta}_{n+1}^{-1}(s_{t!}^{n+1}(Q^{\boxtimes n+1}) ))=F_{n+1}^{nd}(Q). \end{align*}
We also set the following automorphism $t_n$ of $F_{n}^{nd}(Q)$:
\[t_n=(-1)^n \pi_{(n+1)!}(\Tilde{\Delta}_{n+1}^{-1}(\tau_n^{-1}))   :F_{n+1}^{nd}(Q)\rightarrow F_{n+1}^{nd}(Q).\]
One can check the following identities
\begin{align*}
  &d_jd_i=d_id_{j-1},\quad i<j,\\
  &t_n d_i=-d_{i-1}t_{n-1},\quad i\in [n],\\
  &t_n d_0=(-1)^n d_{n}.
\end{align*}
Consequently, when $\bK$ is a field, the data
\begin{equation}
 F_\bullet^{nd}(\bK_{\Delta_{X^2}\times \{0\}}\rightarrow Q)  \coloneqq (   F_{n+1}^{nd}(Q),d_i,t_n)_{n\in \bN_0} 
\end{equation}
form a pre-cocyclic complex of sheaves over $\bR$. If we take stalk at $T\geq 0$, we get a pre-cocyclic complex of $\bK$-vector spaces:
\begin{equation}
 F_{\bullet,T}^{nd}(\bK_{\Delta_{X^2}\times \{0\}}\rightarrow Q)  \coloneqq (   F_{n+1}^{nd}(Q)_T,d_i,t_n)_{n\in \bN_0} .
\end{equation}
Moreover, if we denote the dg category of pre-cocyclic complexes of $\bK$-vector spaces by $\Delta C_{0,dg}(\bK-\Mod)$, we actually define a functor
\[ F_{\bullet,T}^{nd}: {^{\triangleright}C(X^2\times \bR) }\rightarrow \Delta C_{0,dg}(\bK-\Mod).\]
As $\Delta C_{0,dg}(\bK-\Mod)$ is dg, we can compose the functor with the dg-quotient functor:
\[{^{\triangleright}C(X^2\times\bR)}\xrightarrow{F_{\bullet,T}^{nd}} \Delta C_{0,dg}( \bK-\Mod)  \xrightarrow{Q}  D_{dg}(\Delta C_{0,dg}(\bK-\Mod))). \]
However, one can check that weak equivalences in ${^{\triangleright}C(X^2\times\bR)}$ induce isomorphisms in $D_{dg}(\Delta C_0( \bK-\Mod))$. Then one can define the derived functor $F_{\bullet,T}$ of $F_{\bullet,T}^{nd}$ as the left Kan extension of $Q\circ F_{\bullet,T}^{nd}$ along the localization functor $\gamma: {^{\triangleright}C(X^2\times\bR)}\rightarrow ho({^{\triangleright}C(X^2\times\bR)})$, where $ho({^{\triangleright}C(X^2\times\bR)})$ is the homotopy category of the model category ${^{\triangleright}C(X^2\times\bR)}$. In particular, we have the following commutative diagram:
\begin{equation*}
 \begin{tikzcd}
{^{\triangleright}C(X^2\times\bR)} \arrow[rr, "F_{\bullet,T}^{nd}"] \arrow[d, "\gamma"] \arrow[rrd] &  & \Delta C_{0,dg}( \bK-\Mod)  \arrow[d, "Q"] \\
ho({^{\triangleright}C(X^2\times\bR)}) \arrow[rr, "F_{\bullet,T}"]                                     &  & D_{dg}(\Delta C_{0,dg}(\bK-\Mod)).             
\end{tikzcd}   
\end{equation*}
Notice that the functors $F^{nd}_{\bullet,T}$ and $F_{\bullet,T}$ are additive functors such that a morphism on the left side is mapped to a degree $0$ closed morphism.
\begin{RMK}
Now, the functor $F_{\bullet,T}^{nd}$ is defined over fields to ensure that we can use the projection formula in the non-derived sense. Therefore, our derived functor $F_{\bullet,T}$ is also defined over fields. However, notice that the chain-level construction still works if $Q$ is flat, and we can define the Kan extension locally. Therefore, it is possible to define $F_{\bullet,T}^{nd}$ for special situations directly.
\end{RMK}

\subsection{Mixed complexes}\label{subsection: mixed complexes}
In the previous subsection, we discussed pre-cocyclic complexes. In this subsection, let us discuss mixed complexes. Here, we take $R$ to be a commutative ring without assuming that $R$ is a field.

Consider the dg-algebra $R[\epsilon]$, where $|\epsilon|=-1$ and $\epsilon^2=0$. A dg-module over $R[\epsilon]$, also called a mixed complex \cite{kassel1987cyclic}, is defined as a triple $(M,b,B)$, where $(M,b)$ is a cochain complex, $|b|=-|B|=1$, and $b^2=B^2=Bb+bB=0$. The derived category of dg-modules over $R[\epsilon]$ is called the mixed derived category. In our subsequent applications, we will make use of the dg derived category for dg algebras.

On the other hand, consider the dg-algebra $R[u]$, where $|u|=2$ and $d=0$. We have the following Koszul duality (see \cite{goresky-kottwitz-macpherson-koszulduality}):
\begin{equation}\label{Koszul duality of mixed complexes}
    D_{dg}(R[\epsilon]-\Mod)\cong D_{dg}(R[u]-\Mod),\quad (M,b,B)\mapsto M^{hS^1}=(M\fpw{u}, \delta=b+uB),
\end{equation}
where $M\fpw{u}$ denotes the $(u)$-adic completion of $M\otimes R[u]$. Using the following free resolution of $R$:
\[\cdots\xrightarrow{\epsilon} R[\epsilon] \xrightarrow{\epsilon}\cdots\xrightarrow{\epsilon} R[\epsilon] \xrightarrow{\epsilon} R[\epsilon] \xrightarrow{\epsilon=0}  R,\]
we find that $\RHOM_{ R[\epsilon]}( R,M)\cong M^{hS^1}=(M\fpw{u}, \delta=b+uB)$ represents the homotopy fixed points of the mixed complex $(M,b,B)$.

The functor $(M,b,B)\mapsto (M,b)$ defines a forgetful functor
\begin{equation}\label{equation: forgetful functor}
    D_{dg}(R[\epsilon]-\Mod) \rightarrow 
D_{dg}(R-\Mod).
\end{equation}
The forgetful functor is conservative in virtue of \cite[Corollary 2.4]{zhaoperiodic}.
\begin{RMK}In the literature, $A_\infty$-modules over $R[\epsilon]$ are referred to as $S^1$-complexes, multicomplexes, or $\infty$-mixed complexes (see \cite{DSV,BO2016,zhaoperiodic,GanatraS1CY} and references therein). The mixed complex, as defined earlier, is known as the strict $S^1$-complex. In terms of module theory over $R[\epsilon]$, we primarily follow the exposition in \cite{GanatraS1CY}, although the author's focus is on $A_\infty$-modules. For additional variants involving the completion $R\fpw{u}$, we refer to \cite{Eightflavors}.\end{RMK}

We will explain how to construct a mixed complex from a pre-cocyclic chain complex. This idea was first discovered by Tsygan in \cite{Tsygan-homology-Lie-algebra} and independently by Loday-Quillen in \cite{Loday-Quillen-cyclic}.

Take a pre-cocyclic complex $N=(N_p,d_i,t_p,\partial_p)$ in $\Delta C_{0,dg}( R-\Mod)$, where $d_i$ are face maps, $t_p$ are cyclic permutations, and $\partial_p$ is the differential of the complex $N_p$. First, we define 
\[d=\sum_{i=0}^n(-1)^id_i,\quad d'=\sum_{i=0}^{n-1}(-1)^id_i.\]
They satisfies $d^2=(d')^2=0$, and $d\partial=\partial d,\,d'\partial=\partial d'$. So $(N_p^q,d,(-1)^{p+q+1}\partial)$ and $(N_p^q,d',(-1)^{p+q+1}\partial)$ are two bicomplexes. Next, consider the product totalization of these two bicomplexes. Specifically, we set $Tot(N)_n=\prod_{q+p=n}N_p^q$ and the differential $b=d+(-1)^{n+1}\partial$. Then the totalization is $(Tot(N),b=d+(-1)^{n+1}\partial)$. Similarly, we have $(Tot(N),b'=d'+(-1)^{n+1}\partial)$.

On the other hand, the cyclic permutation satisfies the relations $(1-t)d=d'(1-t)$. Therefore, we have a chain map:
\[(Tot(N),b=d+(-1)^{n+1}\partial)\rightarrow (Tot(N),b'=d'+(-1)^{n+1}\partial).\]
The mapping cone of the morphism is $Cone(N)=(Tot(N)\oplus Tot(N)[-1], \overline{b})$. We also define a degree $-1$ morphism $\overline{B}:Cone(N)\rightarrow Cone(N)$. Precisely, we have
\[\overline{b}=\begin{pmatrix}  b & 0 \\  1-t & -b'\end{pmatrix},\quad \overline{B}=\begin{pmatrix}  0 & B \\  0 & 0\end{pmatrix}, \]
where $B=1+t+\cdots+t^n$. 
It is direct to see that $\overline{B}^2=\overline{B}\,\overline{b}+\overline{b}\,\overline{B}=0$. So, we have that
\[CC^*(N)\coloneqq (Cone(N),\overline{b},\overline{B})\] form a mixed complex.

If $N$ comes from a cocyclic complex $C=(N_p,d_i,s_i,t,\partial)$ (satisfies some other relations about $s_i$), then we have $b's+sb'=\id$ where $s=s_{n}$, and the following morphism is a homotopy equivalence of mixed complexes
\[\begin{pmatrix}  1 & s(1-t) \end{pmatrix}^t:(Tot(N),b,B)\rightarrow (Cone(N),\overline{b},\overline{B}),\]
where $B=Ns(1-t)$. 

The mixed complex $CC^*(N)= (Cone(N),\overline{b},\overline{B})$, which is called the cyclic cochain of $N$, is functorial with respect to $N$. It is known that the functor descent to derived categories, yielding the following functor:
\begin{equation}
CC: D_{dg}(\Delta C_{0,dg}(R-\Mod)) \rightarrow D_{dg}( R[\epsilon]-\Mod).
\end{equation}
\begin{RMK}\label{remark: simplicial cyclic chain}
   We also remark that the construction also works for pre-cyclic cochains and (homological convention) chain complexes. We will use the same notation without explicit emphasis.
\end{RMK}

\subsection{\texorpdfstring{$S^1$-equivariant Chiu-Tamarkin invariant}{}}\label{subsection: s1-theory}
Now, let us discuss microlocal kernels. For an open set $U$, we have a morphism $\bK_{\Delta_{X^2}\times [0,\infty)}\rightarrow Q_U$ in $\cD(X^2)$. We take a fibrant resolution of $Q_U$, say $Q_U\rightarrow IQ_{U}$. Then $IQ_{U}$ is K-injective and component-wise injective. In particular, $IQ_{U}$ is a complex of c-soft sheaves over $X^2\times\bR$. To our later discussion, we also take a flat and c-soft resolution of $\bK_{X}$, say $\bK_{X}\rightarrow I$ (for example, the Godement resolution), where $I\in C^{\geq 0}(X)$. Then for $\Delta=\Delta_{X^2}$, we have $\bK_{\Delta_{X^2}\times [0,\infty)}=\Delta_{!}\bK_X\boxtimes \bK_{[0,\infty)}\rightarrow \Delta_{!}I\boxtimes \bK_{[0,\infty)}$ is a flat and c-soft (relative to $X^2$) resolution of $\bK_{\Delta_{X^2}\times [0,\infty)}$. 

Since $Q_U\rightarrow IQ_{U}$ is a fibrant resolution, we have the following isomorphisms of functors,
\begin{equation}\label{Q functors isomorphisms}\HOM_{D}(-,Q_U)\cong 
\HOM_{D}(-,IQ_{U}) \cong \HOM_{K}(-,IQ_{U}),    
\end{equation}
where $D=D(X^2\times \bR)$ and $K=K(X^2\times \bR)$.

Now, we apply the isomorphisms \eqref{Q functors isomorphisms} to the diagram 
\begin{equation*}
 \bK_{\Delta_{X^2}\times \{0\}} \leftarrow {\bK_{\Delta_{X^2}\times [0,\infty)}} \xrightarrow{qis} \Delta_{!}I\boxtimes \bK_{[0,\infty)}.
\end{equation*}
The resulting diagram is a commutative diagram of isomorphisms since $IQ_{U}\cong Q_U\in D(X^2 \times \bR)$. Consequently, we have
\begin{equation}\label{chain representation}
\HOM_{D}(\bK_{\Delta_{X^2}\times [0,\infty)},Q_U)\cong \HOM_{K}(\bK_{\Delta_{X^2}\times \{0\}},IQ_{U})\cong \HOM_{K}(\Delta_{!}I\boxtimes \bK_{[0,\infty)},IQ_{U}).    
\end{equation}
Then we can take two chain maps 
\begin{equation}\label{two chain maps}
  \bK_{\Delta_{X^2}\times \{0\}} \rightarrow IQ_{U} \leftarrow \Delta_{!}I\boxtimes \bK_{[0,\infty)},   
\end{equation}
representing $\bK_{\Delta_{X^2}\times [0,\infty)}\rightarrow Q_U$ in $D(X^2 \times \bR)$, using the isomorphisms \eqref{chain representation}.
In particular, we have the following commutative diagram of chain maps:
\begin{equation}\label{diagram of Q}
    \begin{tikzcd}
\bK_{\Delta_{X^2}\times \{0\}} \arrow[r]                             & IQ_{U}                                                 \\
{\bK_{\Delta_{X^2}\times [0,\infty)}} \arrow[u] \arrow[r] \arrow[ru] & {\Delta_{!}I\boxtimes \bK_{[0,\infty)}} \arrow[u]
\end{tikzcd}
\end{equation}
One can verify that two choices of chain representatives $\bK_{\Delta_{X^2}\times \{0\}} \rightarrow IQ_{U}$ define the same object in $ho({^{\triangleright}C(X^2\times\bR)})$. We denote this object by $[\bK_{\Delta_{X^2}\times \{0\}} \rightarrow Q_U] \in ho({^{\triangleright}C(X^2\times\bR)})$. It should be noted that we abuse the notation $\bK_{\Delta_{X^2}\times \{0\}} \rightarrow Q_U$ here, which can represent either a morphism in the triangulated derived category or an object in the homotopy category of a model category. The latter represents a chain map that represents the former (which is an equivalence class of roofs of homotopy classes of chain maps) in the triangulated derived category. It will be clear how to distinguish between them in the following discussion.

We introduce the following proposition, which serves as a motivation for our construction.
\begin{Prop}\label{proposition: derived comparision}If $\bK$ is a field, then for each $n\in \bN$, the $n$-th component of $F_{\bullet,T}(\bK_{\Delta_{X^2}\times \{0\}} \rightarrow Q_U)$ is isomorphic to $F_n^{out}(U,\bK)_T$ in $D(\bK-\Mod)$, and the morphism $d_i$ represents the isomorphism $F_n^{out}(U,\bK)_T\cong F_{n+1}^{out}(U,\bK)_T$, of \autoref{cyclicstructure step1}, in $D(\bK-\Mod)$.
\end{Prop}
\begin{proof}We take a fibrant resolution of $Q_U$, say $Q_U\rightarrow IQ_U$, for computation. Since $IQ_U$ is component-wise c-soft and $\bK$ is a field, it follows that $IQ_U^{\boxtimes n}$ is c-soft \cite[Exercise II.14]{BredonSheaf}. Therefore, $IQ_U^{\boxtimes n}$ is both relative c-soft with respect to projection and summation. By examining the proof of \autoref{lemma 3.19}, one can observe that each step of the construction descends to the derived category.  Eventually, we obtain $\pi_{n !}(\Tilde{\Delta}_{n}^{-1}s_{t!}^n(IQ_U^{\boxtimes n}))\cong F_n^{out}(U,\bK)$, since all the functors involved have finite cohomological dimension.

The second statement follows that the object $[\bK_{\Delta_{X^2}\times \{0\}} \rightarrow Q_U]$ in $ho({^{\triangleright}C(X^2\times\bR)})$ represents the morphism $\bK_{\Delta_{X^2}\times [0,\infty)} \rightarrow Q_U$ in $D(X^2\times \bR)$.
\end{proof}
Therefore, for $[\bK_{\Delta_{X^2}\times \{0\}} \rightarrow Q_U] \in ho({^{\triangleright}C(X^2\times\bR)})$, we {\em define}
\begin{equation}
 F_{\bullet}^{out}(U,\bK)_T\coloneqq  F_{\bullet,T} (\bK_{\Delta_{X^2}\times \{0\}} \rightarrow Q_U)  \in D_{dg}(\Delta C_{0,dg}(\bK-\Mod)).
\end{equation}
On the other hand, let us define another pre-cocyclic $\bK$-module $F^{nd}_\bullet(T^*X,\bK)_T$ as follows: Recall that we took a resolution $\bK_{X}\xrightarrow{qis} I$, then we have a resolution ${\bK_{\Delta_{X^2}\times [0,\infty)}} \xrightarrow{qis} \Delta_{!}I\boxtimes \bK_{[0,\infty)}$. 

We first set $F^{nd}_{n+1}(T^*X,\bK)_T= \left(\pi_{(n+1)!}(\Tilde{\Delta}_{n+1}^{-1}(s_{t!}^{n+1}(\Delta_! I\boxtimes \bK_{[0,\infty)})^{\boxtimes n+1}))\right)_T$ for $n\in \bN_0$, which is isomorphic to $\left(\pi_{(n+1)!}(\Tilde{\Delta}_{n+1}^{-1}((\Delta_! I)^{\boxtimes n+1}\boxtimes \bK_{[0,\infty)}))\right)_T$ as chain complexes. Next, the morphism $d_i$ is induced by $\bK_{\Delta_{X^2}\times [0,\infty)}\rightarrow \Delta_!I\boxtimes \bK_{[0,\infty)}$. In this case, it is by construction that $(\Delta_! I\boxtimes \bK_{[0,\infty)})^{\boxtimes n+1}$ is already $s_{t!}^{n+1}$-acyclic and $\pi_{n+1}$-c-soft. Therefore, we would like to use $\bK_{\Delta_{X^2}\times [0,\infty)}$ rather than $\bK_{\Delta_{X^2}\times \{0\}}$ to maintain the correct direction of the arrows. The definition of $t_n$ is the same. We also have that $F^{nd}_{n+1}(T^*X,\bK)_T$ represents $F_{n+1}(T^*X,\bK)_T$ in $D(\bK-\Mod)$. Notice that $I$ is flat for all rings $\bK$ regardless requiring $\bK$ to be fields. However, to be consistent with the definition of $F_{\bullet}^{out}(U,\bK)_T$, we continue to assume that $\bK$ is a field here.

We can also verify that if we take two different resolutions of $\bK_{X}\rightarrow I,I'$, we will obtain the same object in $D_{dg}(\Delta C_{0,dg}(\bK-\Mod))$. Therefore, we denote the resulting object in $D_{dg}(\Delta C_{0,dg}(\bK-\Mod))$ by $F_\bullet(T^*X,\bK)_T$. By direct computation, we can find that $F^{S^1}_{\bullet}(T^*X,\bK)_{T}$ is isomorphic to $\tnR\Gamma_c(X,\bK_X)$ with the trivial action for all $T\geq 0$. 
 
Next, we construct a morphism $F_\bullet(T^*X,\bK)_T \rightarrow F_\bullet^{out}(U,\bK)_T$. At the chain level, we have already constructed a morphism $\Delta_!I\boxtimes \bK_{[0,\infty)}\rightarrow IQ_U$  along with the diagram \eqref{diagram of Q}. Using this, we can directly construct a morphism of pre-cocyclic complexes of $\bK$-vector spaces:
\[F_\bullet^{nd}(T^*X,\bK)_T \rightarrow F_{\bullet,T}^{nd}(\bK_{\Delta_{X^2}\times \{0\}} \rightarrow IQ_U ).\]
It can be verified that this morphism is independent of the chosen resolutions, thereby defining a degree $0$ closed morphism
\begin{equation}\label{equation: structure morphism of S^1}
  F_\bullet(T^*X,\bK)_T \rightarrow F_\bullet^{out}(U,\bK)_T \quad   \in D_{dg}(\Delta C_{0,dg}(\bK-\Mod)).
\end{equation}
Now, we apply the cyclic cochain functor to \eqref{equation: structure morphism of S^1}, we have
\begin{Prop}If $\bK$ is a field, then for an open set $U$ and $T\geq 0$,
\begin{align*}
 F_{\bullet}^{S^1,out}(U,\bK)_{T}\coloneqq CC^*(F_{\bullet}^{out}(U,\bK)_T),\qquad 
 F_{\bullet}^{S^1}(T^*X,\bK)_{T}\coloneqq CC^*(F_{\bullet}(T^*X,\bK)_T),
 \end{align*}
 define objects of $D_{dg}(\bK[\epsilon]-\Mod)$.
 The degree $0$ closed morphism
 \begin{equation}\label{equation: defining morphism of CT}
    [F_{\bullet}^{S^1}(T^*X,\bK)_{T} \rightarrow  F_{\bullet}^{S^1,out}(U,\bK)_{T}] \coloneqq CC^*(F_\bullet(T^*X,\bK)_T \rightarrow F_\bullet^{out}(U,\bK)_T) 
 \end{equation}
defines an object in $\mathcal{M}\mathit{or}(D_{dg}(\bK[\epsilon]-\Mod))$.
\end{Prop}
Currently, although we cannot directly define $F^{S^1}_{\bullet}(U,\bK)_T$ from resolutions of $P_U$ for all open sets $U$, the situation for $U=T^*X$ is quite straightforward as we have presented. The tautological triangle \eqref{Tautological exact triangle} suggests that we should expect $F^{S^1}_{\bullet}(U,\bK)$ to be the cocone of the morphism \eqref{equation: defining morphism of CT}. Therefore, we {\em define}
\begin{equation}\label{equation: defining CT}
F_{\bullet}^{S^1}(U,\bK)_{T}\coloneqq cocone(F_{\bullet}^{S^1}(T^*X,\bK)_{T}\rightarrow F_{\bullet}^{S^1,out}(U,\bK)_{T}). \end{equation}
\begin{Def}\label{def: s1 CT Complex}For a field $\bK$, $T\geq 0$ and an open set $U\subset T^*X$, we define the $S^1$-equivariant Chiu-Tamarkin invariant to be
\begin{align*}
C^{S^1,out}_{T}(U,\bK)\coloneqq&\RHOM_{\bK[\epsilon]}(F_{\bullet}^{S^1,out}(U,\bK)_{T},\bK[-d]) , \\
C^{S^1}_{T}(U,\bK)\coloneqq& \RHOM_{\bK[\epsilon]}(F_{\bullet}^{S^1}(U,\bK)_{T},\bK[-d]),
\end{align*}
where we equip $\bK$ with the trivial mixed complex structure, and $\RHOM_{\bK[\epsilon]}$ denotes the derived hom in the dg mixed derived category.
\end{Def}
\begin{RMK}\label{cyclic definition remark}
In the dg derived category $D_{dg}(\bK[\epsilon]-\Mod)$, the definition of $F_{\bullet}^{S^1}(U,\bK)_{T}$ and $C^{S^1}_{T}(U,\bK)$ is canonical. If we work in the triangulated derived category, $C^{S^1}_{T}(U,\bK)$ is not canonical, but its cohomology, what we really care about, is well defined. 

It is worth mentioning that in previous works, our focus was on the linear dual $C^{\bZ/\ell}_{T}(U,\bK)$ rather than $F_\ell(U,\bK)_{T}$ because we were unable to prove the invariance (\cite[Theorem 4.7]{chiu2017}) for $F_\ell(U,\bK)_{T}$ directly. However, at present, we can directly establish the invariance for all versions of $F^{S^1}_\bullet(U,\bK)_{T}$ without resorting to the linear dual $C^{S^1}_{T}(U,\bK)$. This also holds true for the $\bZ/\ell$ versions.
\end{RMK}
Tautologically, the definition gives rise to the {\em tautological exact triangles}:
\begin{equation}\label{S^1 Tautological exact triangle}
\begin{split}
F^{S^1}_\bullet(U,\bK)_{T}\rightarrow \tnR\Gamma_c(X,\bK_X^{S^1})\rightarrow  F^{S^1,out}_\bullet(U,\bK)_{T}\xrightarrow{+1},\\
      C_{T}^{S^1,out}(U,\bK)\rightarrow\tnR \Gamma(X,\omega_X^{S^1})\rightarrow C_{T}^{S^1}(U,\bK) \xrightarrow{+1}.
\end{split}
\end{equation}
By definition, we have that $H^*C^{S^1}_{T}(U,\bK)$ and $H^*C^{S^1,out}_{T}(U,\bK)$ are left modules over $\textnormal{Ext}^*_{S^1}(\bK,\bK)\cong \bK[u]$. As we pointed out before, one has $H^*C^{S^1}_{T}(T^*X,\bK)\cong H_{d-*}^{BM}(X,\bK)\otimes \bK[u]$ for all $T\in [0,\infty]$, with the trivial $\bK[u]$ action on $H_{d-*}^{BM}(X,\bK)$.

Lastly, it is worth noting that although we define our positive $S^1$-equivariant Chiu-Tamarkin invariant in the mixed derived category, it can be computed in the derived category of pre-cocyclic complexes. Precisely, we have
\begin{Prop}\label{prop: computation in cyclic category}For a field $\bK$, $T\geq 0$ and an open set $U\subset T^*X$, we have
\begin{align*}
C^{S^1,out}_{T}(U,\bK)&\cong \RHOM(F_{\bullet}^{out}(U,\bK)_{T},\bK[-d]) ,\\
C^{S^1}_{T}(T^*X,\bK)&\cong \RHOM(F_{\bullet}(T^*X,\bK)_{T},\bK[-d])
,
\end{align*}
and 
\[C^{S^1}_{T}(U,\bK)\cong cone( \RHOM(F_{\bullet}^{out}(U,\bK)_{T},\bK[-d]) \rightarrow  \RHOM(F_{\bullet}(T^*X,\bK)_{T},\bK[-d])) . \]
where $\RHOM$ here are computed in ${D_{dg}(\Delta C_{0,dg}(\bK-\Mod))}$.
\end{Prop}
\begin{proof}
For the first two isomorphisms, a similar statement is presented for the cyclic cohomology of cocyclic modules in \cite[6.2.9]{LodayCyclic} with the same proof for the cyclic homology in Theorem 6.2.8 of loc. cit. The same method works for pre-cocyclic complex here since the proof therein did not use degeneracy maps. For the second, we only need to notice that the proof of the first two is functorial. 
\end{proof}

\subsection{Persistence module}\label{s1 persistence} The persistence module structure still exists in the $S^1$-theory.

Recall that the translation $\tnT_c$ is a diffeomorphism, so $\tnT_{c*}$ defines an isomorphism of the abelian categories of chain complexes for all $c\in \bR$. In particular, for the fibrant resolution $Q_U\rightarrow IQ_U$ of $Q_U$, we have that $\tnT_{c*}Q_U\rightarrow \tnT_{c*}IQ_U$ is a fibrant resolution of $\tnT_{c*}Q_U$ for all $c\in \bR$. When $c\geq 0$, we have a morphism $Q_U\rightarrow \tnT_{c*}Q_U$ in the derived category. By the functoriality of fibrant resolutions, we have a chain map $IQ_U\rightarrow \tnT_{c*}IQ_U$ that represents $Q_U\rightarrow \tnT_{c*}Q_U$.

Now, we can construct the following diagram using similar arguments as \autoref{lemma 3.19}:
\begin{equation*}
  \begin{tikzcd}
\delta_{n!}\Tilde{\Delta}_{n}^{-1}s_{t!}^n(IQ_U^{\boxtimes n}) \arrow[d] \arrow[r, "\cong"]                  & { \Tilde{\Delta}_{n+1}^{-1}(s_{t!}^{n+1}([n]IQ_U^{\boxtimes n+1}) )} \arrow[r] \arrow[d] \arrow[r]          &  \Tilde{\Delta}_{n+1}^{-1}(s_{t!}^{n+1}(IQ_U^{\boxtimes n+1}) ) \arrow[d]                               \\
\delta_{n!}\Tilde{\Delta}_{n}^{-1}s_{t!}^n(\tnT_{c/n*}IQ_U^{\boxtimes n}) \arrow[r, "\cong"] \arrow[d, "\cong"] & { \Tilde{\Delta}_{n+1}^{-1}(s_{t!}^{n+1}([n](\tnT_{c/n*}IQ_U^{\boxtimes n+1}) )} \arrow[r] \arrow[d, "\cong"] &  \Tilde{\Delta}_{n+1}^{-1}(s_{t!}^{n+1}(IQ_U\boxtimes \tnT_{c/n*}IQ_U^{\boxtimes n+1}) ) \arrow[d, "\cong"] \\
 \tnT_{c*}\delta_{n!}\Tilde{\Delta}_{n}^{-1}s_{t!}^n(IQ_U^{\boxtimes n}) \arrow[r, "\cong"]                     & {  \tnT_{c*}\Tilde{\Delta}_{n+1}^{-1}(s_{t!}^{n+1}([n]IQ_U^{\boxtimes n+1}) )} \arrow[r]                       &  \tnT_{c*}\Tilde{\Delta}_{n+1}^{-1}(s_{t!}^{n+1}(IQ_U^{\boxtimes n+1}) )   .                               
\end{tikzcd}   
\end{equation*}Now, after applying the direct image functor $\pi_{(n+1)!}$, we construct a morphism of pre-cosimplicial complexes of c-soft sheaves over $\bR$ for $c\geq 0$, say:
\begin{equation}\label{equation: persistence structure sheaf level}
 (F_{n+1}^{nd}(IQ_U),d_i)_{n\in \bN_0}  \rightarrow (\tnT_{c*}F_{n+1}^{nd}(IQ_U),d_i)_{n\in \bN_0}  .      
\end{equation}
As we choose a particular morphism $[n](\tnT_{c/n*}IQ_U)^{\boxtimes n+1} \rightarrow IQ_U\boxtimes (\tnT_{c/n*}IQ_U)^{\boxtimes n+1}$, we should be careful that if it yields a morphism of pre-cocyclic complexes of sheaves. Fortunately, the summation maps and projections $\pi_{n+1}$ are invariant under cyclic permutations, and thus, the morphism \eqref{equation: persistence structure sheaf level} does define a morphism of pre-cocyclic complexes of sheaves over $\bR$. Subsequently, by taking stalks at $T+c$ and $T$ for $T, c\geq 0$, \eqref{equation: persistence structure sheaf level} gives us a morphism of pre-cocyclic complexes of $\bK$-vector spaces:
\[F_{\bullet,T+c}^{nd} (\bK_{\Delta_{X^2}\times \{0\}} \rightarrow IQ_U) \rightarrow  F_{\bullet,T}^{nd} (\bK_{\Delta_{X^2}\times \{0\}} \rightarrow IQ_U).\]
Therefore, we obtain a degree $0$ closed morphism in ${D_{dg}(\Delta C_{0,dg}(\bK-\Mod))}$, and then a degree $0$ closed morphism in $D_{dg}(\bK[\epsilon]-\Mod)$ after applying the cyclic cochain functor:
\[F_{\bullet}^{S^1,out}(U,\bK)_{T+c} \rightarrow F_{\bullet}^{S^1,out}(U,\bK)_{T}.\]
We can check that we have similar morphisms for $F_{\bullet}^{S^1}(T^*X,\bK)_{T}$, and we have a commutative diagram: 
\begin{equation*}
 \begin{tikzcd}
{F_{\bullet}^{S^1}(T^*X,\bK)_{T+c}} \arrow[d] \arrow[r] & {F_{\bullet}^{S^1,out}(U,\bK)_{T+c}} \arrow[d] \\
{F_{\bullet}^{S^1}(T^*X,\bK)_{T}} \arrow[r]             & {F_{\bullet}^{S^1,out}(U,\bK)_{T}}  .          
\end{tikzcd}   
\end{equation*}
Consequently, by utilizing the functoriality of the cone in dg categories, we obtain the following morphism:
\begin{equation}\label{equation: structure maps of persistence module}
  F_{\bullet}^{S^1}(U,\bK)_{T+c} \rightarrow F_{\bullet}^{S^1}(U,\bK)_{T}.  
\end{equation}
Subsequently, our construction can be summarized as:
\begin{Prop}\label{prop: s1 persistence}The maps $T\mapsto H^*C^{S^1}_{T}(U,\bK)$ and $T\mapsto H^*C^{S^1,out}_{T}(U,\bK)$ are two persistence module over $\bK[u]$ where $|u|=2$.    
\end{Prop}
We can take limits to define $H^*C^{S^1}_{\infty}(U,\bK)$ and $H^*C^{S^1,out}_{\infty}(U,\bK)$. Moreover, we can define $H^*C^{S^1}_{(T,T']}(U,\bK)$ and $H^*C^{S^1,out}_{(T,T']}(U,\bK)$ as in \autoref{subsection Z/l persistence}.

\subsection{Gysin sequence}\label{section: gysin}
An important ingredient in the theory of mixed complexes is the Connes long exact sequence, also known as the Gysin long exact sequence.

Specifically, for a trivial mixed complex $\bK$, we have $\textnormal{Ext}^*_{\bK[\epsilon]}(\bK,\bK)\cong \bK[u]$, where $\bK[u]$ is a polynomial ring with $|u|=2$. In this context, there exists a morphism $u:\bK\rightarrow \bK[2]$. This morphism can be embedded into the distinguished triangle:
\[\bK[\epsilon] \xrightarrow{\epsilon=0} \bK\xrightarrow{u}\bK[2] \xrightarrow{+1}.\]
For any mixed complex $M=(M,b,B)$, applying the functor $\RHOM_{\bK[\epsilon]}(-,M)$, we obtain the distinguished triangle:
\[M\fpw{u}/uM\fpw{u}\rightarrow  M\fpw{u}\xrightarrow{u} M\fpw{u}  \xrightarrow{+1}\]
of $\bK[u]$-modules. Then it induces long exact sequence of $\bK$-vector spaces:\[ \textnormal{Ext}^{p-2}_{\bK[u]}(M\fpw{u},\bK)\xrightarrow{\cdot u}\textnormal{Ext}^p_{\bK[u]}(M\fpw{u},\bK)\rightarrow\textnormal{Ext}^p_{\bK[u]}(M\fpw{u}/uM\fpw{u},\bK)\xrightarrow{+1}.\]
Using the Koszul duality \eqref{Koszul duality of mixed complexes}, we have the following long exact sequence
\[ \textnormal{Ext}^{p-2}_{\bK[\epsilon]}(M,\bK)\xrightarrow{\cdot u}\textnormal{Ext}^p_{\bK[\epsilon]}(M,\bK)\rightarrow\textnormal{Ext}^p_{\bK}(M,\bK)\xrightarrow{+1}.\]
The third term follows from $\textnormal{Ext}^p_{\bK[u]}(M\fpw{u}/uM\fpw{u},\bK)\cong  \textnormal{Ext}^p_{\bK}(M,\bK)$, where we treat $M=(M,b)$ as a chain complex (rather than a mixed complex) on the right hand side.

Now, we can take $M=F_\bullet^{S^1,out}(U,\bK)$ and $M=F_\bullet^{S^1}(U,\bK)$ to establish the following long exact sequences:
\begin{align*}
 &H^{p-2}C_{T}^{S^1,out}(U,\bK)\xrightarrow{\cdot u}H^{p}C_{T}^{S^1,out}(U,\bK)\rightarrow\textnormal{Ext}^p_{\bK}(F_{\bullet}^{S^1,out}(U,\bK)_T,\bK)\xrightarrow{+1},\\
 &H^{p-2}C_{T}^{S^1}(U,\bK)\xrightarrow{\cdot u}H^{p}C_{T}^{S^1}(U,\bK)\rightarrow\textnormal{Ext}^p_{\bK}(F_{\bullet}^{S^1}(U,\bK)_T,\bK)\xrightarrow{+1}.
\end{align*}
Let us now proceed to compute $\textnormal{Ext}^p_{\bK}(F_{\bullet}^{S^1,out}(U,\bK)_T,\bK)$ and $\textnormal{Ext}^p_{\bK}(F_{\bullet}^{S^1}(U,\bK)_T,\bK)$. We begin by taking a fibrant resolution of $Q_U$, denoted as $Q_U\rightarrow IQ_U$. Consequently, $F_{\bullet}^{out}(U,\bK)_T$ is represented by $F^{nd}_{\bullet,T}(\bK_{\Delta_{X^2}\times \{0\}} \rightarrow IQ_U) $. By forgetting cyclic permutations in $F_{\bullet,T}(\bK_{\Delta_{X^2}\times \{0\}} \rightarrow IQ_U) $, we obtain a pre-cosimplicial complex of $\bK$-vector spaces denoted as $F=(F_{n+1}^{nd}(IQ_U)_T, d_i )_{n\in \bN_0}$. Therefore, $F_{\bullet}^{S^1,out}(U,\bK)_T$ is isomorphic to the non-normalized chain complex of $F$ in the derived category $D(\bK-\Mod)$.

Consider the pre-cosimplicial complex $G=(G_{n+1}, 0)_{n\in \bN_0}$ defined by $G_1=F_{1}^{nd}(IQ_U)_T$ and $G_{n+1}=0$ for $n>0$. In other words, $G$ represents the 0-coskeleton of $F$. We have a cosimplicial map $F\rightarrow G$. Since we already know that all $d_i$ represent isomorphisms in the derived category of sheaves $D(\bR)$ (as stated in Proposition \autoref{proposition: derived comparision}), it follows that $d_i$ are quasi-isomorphisms of chain complexes. Hence, the non-normalized chain complex functor yields a quasi-isomorphism $(Tot(F),b+\partial)\rightarrow(Tot(G),0+\partial)$. However, since $G$ is concentrated in the level $n=0$, we have $(Tot(G),\partial)= F^{nd}_1(IQ_U)$ as a chain complex. Consequently, we obtain
\[F^{S^1,out}_{\bullet}(U,\bK)_T\cong (Tot(G),\partial)=F^{nd}_1(IQ_U)\cong F^{out}_1(U,\bK)\,\in D(\bK-Mod).\] 
For $F_{\bullet}^{S^1}(U,\bK)_T$, upon applying the forgetful functor, we have a distinguished triangle in $D(\bK-\Mod)$:
\[ F_{\bullet}^{S^1}(U,\bK)_{T}\rightarrow \tnR\Gamma_c(X,\bK) \rightarrow  F_{\bullet}^{S^1,out}(U,\bK)_{T} \xrightarrow{+1}. \]
However, we already know that $F_{\bullet}^{S^1,out}(U,\bK)_{T}\cong F^{out}_1(U,\bK)_{T}$ in $D(\bK-\Mod)$. Therefore, we can conclude that $F_{\bullet}^{S^1}(U,\bK)_{T}\cong F_1(U,\bK)_{T}$ in $D(\bK-\Mod)$ by utilizing the tautological triangle \eqref{pre Tautological exact triangle} of $F_1$.

In particular, we prove that under the forgetful functor $For:D_{dg}(\bK[\epsilon]-\Mod)\rightarrow D_{dg}(\bK-\Mod)$, we have 
\begin{equation}\label{equation: forgetful S^1 to nonequivariant}
   For(C_{T}^{S^1}(U,\bK))\cong C_{T}(U,\bK)\quad \in\quad D_{dg}(\bK-\Mod) .
\end{equation}
And, we obtain the Gysin long exact sequence for the $S^1$-equivariant Chiu-Tamarkin cohomology:
\begin{equation}\label{Gysin long exact sequence}
\begin{split}
    &H^{p-2}C_{T}^{S^1,out}(U,\bK)\xrightarrow{\cdot u}H^{p}C_{T}^{S^1,out}(U,\bK)\rightarrow H^pC^{out}_{T}(U,\bK)\xrightarrow{+1},\\
    &H^{p-2}C_{T}^{S^1 }(U,\bK)\xrightarrow{\cdot u}H^{p}C_{T}^{S^1 }(U,\bK)\rightarrow H^pC _{T}(U,\bK)\xrightarrow{+1}.
\end{split}
\end{equation}

\subsection{Restrict to cyclic groups}\label{subsection: Restriction morphisms}For $\ell \geq 1$, we will construct a module morphism \[H^qC_{T}^{S^1}(U,\bK) \rightarrow H^qC^{\bZ/\ell}_{T}(U,\bK).\]
When $\ell=1$, we can directly use the forgetful functor \eqref{equation: forgetful functor}. Then combining the discussion in \autoref{section: gysin}, we have
\[H^qC_{T}^{S^1}(U,\bK) \rightarrow For(H^qC^{S^1}_{T}(U,\bK))=H^qC_{T}(U,\bK).\]

Next, we assume $\ell \geq 2$. Given a pre-cocyclic complex $(C_{n+1},d_i,t_n)_{n\in \bN_0}$, it is clear that $(C_\ell,t_{\ell-1})$ forms a chain complex of  $\bZ/\ell$-modules since $t_{\ell-1}^\ell=1$. Therefore, we obtained a dg functor
\[\Delta C_{0,dg}(\bK-\Mod) \rightarrow  C_{dg}(\bK[\bZ/\ell]-\Mod),\qquad (C_{n+1},d_i,t_n)_{n\in \bN_0}\mapsto (C_\ell,t_{\ell-1}).\]
It is clear we can descent the dg functor to derived categories
\[R_\ell: D_{dg}(\Delta C_{0,dg}(\bK-\Mod)) \rightarrow  D_{dg}(\bK[\bZ/\ell]-\Mod)\cong D_{\bZ/\ell,dg}(\pt).\]
Now, we apply the dg functor $R_\ell$ to $F_\bullet(T^*X,\bK)_T$, $F_\bullet^{out}(U,\bK)_T$ and the morphism
\[F_\bullet(T^*X,\bK)_T \rightarrow F_\bullet^{out}(U,\bK)_T.\]
By definition, we have the following isomorphisms and commutative diagram in $D_{dg}((\bK[\bZ/\ell])-\Mod)\cong D_{\bZ/\ell,dg}(\pt)$:
\[R_\ell(F_\bullet(T^*X,\bK)_T) \cong F_\ell(T^*X,\bK)_T,\qquad R_\ell(F_\bullet^{out}(U,\bK)_T) \cong F_\ell^{out}(U,\bK)_T,\]
and
\begin{equation*}
   \begin{tikzcd}
{R_\ell(F_\bullet(T^*X,\bK)_T) } \arrow[d, "\cong"] \arrow[r] & {R_\ell(F_\bullet^{out}(U,\bK)_T)} \arrow[d, "\cong"] \\
{F_\ell(T^*X,\bK)_T} \arrow[r]                                & { F_\ell^{out}(U,\bK)_T}  .                                
\end{tikzcd} 
\end{equation*}
Moreover, since $R_\ell$ is a functor, we have a morphism
\begin{align*}
    & \RHOM_{D_{dg}(\Delta C_{0,dg}(\bK-\Mod))} (F_\bullet^{out}(U,\bK)_T,\bK[-d])  \\
 \rightarrow   & \RHOM_{D_{\bZ/\ell,dg}(\pt)} (R_\ell(F_\bullet^{out}(U,\bK)_T),R_\ell(\bK[-d]))\\
 \cong & \RHOM_{D_{\bZ/\ell,dg}(\pt)}  (F_\ell^{out}(U,\bK)_T,\bK[-d]).
\end{align*}
Similarly, we can replace $F_\bullet^{out}(U,\bK)_T$ with $F\bullet(T^*X,\bK)_T$ and consider the canonical morphism between them. Then, in light of \autoref{prop: computation in cyclic category}, we obtain the following morphism:
\begin{equation}\label{equation: restriction morphism 2}
\begin{split}
    H^qC_{T}^{S^1,out}(U,\bK) &\rightarrow H^qC^{\bZ/\ell,out}_{T}(U,\bK),\\
     H^qC_{T}^{S^1}(U,\bK) &\rightarrow H^qC^{\bZ/\ell}_{T}(U,\bK).
\end{split}
\end{equation}
\begin{RMK}If we set $\ell=1$, the procedure still works, and we can observe that the resulting morphism is the same as the morphism constructed by the forgetful functor approach we explained at the beginning of this subsection.
\end{RMK}

\subsection{Functoriality and invariance}
In this subsection, we would like to prove the following theorem:
\begin{Thm}\label{invariance1} Let $U,U_1,U_2$ be open sets and let $U_1\xhookrightarrow{i} U_2$ be an inclusion. For $T\geq 0$, we have  
\begin{enumerate}[fullwidth]
    \item There exists a morphism $C^{S^1}_T(U_2,\bK) \xrightarrow{i^*} C^{S^1}_T(U_1,\bK)$, which is functorial with respect to inclusions of open sets.
    \item For a compactly supported Hamiltonian homeomorphism $\varphi:T^*X  \rightarrow T^*X $ in the sense of \cite{Oh-MullerHOMEO}, there exists an isomorphism $\Phi^{S^1}_{T }:C^{S^1}_T(U,\bK) \xrightarrow{\cong} C^{S^1}_T\left(\varphi(U),\bK\right)$ in the mixed derived category. The isomorphism $\Phi^{S^1}_{T}$ is functorial with respect to the restriction morphisms in (1). When $U=T^*X$, we have $\Phi^{S^1}_{T}=\id$.
\end{enumerate}
Notably, we are able to prove the corresponding results for all versions of $F^{S^1}(U,\bK)_T$, and the results for $C_T^{S^1}(U,\bK)$ follow directly after a linear dual.
\end{Thm}
\begin{RMK}
The theorem also holds for the $\bZ/\ell$-theory \cite{chiu2017,Capacities2021}. However, in the cited references, the invariance for Hamiltonian homeomorphisms was not explicitly stated, although it is indeed true. In the following discussion, we will explain why our previous argument extends to the case of Hamiltonian homeomorphisms.

On the other hand, due to the absence of $S^1$-equivariant sheaves, we require a different argument to establish corresponding results for the $S^1$-theory. Actually, the present proof for the $S^1$-theory can also provide a new proof for the corresponding result in the $\bZ/\ell$-theory. 
\end{RMK}
{\bf Functoriality}:

\begin{proof}[{{Proof of} {\em (1)}}]In the derived category $D(X^2\times \bR)$, we have a commutative diagram (\autoref{functorial})
\begin{equation}\label{diagram of Q functorial-derived}
    \begin{tikzcd}
{{\bK_{\Delta_{X^2}\times [0,\infty)}}} \arrow[r] \arrow[rr, bend left] & Q_U \arrow[r] & Q_V
\end{tikzcd}
\end{equation}
Now, we take a fibrant resolution $Q_U\rightarrow IQ_U$, a flat and c-soft resolution $\bK_X\rightarrow I$, and some chain maps such that the diagram \eqref{diagram of Q} is commutative. 

On the other hand, we take a fibrant resolution $Q_V\rightarrow IQ_V$. As $IQ_V$ is fibrant, we have 
\[Hom_D(Q_U,Q_V)\cong Hom_K(IQ_U,IQ_V).\]
So, one can take a chain map $IQ_U\rightarrow IQ_V$ representing $Q_U\rightarrow Q_V$ in $D(X^2\times \bR)$.

Therefore, the chain map $IQ_U\rightarrow IQ_V$ and the diagram \eqref{diagram of Q} define the following commutative diagram
\begin{equation}\label{diagram of Q functorial-non-derived}
    \begin{tikzcd}
\bK_{\Delta_{X^2}\times \{0\}} \arrow[r] \arrow[rr, bend left]          & IQ_{U} \arrow[r]                                               & IQ_V \\
{{\bK_{\Delta_{X^2}\times [0,\infty)}} } \arrow[u] \arrow[r] \arrow[ru] & {{\delta_{!}I\boxtimes \bK_{[0,\infty)}}} \arrow[u] \arrow[ru] &  .   
\end{tikzcd}
\end{equation}
Arrows in the diagram \eqref{diagram of Q functorial-non-derived} define a morphism 
\[ [\bK_{\Delta_{X^2}\times \{0\}} \rightarrow Q_U] \rightarrow [\bK_{\Delta_{X^2}\times \{0\}} \rightarrow Q_V]\]
in $ ho({^{\triangleright}C(X^2\times\bR)})$, which represent the diagram \eqref{diagram of Q functorial-derived} in $D(X^2\times \bR)$. Therefore, we have a morphism
\[F_{\bullet}^{S^1,out}(U,\bK)_{T} \rightarrow F_{\bullet}^{S^1,out}(V,\bK)_{T}\]
since $CC^*\circ F_{\bullet,T}$ is a functor.
Moreover, the diagram \eqref{diagram of Q functorial-non-derived} also verifies the following commutative diagram:
\begin{equation*}
   \begin{tikzcd}
{F_{\bullet}^{S^1}(T^*X,\bK)_{T}} \arrow[d,"\id"] \arrow[r] & {F_{\bullet}^{S^1,out}(U,\bK)_{T}} \arrow[d] \\
{F_{\bullet}^{S^1}(T^*X,\bK)_{T}} \arrow[r]             & {F_{\bullet}^{S^1,out}(V,\bK)_{T}} ,          
\end{tikzcd}  
\end{equation*}
in the mixed derived category. Consequently, we have the morphism, by taking cocone,
\[F_{\bullet}^{S^1}(U,\bK)_{T} \rightarrow F_{\bullet}^{S^1}(V,\bK)_{T}.\qedhere\]
\end{proof}

{\bf Invariance} 
The proof of invariance is more intricate since the adjunction argument used in \cite{chiu2017,Capacities2021} does not apply here due to the lack of sufficient $S^1$-equivariant sheaves to apply the six operations. However, since the forgetful functor is conservative, it suffices to construct a morphism $F_{\bullet}^{S^1,out}(U,\bK)_{T}$ and $F_{\bullet}^{S^1,out}(\varphi(U),\bK)_{T}$  in the mixed derived category such that we obtain an isomorphism between the non-equivariant Chiu-Tamarkin invariants after forgetting the $S^1$-action.

\begin{Lemma}\label{cyclic invariance of trace}For $n\in \bN$ and complexes of sheaves $K_1,\, K_2,\cdots,K_n$ on $X^2\times \bR$, there exists a natural isomorphism that is invariant under cyclic permutation such that
\[\pi_{X^n!}\Tilde{\Delta}_{n}^{-1}s_{t!}^n(K_1\boxtimes K_2 \boxtimes\cdots \boxtimes K_n)\cong \pi_{X^n!}\Tilde{\Delta}_{n}^{-1}s_{t!}^n(K_n\boxtimes K_1 \boxtimes\cdots \boxtimes K_{n-1}),\]
where $\pi_{X^n}: X^n\times \bR_t\rightarrow \bR_t$ is the projection map.
\end{Lemma}
\begin{proof}Let $\tau$ denote the cyclic permutation of $n$ factors in a $n$-fold product space (such as $X^n$, $(X^2)^n$, or $(X^2\times \bR)^n$). We have the following isomorphism:
\[K_1\boxtimes K_2 \boxtimes\cdots \boxtimes K_n \cong \tau_!(K_n\boxtimes K_1 \boxtimes\cdots \boxtimes K_{n-1}),\]
where the tensor factors commute according to the Koszul rule. 
 
Now, observe that all $\pi_{X^n}$, $\Tilde{\Delta}_{n}$, and $s_{t}^n$ are equivariant with respect to cyclic permutations. Therefore, we have a natural isomorphism of functors: $\pi_{X^n!}\Tilde{\Delta}_{n}^{-1}s_{t!}^n=\pi_{X^n!}\Tilde{\Delta}_{n}^{-1}s_{t!}^n\tau_!$ which is equivariant with respect to the cyclic permutation. The desired result follows.
\end{proof}
To illustrate our idea and understand why the morphism we are going to construct is an isomorphism in the non-equivariant case, let us provide a new proof of the invariance of $F_1(U,\bK)$ under Hamiltonian homeomorphisms.

The crucial thing is that there exists a pair of sheaves, denoted as $\cK$ and $\overline{\cK}$, satisfying the property $\overline{\cK}\star {\cK}\cong \bK_{\Delta_{X^2}\times \{0\}}$, which quantize the Hamiltonian homeomorphism $\varphi$, its inverse $\varphi^{-1}$, and the composition $\varphi^{-1}\varphi=\id$.

In the smooth case, we assume that $\varphi=\varphi^H_1$ for a compactly supported Hamiltonian function $H$. In this scenario, we can choose $\cK=\cK(\varphi^H)_1$ and $\overline{\cK}=\cK^{-1}=\cK(\varphi^H)_{-1}$, as we did in \cite{Capacities2021}. If $\varphi$ is a Hamiltonian homeomorphism, we can utilize the quantization $\cK$ and $\overline{\cK}$ constructed by Asano-Ike using the homotopy colimit (\cite[Section 5]{asanoike2022complete}).

Then we can apply {\cite[Proposition 4.5]{chiu2017}} to conclude that $P_{\varphi(U)}\cong \cK\star P_U\star \overline{\cK}$. It is worth noting that the argument therein is applicable to Hamiltonian homeomorphisms constructed in \cite{asanoike2022complete} as well. Therefore, we conclude that
\[F_1(\varphi(U),\bK)\cong\tnR\pi_{X!} \Delta_{X^2}^{-1}(\cK\star P_U\star \overline{\cK})\cong \tnR\pi_{X^3!}\Tilde{\Delta}_{3}^{-1}\tnR s_{t!}^3 (\cK\dboxtimes P_U \dboxtimes \overline{\cK}).\]
Now, the derived version of \autoref{cyclic invariance of trace} shows that
\[ \tnR\pi_{X^3!}\Tilde{\Delta}_{3}^{-1}\tnR s_{t!}^3 (\cK\dboxtimes P_U \dboxtimes \overline{\cK})\cong \tnR\pi_{X^3!}\Tilde{\Delta}_{3}^{-1}\tnR s_{t!}^3 ( P_U \dboxtimes \overline{\cK}\dboxtimes \cK)\cong \tnR\pi_{X!} \Delta_{X^2}^{-1}(P_U\star \overline{\cK}\star\cK ).\]
Recall that the quantizations $\cK, \overline{\cK}$ satisfy $\overline{\cK}\star\cK \cong \bK_{\Delta_{X^2}\times \{0\}}$. Then we have that 
\[F_1(\varphi(U),\bK)\cong \tnR\pi_! \Delta_{X^2}^{-1}(P_U\star \overline{\cK}\star\cK )\cong \tnR\pi_!\Delta_{X^2}^{-1} (P_U \star \bK_{\Delta_{X^2}\times \{0\}})\cong \tnR\pi_! \Delta_{X^2}^{-1}(P_U ) =F_1(U,\bK).\]
In general, we can extend the proof to the $\bZ/\ell$-case and provide a new proof of the Hamiltonian invariance for the $\bZ/\ell$-equivariant Chiu-Tamarkin invariant using the morphism $\Phi_\ell$ we constructed in \eqref{equation: Phi_n}. Now, let us return to the $S^1$ case.
\begin{proof}[{{Proof of} {\em (2)}}]We consider the sheaf quantization $\cK$ of $\varphi$ as mentioned earlier. It follows from {\cite[Proposition 4.5]{chiu2017}} that $Q_{\varphi(U)}\cong \cK\star Q_U\star \overline{\cK}$. To complete the proof, we take resolutions explicitly and work within the category of complexes.

We fix fibrant resolutions $Q_{\varphi(U)}\rightarrow IQ_1$ and $Q_U\rightarrow IQ_0$ along with the associated morphisms, as discussed in \autoref{subsection: s1-theory}. Additionally, we select fibrant resolutions $\cK\rightarrow K$ and $\overline{\cK}\rightarrow \overline{K}$. Throughout the following discussion, we will use the notation of convolution, both in the derived and non-derived sense, interchangeably. We will explicitly mention the derived situations when necessary. However, since we have chosen fibrant resolutions, we can ensure that the chain maps we employ descend to the corresponding morphisms in the derived category, as outlined in \autoref{proposition: derived comparision}.

{\bf Step 1}: Realize the isomorphism $Q_{\varphi(U)}\cong \cK\star Q_U\star \overline{\cK}$, in $D(X^2\times \bR)$, as an isomorphism in $ho({^{\triangleright}C(X^2\times\bR)})$. Utilizing the chosen resolutions, we obtain the following morphism:
\[   IQ_1  \cong Q_{\varphi(U)}  \cong \cK\star Q_U\star  \overline{\cK}\cong K\star IQ_0\star \overline{K}  \qquad \in D(X^2\times \bR) .\]
Since $IQ_1$ is fibrant, we can lift the aforementioned isomorphism to a chain map, which is in fact a quasi-isomorphism:
\[IQ_1'\coloneqq K\star IQ_0\star \overline{K}  \rightarrow IQ_1.\]
As we have a chain map $\bK_{\Delta_{X^2}\times \{0\}}\rightarrow IQ_0$, we obtain the following chain maps:
\[ K\star \overline{K} \cong  K\star \bK_{\Delta_{X^2}\times \{0\}}  \star \overline{K} \rightarrow  IQ_1' = K\star IQ_0\star \overline{K} \rightarrow IQ_1,\]
where the first isomorphism is an isomorphism of chain complexes.

However, due to the construction of the quantization, we have an isomorphism in the derived category: $ \bK_{\Delta_{X^2}\times \{0\}} \cong K\star \overline{K} $. Although $K\star \overline{K}$ is not fibrant, it is c-soft. Additionally, $\bK_{\Delta_{X^2}\times \{0\}}$ is a bounded complex. Therefore, we can also find a chain map (which is actually a quasi-isomorphism) that represents this isomorphism. We denote the resulting chain map as:
\begin{equation}\label{equation: chain rep of GKS inverse iso.}
  f:\bK_{\Delta_{X^2}\times \{0\}} \rightarrow K\star \overline{K}.  
\end{equation}
Therefore, we obtain a sequence of chain maps:
\begin{equation*}
    \begin{tikzcd}
\bK_{\Delta_{X^2}\times \{0\}} \arrow[r, "f"] & K\star \overline{K} \arrow[r] & IQ_1'  \arrow[r] &  IQ_1,
\end{tikzcd}
\end{equation*}
where $IQ_1'\rightarrow IQ_1$ is a quasi-isomorphism. In particular, we have
\[[\bK_{\Delta_{X^2}\times \{0\}} \rightarrow IQ_1']\cong [\bK_{\Delta_{X^2}\times \{0\}} \rightarrow IQ_1]  \in  ho({^{\triangleright}C(X^2\times\bR)}),\]
and, consequently, we have
\[F_{\bullet}^{out}(\varphi(U),\bK)_{T} \cong F_{\bullet,T}^{out}(\bK_{\Delta_{X^2}\times \{0\}} \rightarrow IQ_1')\in D_{dg}(\Delta C_{0,dg}(\bK-\Mod)). \]
{\bf Step 2}: We will construct a morphism of pre-cocyclic complexes 
\[\Phi_\bullet :F^{nd}_{\bullet,T}([\bK_{\Delta_{X^2}\times \{0\}} \rightarrow IQ_1']) \rightarrow F^{nd}_{\bullet,T}([\bK_{\Delta_{X^2}\times \{0\}} \rightarrow IQ_0])\]
such that 
\[\Phi_1: F^{nd}_{1}( IQ_1')_T \rightarrow F^{nd}_{1}(IQ_0)\]
representing the isomorphism
\[F^{out}_1(\varphi(U),\bK)_T \cong F^{out}_1(U,\bK)_T \qquad \in D(\bK-\Mod).\]
Then the result follows from that the forgetful functor \eqref{equation: forgetful functor} is conservative, and we have an isomorphism $F^{S^1,out}_{\bullet}(U,\bK)_T \cong F^{out}_1(U,\bK)_T$ in $D(\bK-\Mod)$ for all open sets $U$.

The construction is based on a careful improvement of \autoref{cyclic invariance of trace}.

We start from an isomorphism of chain complexes:
\[IQ_1' = K \star IQ_0 \star \overline{K} \cong \pi_{\textnormal{in}!}s^3_{u!}d^{-1}  (K \boxtimes IQ_0\boxtimes \overline{K}). \]
Let us carefully name coordinates for products of $X$ and $\bR$:
\begin{align*}
    d: &\qquad X^4  \times \bR^3_u& &\rightarrow \qquad X^6 \times \bR^3_u, \\
         & (y_1,y_2,y_3,y_4,u_1,u_2,u_3)& &\mapsto (w_1,\dots,w_6,u_1,u_2,u_3)=(y_1,y_2,y_2,y_3,y_3,y_4,u_1,u_2,u_3);\\
    s^3_{u}: & \qquad X^4  \times \bR^3_u & &\rightarrow \qquad X^4  \times \bR_t,\\
    &(y_1,y_2,y_3,y_4,u_1,u_2,u_3) & &\mapsto (y_1,y_2,y_3,y_4,t)=(y_1,y_2,y_3,y_4,u_1+u_2+u_3); \\
\pi_{\textnormal{in}}: & \qquad X^4  \times \bR_t & &\rightarrow \qquad X^2  \times \bR_t,\\
    &(y_1,y_2,y_3,y_4,t) & &\mapsto (x_1,x_2,t)=(y_1,y_4,t).
\end{align*}
Under the coordinate convention, we require that $K$ is over $(w_1,w_2,u_1)$, $IQ_0$ is over $(w_3,w_4,u_2)$, and $\overline{K}$ is over $(w_5,w_6,u_3)$. Furthermore, under these identifications, we have $x_1=y_1=w_1$ and $x_2=y_4=w_6$.

When we have $n$ copies of $x_a, y_b, w_c$, etc., we use a superscript to distinguish them, such as $x^i_a,y^i_b,w^i_c$ for $i\in [n]$. Therefore, we have an isomorphism of chain complexes
\begin{equation}\label{equation: resolve def F_n IQz'}
\begin{split}
   &F_n^{nd}(IQ_1')=\pi_{n !}(\Tilde{\Delta}_{n}^{-1}s_{t!}^n((IQ_1')^{\boxtimes n})) \\
   \cong &\pi_{3n !}(\Tilde{\Delta}_{3n}^{-1}s_{u!}^{3n}( K\boxtimes IQ_0 \boxtimes \overline{K} \boxtimes\cdots \boxtimes K\boxtimes IQ_0 \boxtimes \overline{K}  )).    
\end{split}
 \end{equation}
Again, let us be careful with the coordinates. In the middle term, we use the coordinate $(\bq_i;t_i)_{i\in [n]}$ as usual, while on the right-hand side, we use the coordinates $(w_c^i,u_d^i)_{c,d;i\in [n]}$ for $c=1,2,3,4,5,6$ and $d=1,2,3$. Here, we identify $[n]$ with $\bZ/n$.

The map $\Tilde{\Delta}_{n}$ identifies $\bq_i=x_{2}^i=x_{1}^{i+1}$ for all $i\in \bZ/n$, while the map $\Tilde{\Delta}_{3n}$ identifies $w_2^i=w_3^i$, $w_4^i=w_5^i$ and $w_6^i=w_1^{i+1}$. Under the identification $x_1^i=w_1^i$ and $x_2^i=w_6^i$, $\Tilde{\Delta}_{n}$ and $\Tilde{\Delta}_{3n}$ are compatible. Notice that the $\bZ/n$-cyclic permutation operator on the right hand side of \eqref{equation: resolve def F_n IQz'} is given by 
\[(w_c^i,u_d^i)_{c,d;i\in [n]}\mapsto (w_c^{i+1},u_d^{i+1})_{c,d;i\in [n]}.\]
Then we have that \eqref{equation: resolve def F_n IQz'} is an equivariant isomorphism of chain complexes with respect to the $\bZ/n$-cyclic permutation. 

Now, similar to \autoref{cyclic invariance of trace}, we have an $\bZ/n$-equivariant isomorphism of chain complexes:
\begin{equation}\label{equation: move bar K isomorphism}
    \begin{split}
       & \pi_{3n !}(\Tilde{\Delta}_{3n}^{-1}s_{t!}^{3n}( K\boxtimes IQ_0 \boxtimes \overline{K} \boxtimes\cdots \boxtimes K\boxtimes IQ_0 \boxtimes \overline{K}  ))\\ \cong& \pi_{3n !}(\Tilde{\Delta}_{3n}^{-1}s_{t!}^{3n}(  IQ_0 \boxtimes \overline{K}\boxtimes K \boxtimes\cdots \boxtimes IQ_0 \boxtimes \overline{K} \boxtimes K)),
    \end{split}
\end{equation}
which is induced by the $\bZ/n$-equivariant map 
\[(w_1^i,w_2^i,w_3^i,w_4^i,w_5^i,w_6^i,u_1^i,u_2^i,u_3^i)_{i\in [n]}\mapsto (w_3^i,w_4^i,w_5^i,w_6^i,w_1^i,w_2^i,u_2^i,u_3^i,u_1^i)_{i\in [n]}.\]
Again, similar to \eqref{equation: resolve def F_n IQz'}, we have an $\bZ/n$-equivariant isomorphism of chain complexes:
\begin{equation}\label{equation: reverse def of F_n}
\begin{split}
   & \pi_{3n !}(\Tilde{\Delta}_{3n}^{-1}s_{t!}^{3n}(  IQ_0 \boxtimes \overline{K}\boxtimes K \boxtimes\cdots \boxtimes IQ_0 \boxtimes \overline{K} \boxtimes K))  \\
 \cong & \pi_{n !}(\Tilde{\Delta}_{n}^{-1}s_{t!}^n((IQ_0\star\overline{K}\star K )^{\boxtimes n}))= F_n^{nd}(IQ_0\star \overline{K}\star K ).
\end{split}
 \end{equation}
To summarize, we constructed an $\bZ/n$-equivariant isomorphism of chain complexes by composing \eqref{equation: resolve def F_n IQz'}, \eqref{equation: move bar K isomorphism} and \eqref{equation: reverse def of F_n}, say:
\[ \mu_n:F_n^{nd}(IQ_0 \star \overline{K}\star K) \cong F_n^{nd}(IQ_1').  \]
On the other hand, we have the morphism of chain complexes:
\begin{equation}\label{equation: Phi_n}
\Phi_n: F_n^{nd}(IQ_0 )\cong F_n^{nd}(IQ_0\star \bK_{\Delta_{X^2}\times \{0\}}) \xrightarrow{F_n^{nd}(IQ_0\star f)} F_n^{nd}(IQ_0\star \overline{K}\star K) \xrightarrow{\mu_n} F_n^{nd}(IQ_1').
\end{equation}
We already verify that the map $\Phi_n$ is $\bZ/n$-equivariant, so we only need to verify that $\Phi_\bullet$ defines a map of pre-cosimplical map, which is a long exercise.

Then we apply the cyclic cochain functor and descent to the mixed derived category. Consequently, we have a morphism
\[\Phi^{S^1,out}_{T}:F_{\bullet}^{S^1,out}(\varphi(U),\bK)_{T} \rightarrow F_{\bullet}^{S^1,out}(U,\bK)_{T}.\]
One can also check that, after applying the forgetful functor, $\Phi^{S^1,out}_{T}$ is conjugate to the isomorphism
\[F^{out}_1(\varphi(U),\bK)_T \cong F^{out}_1(U,\bK)_T \qquad \in D(\bK-\Mod).\]
Then we have that $\Phi^{S^1,out}_{T}$ defines an isomorphism in the mixed derived category.

Similarly, we obtain an isomorphism
\[\Phi^{S^1}_{T}: F_{\bullet}^{S^1}(\varphi(T^*X),\bK)_{T} \rightarrow F_{\bullet}^{S^1}(T^*X,\bK)_{T},\]
and one can check that it is the identity map in the mixed derived category.

Finally, our construction also gives a commutative diagram, in the mixed derived category,
\begin{equation*}
 \begin{tikzcd}
{F_{\bullet}^{S^1}(T^*X,\bK)_{T}} \arrow[d,"\id"] \arrow[r] & {F_{\bullet}^{S^1,out}(\varphi(U),\bK)_{T}} \arrow[d,"\Phi^{S^1,out}_{T}","\cong"'] \\
{F_{\bullet}^{S^1}(T^*X,\bK)_{T}} \arrow[r]             & {F_{\bullet}^{S^1,out}(U,\bK)_{T}} .          
\end{tikzcd}   
\end{equation*}
Then one obtains the isomorphism, in the mixed derived category,
\[\Phi^{S^1}_{T}: F_{\bullet}^{S^1}(\varphi(U),\bK)_{T} \rightarrow F_{\bullet}^{S^1}(U,\bK)_{T}.\qedhere\]
\end{proof}
\subsection{Integral coefficient and computation}\label{section: Z coefficient}
In this subsection, we will discuss some results regarding the definition over integral coefficient, and some computation problems. Therefore, we will not assume that $\bK$ is a field in this subsection.

Consider the following assumptions and related consequences: 

For a given open set $U\subset T^*X$, we say that it satisfies the {\em action cut-off condition} if there exists some $t_0\leq 0$ such that $(P_U)_{\{t<t_0\}}\cong 0$ in $D(X^2\times \bR)$. Alternatively, we can also require $(Q_U){{t<t_0}}\cong 0$. If the condition holds over $\mathbb{Z}$, it is valid for all coefficient rings.

In this case, it is evident that the morphism \[ [\bK_{\Delta_{X^2}\times \{0\}} \rightarrow Q_U] \rightarrow [\bK_{\Delta_{X^2}\times \{0\}} \rightarrow (Q_U)_{\{t_0 \leq t\leq N\}}] \in ho({^{\triangleright}C(X^2\times\bR)})\]
induced by the closed inclusion $[t_0,N]\subset \bR$ gives an isomorphism
\[F^{out}_\bullet(U,\bK)_T =F_{\bullet,T}([\bK_{\Delta_{X^2}\times \{0\}} \rightarrow Q_U])\cong F_{\bullet,T}([\bK_{\Delta_{X^2}\times \{0\}} \rightarrow (Q_U)_{\{t_0 \leq t\leq N\}}])\]
for $T\geq 0$ and a sufficiently large $N\in \mathbb{N}$ depending on $T$.

If $U$ satisfies the action cut-off condition, we can compute $F^{out}_\bullet(U,\bK)_T$ using $ (Q_U)_{\{t_0 \leq t\leq N\}}$. The advantage of $(Q_U)_{\{t_0 \leq t\leq N\}}$ is that it is likely a bounded complex and has a special form in the following sense.

We say that an object $F$ in $D(X^2\times \bR)$ is of {\em projective constant (PC) type} if there exists a continuous map of topological spaces $\pi: Z\rightarrow X^2\times \bR$ and a locally closed subset $W\subset Z$ such that $F\cong \tnR \pi_{!}\bK_W[i]$ for some $i\in \bZ$. If we replace $\bK_W[i]$ with some flat complex $G\in D^b(Z)$, we will call $F$ to be of {\em projective flat (PF) type}. For example, constant sheaves over $X^2\times \bR$ are trivially PC type. If $\bK$ is a field, then all sheaves are trivial of PF type. If $F$ is of PC type, it is a bounded complex.

When $F$ is of PF type, a variant of the construction in \autoref{section: cyclic structure} can be applied without assuming that $\bK$ is a field, as we do not need to derive tensor products for flat complexes. Consequently, it becomes possible to define the $S^1$-theory over $\bZ$. Furthermore, when microlocal kernels are of PC type, we will observe that the computation of $F^{S^1}$ reduces to the computation of the $S^1$-equivariant cohomology of an $S^1$-space. Therefore, we consider the following definition.
\begin{Def}\label{def: well-behaved microlocal kernels}For an open set $U\subset T^*X$, we say $U$ admits {\em well-behaved microlocal kernels} if the following conditions are satisfied:
\begin{enumerate}[fullwidth]
\item (Action cut-off condition) There exists some $t_0\leq 0$ such that $(P_U)_{\{t<t_0\}}\cong 0$ in $D(X^2\times \bR)$.

\item For every $N\in \bN$, there exists a topological space $Z_N$, a locally closed subset $W_{N} \subset X^2\times Z_N\times [t_0,N]$ and a closed subset $\Delta \times [0,N] \subset W_{N}$. Denote the projection map by $\pi_{Z_N}: X^2\times Z_N\times \bR \rightarrow  X^2 \times \bR$, the following properties hold: 
\begin{itemize}
    \item $W_N$ is a Euclidean neighborhood retracts (ENR);

    \item$\Delta =\{(x,x,z(x)):x\in X\}$, where $z:X\rightarrow Z_N$ is a continuous map. In particular, $\pi_{Z_N}|_{\Delta \times [0,N]}: \Delta \times [0,N] \rightarrow X^2\times [0,N]$ induces a homeomorphism to $\Delta_{X^2}\times [0,N]$.

    \item We have the following isomorphism of distinguished triangles:
    \begin{equation*}
     \begin{tikzcd}
\tnR\pi_{Z_N!}\bZ_{W_{N}} \arrow[r] \arrow[d] & {\tnR\pi_{Z_N!}\bZ_{\Delta \times [0,N]}} \arrow[r] \arrow[d] & {\tnR\pi_{Z_N!}\bZ_{W_{N}\setminus \Delta \times [0,N]}[1]} \arrow[r, "+1"] \arrow[d] & {} \\
{P_U|_{[t_0,N]}} \arrow[r]               & {\bZ_{\Delta_{X^2} \times [0,N]}} \arrow[r]                    & {Q_U|_{[t_0,N]}} \arrow[r, "+1"]                                                       & {.}
\end{tikzcd}   
    \end{equation*}
    \end{itemize}
\end{enumerate}
\end{Def}
\begin{RMK}
Confirming whether an open set $U$ admits well-behaved microlocal kernels can be challenging. However, it holds true for many interesting examples. For convex toric domains, this can be established using the generating function description provided in \cite[Subsection 3.1]{Capacities2021}. Regarding the unit disk bundle, we will demonstrate this fact in \autoref{kernel of unit disk bundle}. Generally, it is plausible to consider the well-behaved property as being generic for microlocal kernels of bounded open sets. In fact, the PC-type property is preserved by the convolution and composition operations of sheaves. Since the GKS quantization is locally constructed by convolutions of constant sheaves, it is highly probable that, under suitable compactness conditions, we can construct a GKS quantization that yields well-behaved microlocal kernels. 
\end{RMK}
From here until the end of this subsection, we make the assumption that $U$ admits well-behaved microlocal kernels. We will now explain how to define and compute $F_\bullet^{S^1}(U,\bZ)_T$ and $F^{S^1,out}\bullet(U,\bZ)_T$ in an ad-hoc manner. Typically, when $T$ is fixed, we take a specific $N$, which determines $Z_N$ and $W_N$. Consequently, the notation $Z=Z_N$, $W=W_N$, and $\pi_{Z!}=\pi_{Z_N!}$ will not cause much confusion.
 
For $n\geq 1$, consider 
\[\mathcal{W}_n= s_t^n \Tilde{\Delta}^{-1}_{n,Z}\left(W^n 
 \cap \{\sum_i t_i\leq T\}\right)\subset X^n  \times Z^n \times \bR,\]
where 
\[\Tilde{\Delta}_{n,Z}=\Tilde{\Delta}_{n}\times \id_{Z^n}: X^n\times Z^n \times \bR^n \rightarrow X^{2n}\times Z^n \times \bR^n.\]
It is direct to check that 
\begin{align*}
\mathcal{W}_n\qquad \qquad & \rightarrow s_t^{n+1} \Tilde{\Delta}^{-1}_{n+1,Z}\left((\Delta\times \{0\} ) \times W^{n}\cap \{\sum_i t_i\leq T\}\right),\\
 (x_1,\cdots,x_n,w_1,\cdots,w_n,t)&\mapsto (x_1,x_1,x_2\cdots,x_n,z(x_1),w_1,\cdots,w_n,t)
\end{align*}
defines a homeomorphism. Therefore, we define the following closed embedding
\begin{align*}
d_n: \mathcal{W}_n\rightarrow s_t^{n+1} \Tilde{\Delta}^{-1}_{n+1,Z}\left((\Delta\times \{0\} ) \times W^{n}\cap \{\sum_i t_i\leq T\}\right) \\ \rightarrow  s_t^{n+1} \Tilde{\Delta}^{-1}_{n+1,Z}\left((\Delta\times [0,N] ) \times W^{n}\cap \{\sum_i t_i\leq T\}\right) \rightarrow \mathcal{W}_{n+1}.
\end{align*}
Similarly, we can permute the positions of $\Delta\times [0,N]$ to obtain some closed inclusions $d_i$ for $i\in [n-1]_0$. We also denote $t_n$ as the cyclic permutation on $\mathcal{W}_n$. By doing so, we obtain a pre-cocyclic space denoted as $\mathcal{W}_\bullet=( \mathcal{W}_{n+1},d_i,t_n)_{n\in \bN_0}$.

Similarly, by replacing $W$ with $\Delta\times [0,N]$, we can define a pre-cocyclic space $\Delta_{\bullet}$ and a pre-cyclic chain complex $C_c^*(\Delta_{\bullet},\bZ)$. The closed inclusion $\Delta\times [0,N]\subset W$ induces a closed embedding of pre-cocyclic spaces $\Delta_{\bullet}\rightarrow \mathcal{W}_\bullet$. As a result, we obtain a chain map of pre-cyclic complexes:
\[C_c^*(\mathcal{W}_\bullet,\bZ) \rightarrow C_c^*(\Delta_{\bullet},\bZ).\]

\begin{RMK}Here, we choose the Alexander-Spanier cochain complex, as it naturally computes the hypercohomology of the Godement resolution of the constant sheaf, as explained in \cite[Chapter 4, Section 14]{demailly1997complex}. In particular, we have a commutative diagram that arises tautologically:
\begin{equation}\label{equation: diagram compute S1}
 \begin{tikzcd}
{C_c^*(\mathcal{W}_\bullet,\bZ)} \arrow[r] \arrow[d, "="] & {C_c^*(\Delta_{\bullet},\bZ)} \arrow[d, "="] \\
{\tnR\Gamma_c(\mathcal{W}_\bullet,\bZ_{\mathcal{W}_\bullet})} \arrow[r]           & {\tnR\Gamma_c(\Delta_\bullet,\bZ_{\Delta_\bullet})}     ,       
\end{tikzcd}
\end{equation}
where ${\tnR\Gamma_c(\mathcal{W}_\bullet,\bZ_{\mathcal{W}_\bullet})}$ denotes the pre-cyclic complex $({\tnR\Gamma(\mathcal{W}_n,\bZ_{\mathcal{W}_n}}))_n$ computed using corresponding the Godement resolutions.

Additionally, as all $\mathcal{W}_n$ is ENR, the Alexander-Spanier cohomology, and the singular cohomology are isomorphic with or without support conditions.
\end{RMK}

We have the following computational result for $S^1$-equivariant Chiu-Tamarkin invariant.
\begin{Prop}\label{prop: computation}If $\bK$ is a field and $T\geq 0$, then for an open set $U$ admitting well-behaved microlocal kernels, with the notations as above, we have
\[ CC^*(C_c^*(\mathcal{W}_\bullet,\bK)) \cong F^{S^1}_{\bullet}(U,\bK)_T,\qquad \in D_{dg}( \bK[\epsilon]-\Mod).\]
\end{Prop}
\begin{proof}Recall that $F^{S^1}_{\bullet}(U,\bK)_T$ is defined as the cocone of
\[F^{S^1}_{\bullet}(T^*X,\bK)_T \rightarrow F^{S^1,out}_{\bullet}(U,\bK)_T.\]
It is direct to see that 
\[F^{S^1}_{\bullet}(T^*X,\bK)_T \cong CC^*(C_c^*(\Delta_\bullet,\bK)).\]
On the other hand, thanks to the well-behaved assumption, we can consider the Godement resolution $\bK_{X^2\times Z \times [t_0,N]} \rightarrow A$. This gives us a c-soft chain model:
\[Q_N\Big|_{[t_0,N]} = \pi_{Z!}(A\otimes [\bK_W\rightarrow \bK_{\Delta\times [0,N]}]),\]
where $[\bK_W\rightarrow \bK_{\Delta\times [0,N]}]$ means the chain complex that represents the cone of the natural morphism. Consequently, we have
\begin{align*}
    F^{out}_{n}(U,\bK)_T \cong & \pi_{n !}(\Tilde{\Delta}_{n}^{-1}s_{t!}^n(\pi_{Z!}(A\otimes [\bK_W\rightarrow \bK_{\Delta\times [0,N]}])^{\boxtimes n}))_T\\
     \cong & \Gamma_c\left((X^2\times Z)^n\times [0,N], s_{t!}\Tilde{\Delta}_{n,Z}^{-1}(\bK_{\{\sum_i t_i\leq T\}} \otimes A^{\boxtimes n} \otimes [\bK_W\rightarrow \bK_{\Delta\times [0,N]}]^{\boxtimes n})\right)\\
     \cong & \Gamma_c\left((X\,\times\, Z)^n\times [0,N], s_{t!}(\bK_{\{\sum_i t_i\leq T\}} \otimes \Tilde{\Delta}_{n,Z}^{-1}A^{\boxtimes n} \otimes \Tilde{\Delta}_{n,Z}^{-1}[\bK_W\rightarrow \bK_{\Delta\times [0,N]}]^{\boxtimes n})\right).
\end{align*}
Notice that ${\Delta}_n=s_{t}\Tilde{\Delta}_{n,Z}^{-1}(\Delta\times [0,N])^{ n} \cap \{\sum_i t_i\leq T\} $. Then we have the following natural morphisms for each $n$,
\begin{equation}\label{equation: two shitty morphisms}
  \bK_{\mathcal{W}_n}\rightarrow \bK_{{\Delta}_n} \rightarrow s_{t!}\Tilde{\Delta}_{n,Z}^{-1}(\bK_{\{\sum_i t_i\leq T\}} \otimes[\bK_W\rightarrow \bK_{\Delta\times [0,N]}]^{\boxtimes n}),  
\end{equation}
and their composition vanishes on the cohomology level. In particular, they induce two morphisms in the mixed derived category, say:
\[CC^*(\overline{C}_c^*(\mathcal{W}_\bullet,\bK)) \rightarrow F^{S^1}_{\bullet}(T^*X,\bK)_T \rightarrow F^{S^1,out}_{\bullet}(U,\bK)_T,\]
where $\overline{C}_c^*(\mathcal{W}_\bullet,\bK)$ denotes the derived section computed using the resolution $\bK_{(X^2\times Z_N\times [t_0,N])^n} \rightarrow A^{\boxtimes n}$. Their composition vanishes $0$ in $D_{dg}(\bK[\epsilon]-\Mod)$.

Then there exists a unique morphism in $D_{dg}(\bK[\epsilon]-\Mod)$:
\[ CC^*(\overline{C}_c^*(\mathcal{W}_\bullet,\bK)) \rightarrow F^{S^1}_{\bullet}(U,\bK)_T.\]
To establish that this morphism is an isomorphism, we only need to verify that it is an isomorphism in $D_{dg}(\bK-\Mod)$ since the forgetful functor \eqref{equation: forgetful functor} is conservative. We can follow a similar discussion as in \autoref{section: gysin} and test two things: i) We can reduce the problem to the 0-skeleton of $\overline{C}_c^*(\mathcal{W}_\bullet,\bK)$, which is proved using that $P_U$ is idempotent and $P_U\cong \pi_{Z!}(A_W)$. ii) On the 0-skeletons, it is evident that $F^{S^1}_{\bullet}(U,\bK)_T$ is isomorphic to $\tnR\Gamma_c(\mathcal{W}_1,\bK)$, which is also isomorphic to $F_1(U,\bK)_T$. Therefore, we have the isomorphism:
\[CC^*(\overline{C}_c^*(\mathcal{W}_\bullet,\bK)) \cong F^{S^1}_{\bullet}(U,\bK)_T\quad \in D_{dg}(\bK[\epsilon]-\Mod).\]

Lastly, we need to compare $CC^*(\overline{C}_c^*(\mathcal{W}_\bullet,\bK))$ and $CC^*(C_c^*(\mathcal{W}_\bullet,\bK))$. The proof of the Kunneth formula for Alexander-Spanier cohomology (see \cite[Théorème 6.2.1]{Godement1960TopologieAE}) shows that we can find a suitable resolution $A$ such that it is multiplicative. Moreover, for the Godement resolution $\bK_{(X^2\times Z_N\times [t_0,N])^n} \rightarrow B_n$, there exists a commutative diagram of quasi-isomorphisms:
\begin{equation*}
    \begin{tikzcd}
                           & \bK_{(X^2\times Z_N\times [0,N])^n} \arrow[ld] \arrow[rd] &     \\
A^{\boxtimes n} \arrow[rr] &                                               & B_n.
\end{tikzcd} 
\end{equation*}
Then we only need to verify that the morphisms of resolutions are compatible with the morphisms induced by closed inclusions, which is a straightforward task.
\end{proof}
Therefore, we can provide an ad-hoc definition of the $S^1$-equivariant Chiu-Tamarkin invariant over any commutative ring.
\begin{Def}\label{Def: Z def}For a commutative ring $\bK$, $T\geq 0$ and an open set $U\subset T^*X$ admitting well-behaved microlocal kernels, we define
\[F_{\bullet}^{S^1}(U,\bK)_T\coloneqq \textnormal{CC}^*(C_c^*(\mathcal{W}_\bullet,\bK))  \in D_{dg}(\bK[\epsilon]-\Mod).\]  
We also define Chiu-Tamarkin invariants over $\bK$ as follows:
\begin{align*}
C^{S^1}_{T}(U,\bK)\coloneqq& \RHOM_{\bK[\epsilon]}(F_{\bullet}^{S^1}(U,\bK)_{T},\bK[-d])\in D_{dg}(\bK[\epsilon]-\Mod).
\end{align*} 
\end{Def}
\begin{Prop}\label{prop: computation 2}
For a commutative ring $\bK$, $T\geq 0$, and an open set $U$ that admitting well-behaved microlocal kernels, the object $F_{\bullet}^{S^1}(U,\bK)_T$ computes the compactly supported $S^1$-equivariant cohomology of the fat geometric realization $|\mathcal{W}_\bullet|$ of the pre-cocyclic space $\mathcal{W}_\bullet$.

The $S^1$-equivariant Chiu-Tamarkin invariant $C_T^{S^1}(U,\bK)$ computes the Borel-Moore $S^1$-equivariant homology of the fat geometric realization $|\mathcal{W}_\bullet|$ of the pre-cocyclic space $\mathcal{W}_\bullet$.\end{Prop}
\begin{proof}The proof follows a similar approach as in \cite{Jones1987} or \cite[Theorem 7.2.3]{LodayCyclic}. Although those references consider cyclic spaces, the technical arguments can be adapted to pre-cocyclic spaces without significant modifications.
\end{proof}
Regarding the restriction morphisms defined in \autoref{subsection: Restriction morphisms}, we obtain the following result:
\begin{Coro}For an open set $U$ admitting well-behaved microlocal kernels, we have the following commutative diagram:
\begin{equation*}
    \begin{tikzcd}
{H^qC_T^{S^1}(U,\bK)} \arrow[d,"\cong"] \arrow[r]        & {H^qC_T^{\bZ/\ell}(U,\bK)} \arrow[d,"\cong"]   \\
{H_{d-q}^{S^1}(|\mathcal{W}_\bullet|,\bK)} \arrow[r] & {H_{d-q}^{\bZ/\ell}(\mathcal{W}_\ell,\bK)}
\end{tikzcd}
\end{equation*}   
\end{Coro}
Based on the topological description, we can state the following proposition, which follows from \cite[Appendix]{viterbo1997}:
\begin{Prop}\label{prop: finite field reduction}For an open set $U$ admitting well-behaved microlocal kernels, $\ell\geq 3$, $T\geq 0$ and any prime factor $\mu$ of $\ell$, if $H^{*}C^{S^1}_T(U,\bZ)$ has no $\mu$-torsion as an abelian group, then we have
\begin{align*}
    \begin{aligned}
         H^{*}C^{\bZ/\ell}_{T}(U,\bZ/\mu)& \cong H^{*}C^{S^1}_{T}(U,\bZ)\otimes H^*(S^1,\bZ/\mu) , \\
        H^{*}C^{S^1}_{T}(U,\bQ)& \cong H^{*}C^{S^1}_{T}(U,\bZ)\otimes\bQ.
    \end{aligned}
\end{align*}
\end{Prop}

\begin{eg}\label{example: convex toric domain}By combining the results of \autoref{prop: computation} and \autoref{prop: computation 2}, we can compute $H^*C_T^{S^1}(X_\Omega,\bK)$ for a convex toric domain $X_\Omega$. For additional notations and more details, we refer to \cite[Section 4.1]{zhangthesis}.

In particular, we have the $S^1$-version of the structure theorem for the Chiu-Tamarkin invariant. Specifically, for $T\geq 0$ we have
\begin{itemize}[fullwidth]
\item For each $Z\in \Omega^\circ_T$, the inclusion of the segment $\overline{OZ} \subset \Omega^\circ_T$ induces a decomposition of the fundamental class $\eta_{T}^{S^1}(X_{\Omega},\bQ)=u^{I(Z)}\Lambda_{Z}^{S^1}$ for a non-torsion element $\Lambda_{Z}^{S^1}\in H^{-2I(Z)}C_{T}^{S^1}(X_\Omega,\bQ)$. In particular, $\eta_{T}^{S^1}(X_{\Omega},\bQ)$ is non-zero.
\item The minimal cohomology degree of $H^*C_{T}^{S^1}(X_\Omega,\bQ)$ is exactly $-2I(\Omega^\circ_T)$, i.e.,
\[ H^*C_{T}^{S^1}(X_\Omega,\bQ)\cong  H^{\geq -2I(\Omega^\circ_T)}C_{T}^{S^1}(X_\Omega,\bQ), \] 
    and
\[ H^{-2I(\Omega^\circ_T)}C_{T}^{S^1}(X_\Omega,\bQ)\neq 0 . \] 
\item $H^*C_{T}^{S^1}(X_\Omega,\bQ)$ is a finitely generated $\bQ[u]$-module. The free part is isomorphic to $A=\bQ[u]$, so $H^*C_{T}^{S^1}(X_\Omega,\bQ)$ is of rank $1$ over $\bQ[u]$. 

The torsion part is given by $H^*C_{T}^{S^1,out}(X_\Omega,\bQ)$, which is located in cohomology degree $[-2I(\Omega^\circ_T)+1,-1]$.  $H^*C_{T}^{S^1}(X_\Omega,\bQ)$ is torsion free when $X_\Omega$ is an open ellipsoid.
\end{itemize}
\end{eg}

\subsection{Capacities}\label{section: capacity}
In this subsection, we will explore the applications of our $S^1$-theory to symplectic capacities, including the construction of a sequence of symplectic capacities and comparisons with other numerical invariants that we have previously developed. Throughout this subsection, we assume that the manifold $X$ is orientable.

The basic constructions and proofs are similar to those in \cite{Capacities2021}, so we will omit most of the proofs and only highlight the differences. In this subsection, we will only consider rational coefficients, even though most of the results hold for all fields.

For an open set $U \subset T^*X$, we have a morphism (\autoref{invariance1}-(1)) in the mixed derived category:
\[C^{S^1}_T(T^*X,\bQ) \xrightarrow{i_U^*} C^{S^1}_T(U,\bQ),\]
which induces a morphism of $\textnormal{Ext}^*_{S^1}(\bQ,\bQ) \cong \bQ[u]$-modules on cohomology:
\[H^{BM}_{d-*}(X,\bQ)\otimes  \bQ[u]\cong H^*C^{S^1}_T(T^*X,\bQ) \xrightarrow{i_U^*} H^*C^{S^1}_T(U,\bQ).\]

Since $X$ is connected and orientable, we have the fundamental class $[X]$ of $X$ in $H^{BM}_d(X,\bQ)$, which is defined via $1 \in H^0(X,\bQ) \cong H^{BM}_d(X,\bQ)$. We set $[X]^{S^1} = [X] \otimes 1$, where $1 \in \bQ[u]$ is the identity element.
\begin{Def}\label{definition of fundamental class}For an open set $U\xhookrightarrow{i_U} T^*X$ and $T\geq 0$, we define its {\em fundamental class} $\eta_T^{S^1}(U,\bQ)$ as the image of $[X]^{S^1}$ under $i_U^*$, i.e., $\eta^{S^1}_T(U,\bQ)\coloneqq i_U^*([X]^{S^1})\in H^0C^{S^1}_T(U,\bQ)$.
\end{Def}

As a corollary of \autoref{invariance1}, we have
\begin{Prop}\label{functorial fundamental class} \begin{enumerate}[fullwidth]
    \item Let $U\subset U' \subset T^*X$ be an inclusion of open sets. Through the natural morphism
\[H^0{C^{S^1}_T(U',\bQ)}\rightarrow H^0{C^{S^1}_T(U,\bQ)}\]
we have
\[\eta^{S^1}_T(U',\bQ) \mapsto \eta^{S^1}_T(U,\bQ).\]
\item Let $\varphi:T^*X \rightarrow T^*X$ be a compactly supported Hamiltonian homeomorphism and $U$ be an open set. For the isomorphism, defined in \autoref{invariance1},  \[H^*(\Phi^{S^1}_{T}): H^*C^{S^1}_T(U,\bQ) \xrightarrow{\cong} H^*C^{S^1}_T(\varphi(U),\bQ),\]
we have $H^0(\Phi^{S^1}_{T})(\eta^{S^1}_T(\varphi(U),\bQ))=\eta^{S^1}_T(U,\bQ)$.
\end{enumerate}
\end{Prop}
We have $\eta^{S^1}_T(T^*X,\bQ)=[X]^{S^1}$ for all $T\geq 0$. So, if there exists an open set $X'\subset X$ such that $U\subset T^*X'\subset T^*X$, we have $\eta_T(U,\bQ)=i_U^*([X]^{S^1})=i_U^*([X']^{S^1})$ by \autoref{functorial fundamental class}-(1).

\begin{Def}\label{definition of capacities}For an open set $U$ and $k\in \bN$ we define
\begin{equation*}\label{definition of c_k} \textnormal{Spec}(U,k) \coloneqq
\left\lbrace
  T \geq 0:\begin{aligned} \eta^{S^1}_T (U,\bQ)  \in u^kH^{*}C^{S^1}_T(U,\bQ)
  \end{aligned}
\right\rbrace ,
\end{equation*}
and
\begin{equation}
 \overline{c}_k(U)\coloneqq \inf \textnormal{Spec}(U,k) \in [0,+\infty ].
\end{equation}
\end{Def}
Similar to \cite[Theorem 2.23]{Capacities2021}, the functions $\overline{c}_k$ define a sequence of non-trivial symplectic capacities.
\begin{Thm}\label{capacity property symplectic} The functions $\overline{c}_k:\text{Open}(T^*X)\rightarrow [0,\infty]$ satisfy the following:
\begin{enumerate}[fullwidth]
    \item $\overline{c}_k \leq \overline{c}_{k+1}$ for all $k\in \bN$.
    \item For two open sets $U_1 \subset U_2$, we have $\overline{c}_k(U_1) \leq \overline{c}_k(U_2)$.
    \item For a compactly supported Hamiltonian homeomorphism $\varphi: T^*X \rightarrow T^*X$, we have 
    $\overline{c}_k(U)=\overline{c}_k(\varphi(U)).$ 
     \item If $X=\bR^d$, then $\overline{c}_k(rU)=r^2\overline{c}_k(U)$ for all $k\in \bN$ and $r>0$.
    \item If $U=\{H<1\}$ is bounded, and $\partial U=\{H=1\}$ is a non-degenerated hypersurface of restricted contact type defined by a Hamiltonian function $H$. If $\overline{c}_k(U) < \infty $, then $\overline{c}_k(U)$ is represented by the action of a closed characteristic in the boundary  $\partial U$.
    
    \item $\overline{c}_k(U)>0$ for all open sets $U$.
\end{enumerate}
\end{Thm}
The proof is the same as \cite[Theorem 2.23]{Capacities2021}, except for point (5). Regardding on (5), the proof of the corresponding proposition in \cite{Capacities2021} relies on a microsupport estimate of the sheaf $F_\ell(U,\bK)$, which is not available in our case. However, we can replace the microsupport estimate with the following \autoref{S1 actionspectrumestimate} based on the persistence structure. With this replacement, the theorem follows.
\begin{Lemma}\label{S1 actionspectrumestimate}For a bounded open set $U=\{H<1\}$ defined by a Hamiltonian function $H$ such that $\partial U=\{H=1\}$ is a non-degenerated hypersurface of restricted contact type, and for $T'>T\geq 0$, if 
\[\left\lbrace t\in \bR: t= \bigg|\int_c \bp d\bq \bigg|\text{ for a closed orbit }  c\text{ of }\varphi_z^H \text{ in } \partial U \right\rbrace \cap (T,T'] =\varnothing,\]
then the structural morphism \eqref{equation: structure maps of persistence module} of the persistence modules $F^{S^1}_\bullet(U,\bQ)_T$:
\[F^{S^1}_\bullet(U,\bQ)_{T'}\rightarrow F^{S^1}_\bullet(U,\bQ)_T,\]
is an isomorphism in $D_{dg}(\bQ[\epsilon]-\Mod)$.
\end{Lemma}
\begin{proof}We still use the fact that the forgetful functor \eqref{equation: forgetful functor} is conservative. Then the structural morphism \eqref{equation: structure maps of persistence module} is an isomorphism if and only if it is an isomorphism in $D_{dg}(\bQ-\Mod)$. However, in $D_{dg}(\bQ-\Mod)$, the structural morphism is the same as the structural morphism of the persistence modules 
\[F_1(U,\bQ)_{T'}\rightarrow F_1(U,\bQ)_T\qquad \in D_{dg}(\bQ-\Mod).\]
Then the result follows from \cite[Lemma 2.19]{Capacities2021}.
\end{proof}

The last part of this subsection presents two comparison results concerning various numerical invariants defined using sheaves.

The first result concerns the comparison between $\overline{c}_1$ and the sheaf energy $e$ (see \autoref{sheafenergy}). Precisely, we have
\begin{Prop}\label{prop: c1=e}For an open set $U$, we have
\[\overline{c}_1(U)= e(P_U),\]
where $P_U$ is over $\bQ$.
\end{Prop}
\begin{proof}
Under the isomorphism in \autoref{non-equivariant CT and yoneda form}, we can identify the non-equivariant fundamental class $\eta_T(U,\bQ)$ (see \cite[Definition 2.20]{Capacities2021}) with the natural morphism $\tau_T(P_U)$ that we defined in \autoref{section: reeb action}. Therefore, we have the relation:
\[e(P_U)=\inf\{T\geq 0: \eta_T(U,\bQ)=0\}.\]
On the other hand, consider the forgetful map in \autoref{subsection: Restriction morphisms}:
\begin{equation}\label{eq: forgetful capacity}
  H^{p}C_{T}^{S^1 }(U,\bQ)\rightarrow H^pC _{T}(U,\bQ).  
\end{equation}
It can be verified that
\[\eta^{S^1}_T(U,\bQ)\rightarrow \eta_T(U,\bQ)\]
under the map \eqref{eq: forgetful capacity}, and the class $u\in \textnormal{Ext}^2_{S^1}(\bQ,\bQ)$ is mapped to $0$.

Therefore, by the Gysin sequence \eqref{Gysin long exact sequence}, we conclude that $\eta_{T}(U,\bQ)=0$ if and only if $\eta_{T}^{S^1}(U,\bQ)\in 
 uH^*C^{S^1}_T(U,\bQ)$. Consequently, we obtain the equality $\overline{c}_1(U)= e(P_U)$.
\end{proof}
The second result concerns comparing the $S^1$-equivariant capacities $\overline{c}_k$ defined in this article with the $\bZ/\ell$-equivariant capacities $c_k$ defined in \cite{Capacities2021}. Specifically, we aim to show that
\begin{Thm}\label{c_k=cbar_k}For an open set $U$ admitting well-behaved microlocal kernels, if $H^qC^{S^1}_T(U,\bZ)$ is finitely generated as abelian groups for all $T\geq 0$ and all $q$, then we have
\[\overline{c}_k(U)=c_k(U).\]
\end{Thm}

\begin{proof}
Since $U$ admits well-behaved microlocal kernels, we can define $H^qC_T^{S^1}(U,\bZ)$ and $\eta^{S^1}_T(U,\bZ)$. Although we will not verify all the properties it deserves at the moment, we can proceed with these definitions.

If the conditions of \autoref{prop: finite field reduction} hold, it is straightforward to observe that the fundamental class satisfies the following properties:
\begin{equation}\label{eq: finite field reduction-fundamental class}
    \eta^{\bZ/\mu}_T(U,\bZ/\mu)=\eta^{S^1}_T(U,\bZ)\otimes 1_{\bZ/\mu},\qquad \eta^{S^1}_T(U,\bQ)=\eta^{S^1}_T(U,\bZ)\otimes 1_{\bQ}.
\end{equation}
Then we only need to choose $\ell = \mu$ to be a sufficiently large odd prime number to ensure that $H^qC_T^{S^1}(U,\bZ)$ has no $\mu$-torsion. This is possible because we assume that $H^qC_T^{S^1}(U,\bZ)$ is a finitely generated abelian group. Consequently, the result follows from:
\begin{align*}
    &\eta^{S^1}_T(U,\bZ) \in u^kH^*C_T^{S^1}(U,\bZ)\\
    \iff  &\eta^{S^1}_T(U,\bQ) \in u^kH^*C_T^{S^1}(U,\bQ)\\
     \iff  &\eta^{\bZ/\ell}_T(U,\bZ/\ell) \in u^kH^*C_T^{\bZ/\ell}(U,\bZ/\ell).\qedhere
\end{align*}
\end{proof}For example, we know that convex toric domains admit well-behaved microlocal kernels, and our computation in \autoref{example: convex toric domain} shows that $H^qC_T^{S^1}(X_\Omega,\bZ)$ is finitely generated. Therefore, for a convex toric domain $X_\Omega$, we have $\overline{c}_k(X\Omega) = c_k(X_\Omega)$. 

We have already derived a combinatorial formula for $c_k(X_\Omega)$ in \cite[Theorem 3.7]{Capacities2021}. Therefore, this formula is also valid for $\overline{c}_k(X_\Omega)$ now. We would like to emphasize that the computation for $\overline{c}_k(X_\Omega)$ can also be directly obtained from \autoref{example: convex toric domain}.

\section{Viterbo isomorphism} 
Our goal in this section is to compute the Chiu-Tamarkin invariant of unit disk bundles. We begin from a precise formula for the GKS sheaf quantization of the normalized geodesic flow. As a result, we show that unit disk bundles admit well-behaved microlocal kernels in \autoref{kernel of unit disk bundle}. This allows us to apply the results from \autoref{section: Z coefficient}. Furthermore, we will also consider the product structure in our computations.

For a Riemannian manifold $(X,g)$, the open unit (co)disk bundle is $D^*X=\{(\bq,\bp):|\bp|_g<1\}$. Let us take $H(\bq,\bp)=|\bp|_g$, then $D^*X=\{(\bq,\bp):H<1\}$. We denote $d$ the distance induced by $g$. Assume $(X,g)$ is a complete Riemannian manifold and the convex radius $r_{\text{conv}}(X,g)>2$, then the injective radius $r_{\text{inj}}(X,g)>4$. Such a metric $g$ always exists by introducing a suitable conformal factor to an arbitrary metric.

\subsection{Sheaf quantization of the geodesic flow}
For the homogeneous Hamiltonian function $H(\bq,\bp)=|\bp|_g$, the associated $\bR_{>0}$-equivariant Hamiltonian flow is the normalized geodesic flow $\varphi_z^{\text{geo}}$ (where we identify $T^*X$ and $TX$ using the metric $g$). It is known that there exists a GKS sheaf quantization $K_g\in D(\mathbb{R}_z\times X^2)$ that quantizes the geodesic flow in the sense
\[\dot{SS}(K_g)=\left\lbrace (z, - {H_z}\circ{\varphi}_z^{\text{geo}}(\bq,\bp)  ,(\bq,-\bp),{\varphi}_z^{\text{geo}}(\bq,\bp) ) : (\bq,\bp) \in \dot{T}^*X, z\in I \right \rbrace. \]
Below, we provide an explicit formula for $K_g$, which is known among experts, but the author was unable to locate a specific reference.

In\cite{Thickeningkernel}, the authors show that if we restrict to $I = (-2,2) \subset \bR_z$, we have
\begin{equation}\label{small time geodesic flow}
    \bK_{(z,\bq_1,\bq_2): \{d(\bq_1,\bq_2)<z\}}\dotimes q_2^{-1}\omega_{X }\rightarrow K_g|_{(-2,2)_z} \rightarrow \bK_{\{(z,\bq_1,\bq_2): d(\bq_1,\bq_2)\leq -z\}}\xrightarrow{+1}.
\end{equation}
In the following, we need to extend the formula to larger values of $z$. Since $H$ is autonomous, at the flow level, we have $\varphi_{z_1+z_2}=\varphi_{z_1}\circ\varphi_{z_2}$. This property also holds for the kernel $K_{g,z}\coloneqq K_g|_z$, meaning that for $z_1,z_2\in \bR$, we have $K_{g,z_1}\circ K_{g,z_2}\cong K_{g,z_1+z_2}$, where $\circ$ is the sheaf composition (\cite[Subsection 1.6]{GKS2012}). In particular, for $N\in \bN$, we have
\begin{align*}
    K_{g,Nz}\cong & K_{g,z}\circ \cdots \circ K_{g,z}.
\end{align*}
To obtain the family version of the slicewise formula, we need to use the relative composition denoted as $\circ_I$, which is defined in \cite[(1.13)]{GKS2012}. Let us denote $K_g^1 = K_g|_{(-2,2)_z}$, and then consider
\begin{equation}
    {K}_{g}^N\cong r_{N}^{-1}(K_{g}^1\circ_I \cdots \circ_I K_{g}^1), 
\end{equation}
where $r_N(z,\bq_1,\bq_2)=(z/N,\bq_1,\bq_2)$.

Then we have the following microsupport estimates using \cite[(1.15)]{GKS2012}:
\begin{equation*}
    \dot{SS} (K_{g}^N) \subset \{(z,-|p|_g,\bq,-\bp,\bq',\bp'): (\bq',\bp')=\varphi^{\text{geo}}_{z}(\bq,\bp),z\in (-2N,2N)\}.
\end{equation*}
Besides, it is direct to verify that $K_{g}^N|_{\{z=0\}}\cong \bK_{\Delta_{X^2}}$.
Then the uniqueness part of the GKS quantization shows
\begin{equation}
    K_g|_{(-2N,2N)_z}\cong K_{g}^N.
\end{equation} 
In particular,
\begin{align}\label{sheaf of geodesic flow along negative time}
\begin{aligned}
     K_g|_{(-2N,0]}  \cong{\tnR }\pi_{(\bq_1,\dots,\bq_{N-1})!}\bK_{\mathcal{M}^{N}X} \eqqcolon K_{g,-}^N ,
\end{aligned}
  \end{align}
  where 
  \begin{equation}
      \mathcal{M}^{N}X={\{(z,\bq_0,\dots,\bq_N):d(\bq_i,\bq_{i+1})\leq - z/N,\, i\in [N-1]_0,\,-2N<z\leq 0\}}
  \end{equation}
   is the discrete Moore path space and $[N-1]_0=\{0,1,\dots,N-1\}$.

\subsection{Microlocal kernel of unit disk bundles}
To compute $P_{D^*X}$, we can use \cite[Proposition 2.7]{Capacities2021}. To apply this proposition, we require a sheaf quantization of the normalized geodesic flow in the sense of \cite[(2.3)]{Capacities2021}. A natural choice for this quantization is $\cK_g = K_g \boxtimes \bK_{[0,\infty)}$. However, the normalized geodesic flow cannot be extended to the zero-section as a $C^\infty$-Hamiltonian flow. Luckily, the flow preserves the zero section and can be extended to the zero-section as a $C^0$-Hamiltonian flow, i.e. a Hamiltonian homeomorphism. The $C^0$-natural of microsupport makes sure that it makes sense to think of $\cK_g$ as a sheaf quantization of the normalized geodesic flow in the sense of \cite[(2.3)]{Capacities2021}. Furthermore, it has been shown in \cite{Chiu-kernel2021} that the completeness of the flow is sufficient to establish the proof of \cite[Proposition 2.7]{Capacities2021}. In particular, the geodesic flow is tautologically complete due to our assumption for the completeness of the Riemannian metric $g$. Therefore, we can proceed with the computation of $P_{D^*X}$ using the aforementioned tools.

Therefore, for $\Omega=\{\zeta <1\}$ and $\widetriangle{\bK_{\Omega}}=\bK_{\{(z,t):-t\leq z \leq 0\}}$, we have 
\[P_{D^*X}\cong   \widehat{\cK_g} \circ \bK_{(-\infty,1)}\cong \cK_g \star\bK_{\{(z,t):-t\leq z \leq 0\}} \cong K_g\circ \bK_{\{(z,t):-t\leq z \leq 0\}}.  \]If we restrict $P_{D^*X}$ to $t\leq N$ for $N\in \bN$, we have 
\[P_{D^*X}|_{\{t\leq N\}}\cong K_g\circ \bK_{\{(z,t):-N \leq -t\leq z \leq 0\}}\cong (K_{g})_{\{-N \leq z\leq 0\}} \circ  \bK_{\{(z,t):-N \leq -t\leq z \leq 0\}}.\]
Recall \eqref{sheaf of geodesic flow along negative time}, we can take 
\[     K_g|_{(-2N,0]_z}\cong K_{g,-}^N\cong {\tnR } \pi_{(\bq_1,\dots,\bq_{N-1})!}\bK_{\mathcal{M}^{N}X},\]
where \begin{equation*}
      \mathcal{M}^{N}X={\{(z,\bq_0,\dots,\bq_N):d(\bq_i,\bq_{i+1})\leq - z/N,\, i\in [N-1]_0,\,-2N<z\leq 0\}}
  \end{equation*}
   is the discrete Moore path space and $[N-1]_0=\{0,1,\dots,N-1\}$.
Therefore, we have
\begin{align*}
P_{D^*X}|_{\{t \leq  N\}}&\cong\tnR  \pi_{z!} {\tnR } \pi_{(\bq_1,\dots,\bq_{N-1})!}\bK_{\widetilde{\mathcal{M}}^{N}X}\cong  {\tnR } \pi_{(\bq_1,\dots,\bq_{N-1})!}\tnR  \pi_{z!}\bK_{\widetilde{\mathcal{M}}^{N}X},
\end{align*}
where 
\[\widetilde{\mathcal{M}}^{N}X=\{(z,\bq_0,\dots,\bq_N,t):d(\bq_i,\bq_{i+1})\leq - z/N,\, i\in[N-1]_0,\,-N\leq -t\leq z\leq 0\}.\]
The restriction of the projection $\pi_z(z,\bq_0,\dots,\bq_N,t)=(\bq_0,\dots,\bq_N,t)$ on $\widetilde{\mathcal{M}}^{N}X$ is proper, and its fibers are closed intervals. This property allows us to apply the Vietoris-Begle theorem. Therefore, we can conclude that
\begin{equation*}
    \tnR  \pi_{z!}\bK_{\widetilde{\mathcal{M}}^{N}X} \cong  \bK_{\mathcal{M}_0^N X}, 
\end{equation*}
where
\begin{equation}
\begin{split}
 &\mathcal{M}_0^N X\coloneqq \pi_{z}(\widetilde{\mathcal{M}}^{N}X) \\=&\lbrace(\bq_0,\dots,\bq_N,t):d(\bq_i,\bq_{i+1})\leq t/N,\, i\in[N-1]_0,\,0\leq t \leq N\rbrace\\
 \subset &\lbrace(\bq_0,\dots,\bq_N,t): d(\bq_i,\bq_{i+1})\leq 1,\,t\geq 0\rbrace.
 \end{split}
\end{equation}
Therefore, we conclude that  
\begin{Prop}\label{kernel of unit disk bundle}
For a complete Riemannian manifold $(X,g)$ and $N\in \bN$, the microlocal kernel of its open unit disk bundle $D^*X$ is given by
\[ P_{D^*X}|_{\{t \leq N\}}\cong  {\tnR } \pi_{(\bq_1,\dots,\bq_{N-1})!}\bK_{\mathcal{M}_0^N X}.\]
Moreover, one can verify that 
\begin{align*}
   {\bK}_{\Delta_{X^2} \times[0,N]}&\cong {\tnR } \pi_{(\bq_1,\dots,\bq_{N-1})!}\bK_{\Delta_{X^N}\times [0,N]},\\
   Q_{D^*X}|_{\{t \leq N\}}&\cong  {\tnR } \pi_{(\bq_1,\dots,\bq_{N-1})!}\bK_{\mathcal{M}_0^N X \setminus (\Delta_{X^N}\times [0,N])}[1],
\end{align*}
and the defining triangle is the image, under ${\tnR } \pi_{(\bq_1,\dots,\bq_{N-1})!}$, of the following excision triangle
\[    \bK_{\mathcal{M}_0^N X} \rightarrow \bK_{\Delta_{X^N}\times [0,N]}  \rightarrow \bK_{\mathcal{M}_0^N X \setminus (\Delta_{X^N}\times [0,N])}[1] \xrightarrow{+1}.\]
In other words, the unit disk bundle admits well-behaved microlocal kernels in the sense of \autoref{section: Z coefficient}.
\end{Prop}

\subsection{The Viterbo isomorphism}
\begin{Lemma}\label{lemma: comparison cyclic homology}If $f:X_\bullet \rightarrow Y_\bullet$ a morphism between pre-cocyclic spaces such that $f$ induces a levelwise isomorphism on cohomology, then it also induces an isomorphism
\[CC^*(C^*(Y_\bullet,\bK))\rightarrow CC^*(C^*(X_\bullet,\bK))\qquad \in D_{dg}(\bK[\epsilon]-\Mod).\]
\end{Lemma}
\begin{proof}We only need to prove the morphism is an isomorphism in $D_{dg}(\bK-\Mod)$ by virtue of \eqref{equation: forgetful functor}. Then the isomorphism follows from the comparison between the two spectral sequences induced by the skeleton filtration of the pre(semi)-simplicial complexes $C^*(Y_\bullet,\bK)$.
\end{proof}
Compare to notations in \autoref{section: Z coefficient}, we have
\[W_N={\mathcal{M}_0^N X} ,\qquad \Delta\times [0,N]=\Delta_{X^N}\times [0,N].\]
Then for $T\geq 0$, $T\leq n N$ with $N\gg 1$, we have
\begin{equation} 
   \mathcal{W}_n=\mathcal{L}^N_{n}X \coloneqq\left\lbrace(\bq_i^j,t)_{i,j}:\begin{aligned}
  &d(\bq_i^j,\bq_{i+1}^j)\leq 1 ,\,\sum_{j=1}^{n} d(\bq_i^j,\bq_{i+1}^j)\leq t/N,\\ & i\in[N-1]_0,\, j\in \bZ/n,\, 0 \leq t \leq T 
\end{aligned}\right\rbrace,  
\end{equation}
where $j$ indicates points in the $j^{th}$-copy of $\mathcal{M}_0^N X$, and we require $\bq^j=\bq^j_N=\bq^{j+1}_0$. It forms a pre-cocyclic space $\mathcal{W}_{\bullet}=\mathcal{L}^N_{\bullet}X$. 

Then, by \autoref{Def: Z def}, we have
\[F^{S^1}_{\bullet}(D^*X,\bK)_T=CC^*(C^*_c(\mathcal{L}^N_{\bullet}X,\bK)).\]
To simplify computation, we consider
\begin{equation}\label{discretegeodesic space}  \mathcal{L}^N_{n,T}X \coloneqq \left\lbrace   (\bq_i^j)_{i,j}: \begin{aligned}
&d(\bq_i^j,\bq_{i+1}^j)\leq 1 ,\,\sum_{j=1}^{n} d(\bq_i^j,\bq_{i+1}^j)\leq T/N,\\ &i\in[N-1]_0,\, j\in \bZ/n
\end{aligned}
\right\rbrace,  
\end{equation}
and corresponding pre-cocyclic space $\mathcal{L}^N_{\bullet,T}X$. Obviously, the pre-cocyclic closed inclusion
\[\mathcal{L}^N_{\bullet,T}X\rightarrow \mathcal{L}^N_{\bullet}X,\qquad (\bq_i^j)_{i,j}\mapsto (\bq_i^j,T)_{i,j}\]
is a levelwise deformation retraction of pre-cocyclic spaces. Then it induces a morphism
\begin{equation}\label{equation: comparision between homology of two spaces}
    F^{S^1}_{\bullet}(D^*X,\bK)_T = CC^*(C_c^*(\mathcal{L}^N_{\bullet}X,\bK))  \rightarrow  CC^*(C_c^*(\mathcal{L}^N_{\bullet,T}X,\bK)) \in D_{dg}(\bK[\epsilon]-\Mod).
\end{equation}
Therefore, we can apply \autoref{lemma: comparison cyclic homology} to show that \eqref{equation: comparision between homology of two spaces} is an isomorphism in $D_{dg}(\bK[\epsilon]-\Mod)$. It is important to note that the argument used for cohomology in the lemma also holds for compactly supported cohomology.

Now, let us consider the free loop space $\mathcal{L}X=C^\infty(S^1,X)$. The free loop space equips an $S^1$-action given by $(e^{i\theta}\cdot c)(t)=c(t+\theta)$. We define the length function $L: \mathcal{L}X\rightarrow,\, c \mapsto \int_{c} |\dot{c}|_g$, which is $S^1$-invariant. Then
\[\mathcal{L}_{\leq T}X= \{c \in \mathcal{L}X: L(c) \leq T\}\]
is an $S^1$-space by the invariance of the length function. In particular, $\mathcal{L}_{\leq T}X$ is restricted to a $\bZ/n$-space, where the $\bZ/n$-action on $\mathcal{L}_{\leq T}X$ is defined by $(\sigma\cdot c)(t)=c(t+1/n)$ ($\sigma$ is the generator of $\bZ/n$). With this $\bZ/n$-action, we can regard $\mathcal{L}_{\leq T}X$ as a constant cocyclic space denoted by $\mathcal{L}_{\leq T}X_\bullet$. In this constant cocyclic space, the face and degeneracy maps are defined as identity maps, and the cyclic permutations are induced by the restriction $\bZ/n$-action.

Now, we assume that $X$ is compact. We observe that $\mathcal{L}_{\ell,T}^N X$ is also levelwise compact by its definition. Following the discussion in \cite{MilnorMorse}, we note that the closed subset $\mathcal{L}_{\leq T} X$ is homeomorphic to a closed subset of the finite-dimensional compact manifold $X^{N}$, where $N=\lfloor T/2\rfloor$. Consequently, $\mathcal{L}_{\leq T} X$ is compact, we can conclude that
\begin{Prop}\label{Zl loop space homotopy}We have a levelwise homotopy equivalence of pre-cocyclic spaces
\[{\mathcal{L}^N_{\bullet,T}X} \rightarrow \mathcal{L}_{\leq T}X_\bullet.\]
\end{Prop}
\begin{proof}For a fixed level $n\in \bN_0$ and we set $\ell=n+1$. Since $d(\bq_i^j,\bq_{i+1}^j)\leq 1$ for all $i,j$, there exists a unique minimal geodesic $c_{i}^j:[0,1/(N\ell)]\rightarrow X$ from $\bq_i^j$ to $\bq_{i+1}^j$. In particular, we have $L(c_i^j)=d(\bq_i^j,\bq_{i+1}^j)$. Now, let $c^j:[0,1/\ell]\rightarrow X$ be the concatenation of $c_i^j$ for all $i$, and $c: [0,1]\rightarrow X$ be the concatenation of $c^j$ for all $j$. We have $c(\frac{i+1}{N\ell}+\frac{j-1}{\ell})=\bq_i^j$. The condition $\sum_{j=1}^{\ell} d(\bq_i^j,\bq_{i+1}^j)\leq T/N$ ensures that $L(c)=\sum_{i,j}L(c_i^j)\leq \sum _i T/N=T$. The construction defines a piecewise geodesic map $(\bq_i^j)\mapsto c$ for each $\ell$. Moreover, the cyclic permutation here is given by $(\bq_i^j)\mapsto (\bq_i^{j+1})$ and $c^j\mapsto c^{j+1}$. Therefore, the piecewise geodesic map
\[{\mathcal{L}^N_{\ell,T}X} \rightarrow\mathcal{L}_{\leq T}X\]
is $\bZ/\ell$-equivariant. Moreover, it can be checked that piecewise geodesic maps commute with face maps. Therefore, we have a map of pre-cocyclic spaces. 

Finally, recall that the interpolation homotopy in \cite{MilnorMorse} can be taken piecewise on each $c^j$. Therefore, we can construct a $\bZ/\ell$-equivariant homotopy equivalence between the identity map and the piecewise geodesic map of each $\ell$. This completes the proof of the proposition.
\end{proof}
Now, since both ${\mathcal{L}^N_{\bullet,T}X}$ and $\mathcal{L}_{\leq T}X_\bullet$ are levelwise compact, the piecewise geodesic map is proper, and it induces a levelwise isomorphism on cohomology. Therefore, we can apply \autoref{lemma: comparison cyclic homology} to conclude that
\[H^qC_T(D^*X,\bK)\cong \textnormal{Ext}^{q-d}_{\bK[\epsilon]}(CC^*(C^*(\mathcal{L}^N_{\bullet,T}X,\bK)),\bK) \cong \textnormal{Ext}^{q-d}_{\bK[\epsilon]}(CC^*(C^*(\mathcal{L}_{\leq T}X_\bullet,\bK)),\bK).\]
Now, let us focus on the righthand side of the isomorphism. Using a pre-cyclic version of \cite[Theorem 7.2.3]{LodayCyclic}, we have that 
\[\textnormal{Ext}^{q-d}_{\bK[\epsilon]}(CC^*(C^*(\mathcal{L}_{\leq T}X_\bullet,\bK)),\bK) \cong H_{d-q}^{S^1}(|\mathcal{L}_{\leq T}X_\bullet|,\bK),\]
where $|-|$ stands for the {\em fat} geometric realization of pre-cocyclic spaces.

Please note that the {\em fat} geometric realization of $\mathcal{L}_{\leq T}X_\bullet$ is NOT homeomorphic to $\mathcal{L}_{\leq T}X$. However, $\mathcal{L}_{\leq T}X_\bullet$ is actually a constant cocyclic space, and therefore, the geometric realization of $\mathcal{L}_{\leq T}X_\bullet$ as a cocyclic space is homeomorphic to $\mathcal{L}_{\leq T}X$. Moreover, $\mathcal{L}_{\leq T}X_\bullet$ is a good cocyclic space, which means that all degenerate maps are closed inclusions. Then the fat geometric realization and the geometric realization of the cocyclic space $\mathcal{L}_{\leq T}X_\bullet$ are weakly homotopy equivalent as $S^1$-spaces (refer to \cite[Proposition A.1]{Segal1974} for the simplicial version; it is sufficient for us to verify the equivariance of the maps). Therefore, we conclude that
\begin{Thm}\label{Viterbo isomorphism statement}For a compact manifold $X$, $T\in [0,\infty]$, we have 
\begin{equation*}
    H^{q}C_{T}^{S^1}(D^*X,\bK) \cong H_{d-q}^{S^1}(\mathcal{L}_{\leq T}X,\bK).
\end{equation*}
\end{Thm}
The parallel result for the $\bZ/\ell$-equivariant Chiu-Tamarkin invariant is also true using similar arguments, by considering $\mathcal{L}_{\ell, T}^N X$ directly.

\subsection{Product structure}\label{Chas Sullivan product}
In this subsection, we will compute the product that we constructed in \autoref{sec: cup product}. Precisely, we will show that it is equivalent to the loop product, also known as the Chas-Sullivan product, on the homology of the free loop space. For the definition of the loop product, we use the Thom collapse approach given in \cite{Cohen_2002}. For further discussions on string topology, we refer to \cite{cohen2006stringtopology}.

Let us review the notation of \eqref{discretegeodesic space}. In particular, let us denote
\begin{equation}
 \mathcal{L}^N_{T}X=\mathcal{L}^N_{1,T}X=\lbrace(\bq_i)_{i}:  d(\bq_i,\bq_{i+1})\leq \min\{1,T/N\},\, i\in \bZ/N \rbrace.
\end{equation}
Then, when $X$ is orientable and $0\leq T \leq N$ for $N\in \bN$, we have (similar to \eqref{non-equivariant CT and yoneda form})
\begin{equation}\label{loop adjoint identity}
 H^{q}C_{T}(D^*X,\bK)\cong \textnormal{Ext}^{q}(\bK_{\mathcal{L}^N_{T}X},\bK_{X^N}[(N-1)d]).   
\end{equation}
Now, take $\alpha \in H^{a}C_{A}(D^*X,\bK)$ and $\beta \in H^{b}C_{B}(D^*X,\bK)$. Assume $0\leq A,B\leq N $, then the identification \eqref{loop adjoint identity} presents $\alpha$ and $\beta$ as follow morphisms:
\begin{align*}
    \alpha: \bK_{\mathcal{L}^N_{A}X}\rightarrow \bK_{X^N}[a+(N-1)d],\qquad
    \beta: \bK_{\mathcal{L}^N_{B}X}\rightarrow \bK_{X^N}[b+(N-1)d],
\end{align*}
and $\alpha\dboxtimes\beta$ corresponds to
\[\alpha\dboxtimes\beta:\bK_{\mathcal{L}^{N}_{A}X\times \mathcal{L}^N_{B}X}\rightarrow \bK_{X^{2N}}[a+b+2[N-1]d].\]
Apply the collapsing map, we have 
\[\bK_{\mathcal{L}^{N}_{A}X\times \mathcal{L}^N_{B}X}\rightarrow \bK_{X^{2N}}[a+b+2(N-1)d]\xrightarrow{i} \bK_{X^{2N}\cap\{\bq_N=\bq_{2N}\}}[a+b+2(N-1)d] .\]
It is a class $i\circ \alpha\dboxtimes\beta$ in 
\begin{align*}
     &\textnormal{Ext}^{a+b+2(N-1)d}(\bK_{\mathcal{L}^{N}_{A}X\times \mathcal{L}^N_{B}X},\bK_{X^{2N}\cap\{\bq_N=\bq_{2N}\}}) \\
\cong & \textnormal{Ext}^{a+b+2(N-1)d}(\bK_{\mathcal{L}^{N}_{A}X\times \mathcal{L}^N_{B}X\cap\{\bq_N=\bq_{2N}\}},\bK_{X^{2N}\cap\{\bq_N=\bq_{2N}\}}).
\end{align*}
Notice that, if we apply the piecewise geodesic map, we will see that $\mathcal{L}^{N}_{A}X\times \mathcal{L}^N_{B}X\cap\{\bq_N=\bq_{2N}\}$ is homotopy equivalent to the space of free composable loops:
\[\mathcal{F}_{A+B}X=\{(c_1,c_2)\in \mathcal{L}_{\leq A}X\times \mathcal{L}_{\leq B}X: c_1(0)=c_2(0)  \}.\]
Next, we apply the Gysin map associated with the pair $(X^{2N},X^{2N}\cap\{\bq_N=\bq_{2N}\})$, yielding a class $e\circ i\circ (\alpha\dboxtimes \beta)$:
\[\bK_{\mathcal{L}^{N}_{A}X\times \mathcal{L}^N_{B}X \cap\{\bq_N=\bq_{2N}\}}\rightarrow \bK_{X^{2N}\cap\{\bq_N=\bq_{2N}\}}[a+b+2(N-1)d]\xrightarrow{e} \bK_{X^{2N}}[a+b+(2N-1)d]\]
in
\[\textnormal{Ext}^{a+b+(2N-1)d}(\bK_{\mathcal{L}^{N}_{A}X\times \mathcal{L}^N_{B}X\cap\{\bq_N=\bq_{2N}\}},\bK_{X^{2N}}).\]
Under the piecewise geodesic construction, it corresponds to applying the Gysin map associated with the pair $(\mathcal{L}_{A+B}X,\mathcal{F}_{A+B}X)$.

On the other hand, we have the morphism induced by summation and restriction:
\begin{align*}
\textnormal{Ext}^{a+b+(2N-1)d}(\bK_{\mathcal{L}^{N}_{A}X\times \mathcal{L}^N_{B}X \cap\{\bq_N=\bq_{2N}\}},\bK_{X^{2N}})\xrightarrow {\tnR s_{t!}^2 ( j\circ -)} \textnormal{Ext}^{a+b+(2N-1)d}(\bK_{\mathcal{L}^{2N}_{A+B}X },\bK_{X^{2N}})  ,
\end{align*}
which is the tautological map that turns a figure $8$ type curve in $X$ into a closed curve by forgetting the crossing point of $8$. Under the morphism, $ e\circ i\circ (\alpha\dboxtimes \beta)$ is mapped to
\[\tnR s_{t!}^2(j\circ e\circ i\circ( \alpha\dboxtimes \beta)) \in \textnormal{Ext}^{a+b+(2N-1)d}(\bK_{\mathcal{L}^{2N}_{A+B}X },\bK_{X^{2N}})  \cong \textnormal{Ext}^{a+b}(F_1(D^*X,\bK)_{A+B},\bK[-d]).\]
Consequently, we track all steps and apply the piecewise geodesic construction in all steps, we see that $\tnR s_{t!}^2(j\circ e\circ i\circ( \alpha\dboxtimes \beta))$ represents the loop product of $\alpha,\beta$ in $H_{a+b+d}(\mathcal{L}_{\leq A+B}X,\bK)$.

Let us remark here that to compare with string topology, we start with the identification \eqref{loop adjoint identity}. However, in \autoref{sec: cup product}, we identify $H^{q}C_{T}(D^*X,\bK)$ with $\textnormal{Ext}^{q}(P_{D^*X},\bK_{\Delta_{X^2}})$. Let us denote the corresponding classes of $\alpha, \beta \in \textnormal{Ext}^{q}(\bK_{\mathcal{L}^N_{T}X},\bK_{X^N}[(N-1)d])$ by $\overline{\alpha}, \overline{\beta}$. Then, by carefully comparing adjunction isomorphisms and using the naturality of the Euler class, we can see that $j\circ e\circ i\circ \alpha\dboxtimes\beta= e\circ \Tilde{\Delta}_X^{-1}(\overline{\alpha}\dboxtimes \overline{\beta})$. Consequently, we have $\tnR s_{t!}^2(j\circ e\circ i\circ( \alpha\dboxtimes \beta))=\overline{\alpha}\cup \overline{\beta}$ up to natural identifications, and
\begin{Thm}\label{theorem: viterbo isomorphism-product}For a compact orientable manifold $X$, the Viterbo isomorphism 
\[H^{q}C_{T}(D^*X,\bK)\cong H_{d-q}(\mathcal{L}_{\leq T}X,\bK)\]
is an isomorphism of $\bK$-algebras that intertwines the cup product on the Chiu-Tamarkin cohomology and the filtered version of the loop product on the string topology.
\end{Thm}
\appendix
\section{Admissibility}\label{appendix: existence kernel}
The main objective of this appendix is to establish the admissibility of all open sets (i.e., \autoref{all open sets admissible}) within the framework of triangulated categories. 
\begin{Thm}All open sets in a cotangent bundle are admissible.    
\end{Thm}
\begin{RMK}As we have mentioned in \autoref{remark: kernel in infinity category}, the result itself is not new. However, the proof and the techniques employed here, especially those related to triangulated derivators, are not widely known. Therefore, the proof presented here is new and provides a fresh perspective. We are delighted to present this application of Grothendieck derivators.
\end{RMK}
First, let us recall the following microsupport estimate: If $U$ is a bounded open set and $F\in \cD(X)$, then we have
\begin{equation}\label{eq: microsupport estimate for kernel} SS(F \star Q_U)\subset cone(Z ) \cup 0_{X\times \bR}.   
\end{equation}
This estimate can be directly found in the proof of \cite[Corollary 2.9]{Capacities2021}.

Now, for an arbitrary open the $U\subset T^*X$, we take a bounded exhaustion $(U_n)_{n\in \bN}$ of $U$, i.e. a sequence of bounded open set such that $\overline{U_n}\subset U_{n+1}$ and $U=\bigcup_n U_n$. We also define $Z_n=T^*X\setminus U_n$. It follows that $Z_n$ is a decreasing sequence of closed sets that converges to $Z$. Since bounded open sets $U_n$ are admissible, we have a sequence of defining triangles
\[P_n\xrightarrow{a_n} \bK_{ \Delta_{X^2} \times {[0,\infty)} } \xrightarrow{b_n} Q_n \xrightarrow{+1}.\]
Then, according to \autoref{functorial}, we can construct an inductive system of distinguished triangles in $D(X^2\times \bR)$, as $\cD(X^2)$ is a full triangulated subcategory of $D(X^2\times \bR)$:
\begin{equation*}
   \begin{tikzcd}
P_n \arrow[r, "a_n"] \arrow[d, "p_n"] & {\bK_{ \Delta_{X^2} \times {[0,\infty)} } } \arrow[r, "b_n"] \arrow[d, "\id"] & Q_n \arrow[r, "+1"] \arrow[d, "q_n"] & {} \\
P_{n+1} \arrow[r, "a_{n+1}"]          & {\bK_{ \Delta_{X^2} \times {[0,\infty)} } } \arrow[r, "b_{n+1}"]              & Q_{n+1} \arrow[r, "+1"]              & . 
\end{tikzcd}  
\end{equation*}
Here, we expect that the homotopy colimit of the inductive system still forms a distinguished triangle. However, this property does not generally hold for triangulated categories. At this point, we only have two morphisms available:
\[ P_\infty \xrightarrow{a_\infty} {\bK_{ \Delta_{X^2} \times {[0,\infty)} } }\xrightarrow{b_\infty} Q_\infty,\]
where $\textnormal{hocolim}(P_n)=P_\infty$, $\textnormal{hocolim}(Q_n)=Q_\infty$, $\textnormal{hocolim}(a_n)=a_\infty$ and $\textnormal{hocolim}(b_n)=b_\infty$. Notice that $P_\infty$ and $Q_\infty$ are unique upto (not necessarily unique) isomorphisms and $a_\infty$ and $b_\infty$ are not unique.

Nevertheless, it has been shown in \cite[Example 5.17]{MV_in_stable_derivators} that for the stable model category $C(X\times\bR)=C(Sh(X\times\bR))$ (since $Sh(X\times\bR)$ is Grothendieck), there exists a triangulated derivator $\mathbb{D}{Sh(X\times\bR)}$ such that $\mathbb{D}_{Sh(X\times\bR)}(e)=D({Sh(X\times\bR)})=D(X\times\bR)$, where $e$ represents the 1-point category. Therefore, we can utilize \cite[Corollary 11.4]{functorial_cone_Keller_Pedro} twice to obtain two distinguished triangles that relate to homotopy colimits. It should be noted that our homotopy colimit is referred to as the Milnor colimit in {\it loc. cit.} Precisely, we have
\begin{Lemma}There exists two inductive systems $(P_n',\,P_n'\xrightarrow{p_n'} P'_{n+1})$ and $(Q_n',\,Q_n'\xrightarrow{q_n'} Q'_{n+1})$ such that $P_n\cong P_n'$, $Q_n\cong Q_n'$ and we have following distinguished triangles:
\begin{equation}\label{distinguished triangle in general}
\begin{split}
P_\infty \xrightarrow{a_\infty} \bK_{ \Delta_{X^2} \times {[0,\infty)} } \rightarrow  \textnormal{hocolim}(Q_n') \xrightarrow{+1},\\
    \textnormal{hocolim}(P'_n)\rightarrow \bK_{ \Delta_{X^2} \times {[0,\infty)} } \xrightarrow{b_\infty} Q_\infty \xrightarrow{+1}.
    \end{split}
\end{equation}
\end{Lemma}
We set 
\begin{equation*}
a'_\infty: P'_\infty\coloneqq \textnormal{hocolim}(P'_n)\rightarrow \bK_{ \Delta_{X^2} \times {[0,\infty)} },\quad 
b'_\infty:\bK_{ \Delta_{X^2} \times {[0,\infty)} } \rightarrow  \textnormal{hocolim}(Q_n')\eqqcolon Q'_\infty.    
\end{equation*}\begin{Prop}The open set $U$ is admissible according to both of the distinguished triangles in \eqref{distinguished triangle in general}.
\end{Prop}
\begin{proof}We will prove the statement for the first distinguished triangle. The proof for the second one is identical. 

We only need to verify
\[\mu s(F \star Q'_\infty)\subset Z,\text{ and }\RHOM(F\star P_\infty,G)=0, \qquad \forall G \in \cD_Z(M).\]
The main observation is that for all $F\in D(X\times \bR)$, we also have an inductive system $(F\star Q_n',\,F\star Q_n'\rightarrow F\star Q'_{n+1})$. So, we have a well-defined homotopy colimit $\text{hocolim}(F \star Q'_n)$. Moreover, $\star$ is cocontinuous,  then 
\[\text{hocolim}(F \star Q'_n) \cong F\star (\text{hocolim}Q'_n)=F\star Q'_\infty.\]
Similarly, we have that \[\text{hocolim}(F \star P_n) \cong F\star (\text{hocolim}P_n)=F\star P_\infty.\]
For $F\star Q'_\infty$, we have the microsupport estimate (see \cite[Proposition 7.3]{guillermouviterbo_gammasupport} )
\[SS(F\star Q'_\infty)\subset SS(\text{hocolim}(F \star Q'_n))\subset \bigcap_{N\geq 1}\overline{\bigcup_{n\geq N}SS(F \star Q'_n)}.\]
However, since $Q_n\cong Q_n'$, we know from \eqref{eq: microsupport estimate for kernel} that
\[SS(F \star Q'_n)=SS(F \star Q_n)\subset cone(Z_n) \cup 0_{X\times \bR}.\]
As $\{Z_n\}$ is a decreasing sequence, we have 
\[\bigcup_{n\geq N} SS(F \star Q'_n) \subset   cone (Z_N)\cup 0_{X\times \bR}.\]
But one can check the right-hand side is already closed.
Then we have 
\begin{equation}\label{equation: unbounded A1}
  SS(F\star Q'_\infty)\subset \bigcap_{N\geq 1} \left(cone (Z_N)\cup 0_{X\times \bR} \right) \subset cone (Z)\cup 0_{X\times \bR}.  
\end{equation}
In particular, we have $\mu s(F\star Q'_\infty) \subset Z$. 

Next, for all $G\in \cD(M)$, we have
\[\RHOM(F\star P_\infty,G)\cong\RHOM(\text{hocolim}(F \star P_n),G)\cong\text{holim}\,\RHOM(F \star P_n,G).\]
Now, if $\mu s(G) \subset Z$, we have $\mu s(G) \subset Z\subset Z_n$ for all $n$. Then we have $\RHOM(F \star P_n,G)=0$ for all $n$ and all $G\in \cD_Z(M)$. So we have, for all $G\in \cD_Z(M)$,
\[\RHOM(F\star P_\infty,G)=\text{holim}\,\RHOM(F \star P_n,G)=0.\qedhere\]
\end{proof}
Moreover, the uniqueness part of \autoref{functorial} shows that we can choose a canonical isomorphism of distinguished triangles in \eqref{distinguished triangle in general}:
\begin{equation*}
 \begin{tikzcd}
P_\infty \arrow[d, "p", "\cong"'] \arrow[r, "a_\infty"] & {{\bK}_{\Delta_{X^2}\times[0,\infty)}} \arrow[r, "b'_\infty"] \arrow[d, "\id"] & Q'_\infty \arrow[d, "q", "\cong"'] \arrow[r, "+1"] & {} \\
P'_\infty \arrow[r, "a'_\infty"]                & {{\bK}_{\Delta_{X^2}\times[0,\infty)}} \arrow[r, "b_\infty"]                  & Q_\infty \arrow[r, "+1"]                &  { .}
\end{tikzcd}   
\end{equation*}
Then we can conclude that 
\[ P_\infty \xrightarrow{a_\infty} {\bK_{ \Delta_{X^2} \times {[0,\infty)} } }\xrightarrow{b_\infty} Q_\infty\xrightarrow{+1}\]
form a distinguished triangle, and $U$ is admissible by this distinguished triangle.

Upon establishing the existence of microlocal kernels in general, the colimit construction still works even if we do not assume that the exhaustion is given by bounded open sets (noting that \eqref{equation: unbounded A1} provides the unbounded version of \eqref{eq: microsupport estimate for kernel}).
\begin{Coro}\label{coro: colimit formula for kernel in general}For an exhaustion of open sets $(U_n)_{n\in \bN}$ of $U$, i.e. $\overline{U_n}\subset U_{n+1}$ and $U=\bigcup_n U_n$, set 
\[P_n\xrightarrow{a_n} \bK_{ \Delta_{X^2} \times {[0,\infty)} } \xrightarrow{b_n} Q_n \xrightarrow{+1}\]   
to be defining triangles for all $U_n$, we have a defining triangle 
\[ P_\infty \xrightarrow{a_\infty} {\bK_{ \Delta_{X^2} \times {[0,\infty)} } }\xrightarrow{b_\infty} Q_\infty\xrightarrow{+1}\]
for $U$ where $\textnormal{hocolim}(P_n)=P_\infty$, $\textnormal{hocolim}(Q_n)=Q_\infty$.
\end{Coro}

\bibliographystyle{bingyu}
\clearpage
\phantomsection
\bibliography{bibtex}

\noindent
\parbox[t]{28em}
{\scriptsize{
\noindent
Bingyu Zhang\\
Centre for Quantum Mathematics, University of Southern Denmark\\
Campusvej 55, 5230 Odense, Denmark\\
Email: {bingyuzhang@imada.sdu.dk}
}}

\end{document}